\documentclass[twoside,11pt]{article}

\usepackage{bm, mathrsfs, graphics,float,amssymb,amsmath,subeqnarray,setspace,graphicx, enumerate, color}
\usepackage{amsfonts,amstext,mathrsfs}
\usepackage{microtype}
\usepackage{subcaption}
\usepackage{booktabs} 
\usepackage{dsfont}
\usepackage{enumitem}
\usepackage{xparse}
\usepackage{algorithm}
\usepackage{bbm}
\usepackage{algorithmic}
\NewDocumentCommand{\ceil}{s O{} m}{%
	\IfBooleanTF{#1} 
	{\left\lceil#3\right\rceil} 
	{#2\lceil#3#2\rceil} 
}
\newcommand{\be}{\begin{equation}}
\newcommand{\ee}{\end{equation}}
\newcommand{\tf}{\widehat{f}}
\newcommand{\hf}{\widetilde{f}}
\newcommand{\bcG}{\bm{\mathcal{H}}}
\newcommand{\C}{\mathbf{D}}
\newcommand{\D}{\mathbf{C}}
\newcommand{\UU}{\mathbf{V}}

\newcommand{\cA}{\mathcal{A}}
\renewcommand{\u}{\mathbf{u}}
\newcommand{\y}{\mathbf{y}}
\newcommand{\x}{\mathbf{x}}

\newcommand{\B}{\mathbf{B}}
\newcommand{\I}{\mathbf{I}}
\newcommand{\E}{\mathbb{E}}
\newcommand{\Z}{\mathbf{Z}}
\newcommand{\bT}{\mathrm{T}}
\newcommand{\bcV}{\bm{\mathcal{ V}}}
\newcommand{\bcU}{\bm{\mathcal{U}}}

\newcommand{\bxi}{\bm{\xi}}

\newcommand{\uu}{\mathbf{u}}
\newcommand{\RR}{\mathbb{R}}
\newcommand{\PP}{\mathbb{P}}
\newcommand{\cP}{\bm{\mathcal{P}}}
\newcommand{\cO}{\mathcal{O}}
\newcommand{\cS}{{\mathcal{S}}}

\newcommand{\inner}[2]{\left\langle #1, #2 \right\rangle}

\newcommand{\ie}{\emph{i.e.}}
\newcommand{\eg}{\emph{e.g.}}
\usepackage{hyperref}
\usepackage{jmlr2e}
\usepackage{subcaption}
\usepackage{url}
\usepackage{epsfig,epstopdf,algorithmic,algorithm,amsfonts,url,textcomp,multirow,moresize}
\usepackage{graphicx}
\usepackage{booktabs}
\usepackage{algorithm}
\usepackage{algorithmic}
\usepackage{amsfonts}
\usepackage{subcaption}
\usepackage{amsmath}\allowdisplaybreaks
\usepackage[framemethod=tikz]{mdframed}
\usepackage[flushleft]{threeparttable}
\usepackage{enumitem}
\usepackage{etoolbox}

\usepackage{geometry}
\begin{document}

\title{Faster Stochastic Quasi-Newton Methods}

\author{\name Qingsong Zhang \email qszhang1995@gmail.com      \\
\addr Xidian University, Xi’an, China, and also with JD Tech.
\\
\name Feihu Huang \email {huangfeihu2018@gmail.com, feh23@pitt.edu} \\
\addr Department of Electrical and Computer Engineering, University of Pittsburgh, USA \\
   \name Cheng Deng	\email {chdeng@mail.xidian.edu.cn} \\
 \addr School of Electronic Engineering, Xidian University, Xi'an, China
 \\
   \name Heng Huang \email {heng.huang@pitt.edu} \\
\addr JD Finance America Corporation \\ University of Pittsburgh, USA\\
}

\editor{}
\maketitle

\begin{abstract}
Stochastic optimization methods have become a class of popular optimization tools in machine learning. Especially, stochastic gradient descent (SGD) has been widely used for machine learning problems such as training neural networks due to low per-iteration computational complexity. In fact, the Newton or quasi-newton methods leveraging second-order information are able to achieve better solution than the first-order methods. Thus, stochastic quasi-Newton (SQN) methods have been developed to achieve better solution efficiently than the stochastic first-order methods by utilizing approximate second-order information. However, the existing SQN methods still do not reach the best known stochastic first-order oracle (SFO) complexity. To fill this gap, we propose a novel faster stochastic quasi-Newton method (SpiderSQN) based on the variance reduced technique of SIPDER. We prove that our SpiderSQN method {reaches} the best known SFO complexity of  $\mathcal{O}(n+n^{1/2}\epsilon^{-2})$ in the finite-sum setting to obtain an $\epsilon$-first-order stationary point. To further improve its practical performance, we incorporate SpiderSQN with different momentum schemes. Moreover, the proposed algorithms are generalized to the online setting, and the corresponding SFO complexity of $\mathcal{O}(\epsilon^{-3})$ is developed, which also matches the existing best result. Extensive experiments on benchmark datasets demonstrate that our new algorithms outperform state-of-the-art approaches for nonconvex optimization.
\end{abstract}


%

\section{Introduction}
 In  this paper, we focus on the following unconstrained stochastic nonconvex optimization:
\begin{align}
\min_{x\in \mathbb{R}^d} f(x) :=\left\{
\begin{aligned}
 {\mathbb{E}_{{u}\sim \mathbb{P}}[f_u(x)]}  & \qquad \mbox{(online)}\\
 \frac{1}{n}\sum_{i=1}^n f_i(x) & \qquad \mbox{(finite-sum)}
 \end{aligned} \right., \label{eq: P}\tag{P}
\end{align}
where $x\in \mathbb{R}^d $ corresponds to the parameters defining a model, ${\mathbb{E}_{u\sim \mathbb{P}}[f_u(x)] }$ denotes a population risk over $u \sim \mathbb{P}$, and {$f_i: \mathbb{R}^d \to \mathbb{R}$ denotes the loss on the $i$-th sample for $\forall\ i \in {1, ..., n}$ (or $i\sim\mathbb{P}$)}. Problem  (\ref{eq: P}) capsules a widely range of machine learning problems such as truncated square loss \cite{xu2018learning} for regression and deep neural network \cite{goodfellow2016deep}. In fact,
the SGD \cite{ghadimi2016mini} is {a representative} method to solve the problem (\ref{eq: P}) due to its per-iteration computation efficiency. Recently, there have been many works studying SGD and its variance reduction variants, including SVRG \cite{reddi2016stochastic}, SAGA \cite{reddi2016fast}, SCSG\cite{lei2017non}, SARAH \cite{nguyen2017stochastic}, {SNVRG} \cite{zhou2018stochastic} and SPIDER \cite{fang2018spider,NIPS2019_8511}. In particular, SPIDER has been shown in \cite{fang2018spider} to achieve the SFO complexity lower bound for a certain regime. Such idea has been {extended } to optimization over mainfolds in \cite{zhou2019faster}, zeroth-order optimization  in \cite{HuangNonconvex,ji2019improved}, cubic-regularized method in \cite{zhou2020stochastic}, and alternating direction method of multipliers in \cite{huang2019faster}.

Although SGD is very effective, its performance maybe poor owing that it only utilizes the first-order information. In contrast, Newton's method utilizing the Hessian information is more robust and can achieve better accuracy \cite{sohl2014fast,allen2018natasha}, while it is extremely time consuming to compute Hessian matrix and its inverse. Therefore, many works have been {proposed} toward designing better SGD methods integrated with approximate Hessian information, \ie, the SQN methods.  There have been many works focusing on developing SQN methods such as SGD with quasi-Newton (SGD-QN) studied in \cite{bordes2009sgd} and stochastic approximation based L-BFGS proposed in \cite{byrd2016stochastic}. Recently, some SQN methods equipped with the variance reduction technique have been developed to alleviate the effect of variance introduced by stochastic estimator \cite{kolte2015accelerating,lucchi2015variance,moritz2016linearly,gower2016stochastic}. Besides above methods concerning convex or strongly convex problems, progresses have been made toward designing SQN methods for nonconvex cases. Wang \emph{et al.} \cite{wang2017stochastic} analyzed the convergence guarantee of the SGD-QN for nonconvex problems, Wang \emph{et al.} \cite{wang2018stochastic} developed a stochastic proximal quasi-Newton for nonconvex composite optimization, and Gao \emph{et al.}  \cite{gao2018stochastic} proposed the stochastic L-BFGS method for nonconvex sparse learning problems.

Stochastic quasi-Newton methods inherit many appealing advantages from both SGD  and quasi-Newton methods, \eg, efficiency, robustness and better accuracy. However, existing SQN methods still do not reach the best known SFO complexity, resulting the limited application to machine learning. It is thus of vital importance to improve the SFO complexity of SQN methods for nonconvex optimization. For this reason, we propose a faster SQN method (namely SpiderSQN) by leveraging the variance reduction technique of SIPDER.

Albeit SpiderSQN achieves the optimal SFO complexity for nonconvex optimization, its practical performance may not exhibit such optimality. Thus, we consider utilizing momentum acceleration technology to obtain better practical performance. Moreover, to deal with cases where the number of training samples is extremely large or even infinite, the SpiderSQN based algorithms are extended to the online case with theoretical guarantee. To give a thorough comparison of our proposed algorithm with existing stochastic first-order algorithms
and SQN for nonconvex optimization, we summarize the SFO complexity of the most relevant algorithms to achieve an $\epsilon$-first-order stationary point in Table \ref{tab:comparsion}.  The main contributions of this paper are summarized as follows.
\begin{enumerate}[]
    \item
     We propose a novel faster stochastic quasi-Newton method (SpiderSQN) for nonconvex optimization in the form of finite-sum. Moreover, we prove that the SpiderSQN can achieve the best known optimal SFO complexity of $\mathcal{O}(n+n^{1/2}\epsilon^{-2})$ to obtain an $\epsilon$-first-order stationary point.
    \item
    We extend the SpiderSQN to the online setting, and propose the faster online SpiderSQN algorithms for nonconvex optimization. Moreover, we prove that the online SpiderSQN achieve the best known optimal SFO complexity of $\mathcal{O}(\epsilon^{-3})$.
     \item
    To improve the practical performance of the proposed methods, we apply momentum schemes to them, which are demonstrated to have satisfactory practical effects.
    \item
    Moreover, we prove that our SpiderSQN methods have the lower SFO complexity of $\mathcal{O}(n^{1/2}\epsilon^{-1/2})$, which achieves the optimal SFO complexity of $\mathcal{O}(n^{1/2}\epsilon^{-1/2})$.
\end{enumerate}
\begin{table}[!t]
\centering
\caption{Comparison of results on SFO complexity for smooth nonconvex optimization. Note that we omit the poly-logarithmic
factors of $d, n, \epsilon$. Especially, SpiderSQN-M represents SpiderSQN with different momentum schemes.}
\label{tab:comparsion}
\small{\begin{tabular}{lcc}
\toprule
\multicolumn{1}{l}{Algorithm}        & \multicolumn{1}{l}{Finite-sum}               & Online                               \\
\midrule
SGD \cite{ghadimi2016mini}
& $\cO({n\epsilon^{-2}})$                             & $\cO({\epsilon^{-4}})$                      \\
SVRG\cite{reddi2016stochastic}
& {$\cO({n+n^{\frac{2}{3}}\epsilon^{-2}})$}                   & $\cO({\epsilon^{-\frac{10}{3}}})$           \\
SARAH \cite{nguyen2017stochastic}
& $\cO({n+\epsilon^{-4}})$                                        & $\cO({\epsilon^{-4}})$                      \\
SNVRG  \cite{zhou2018stochastic}
& $\cO({n+n^{\frac{1}{2}}\epsilon^{-2}})$                             & $\cO({\epsilon^{-3}})$  \\
SPIDER \cite{fang2018spider,NIPS2019_8511}
& $\cO({n+n^{\frac{1}{2}}\epsilon^{-2}})$                             & $\cO({\epsilon^{-3}})$  \\
SQN with SGD   \cite{wang2017stochastic}
& $\cO({n\epsilon^{-2}})$                             & $N/A$ \\
SQN with SVRG   \cite{wang2017stochastic}
& $\cO({n+n^{\frac{2}{3}}\epsilon^{-2}})$                              & $N/A$ \\
{SpiderSQN ({\bf Ours})}
& $\cO({n+n^{\frac{1}{2}}\epsilon^{-2}})$                             & $\cO({\epsilon^{-3}})$  \\
{SpiderSQN-M ({\bf{Ours}})}
& $\cO({n+n^{\frac{1}{2}}\epsilon^{-2}})$                             & $\cO({\epsilon^{-3}})$  \\
\bottomrule
\end{tabular}}
\end{table}


\section{Preliminaries}\label{sec: pre}
In this section, some preliminaries are presented. Since finding the global minimum of problem (\ref{eq: P}) is general NP-hard \cite{hillar2013most}, this work instead focuses on finding an $\epsilon$-first-order stationary point and studies the SFO complexity of achieving it. First, we give the necessary definitions and assumptions.
\begin{definition}\label{defin1}
{An $\epsilon$-first-order stationary point denotes that for $x$ uniformly drawn from $x_1, \cdots, x_K$, where $K$ is the total number of iterations there is $\mathbb{E}\|\nabla f(x)\| \le \epsilon$, where $\epsilon > 0$ is the accuracy parameter.}
\end{definition}
\begin{definition}\label{definsfo}
Given a sample $i$ ($i\in 1,\cdots, n$ or $i\sim \mathbb{P}$) and a point $x\in \mathbb{R}^d$, a stochastic/incremental first-order oracle (SFO/IFO) \cite{reddi2016stochastic} returns the pair $(f_i(x), \nabla f_i(x))$.
\end{definition}
\begin{assumption}\label{assum1}
	Function $f$ is bounded below, i.e.,
		\begin{align}
		f^*:=\inf_{x\in \mathbb{R}^d} f(x) > -\infty.
		\end{align}
\end{assumption}
\begin{assumption}\label{assum2}
		Individual  function $f_i, i=1,\ldots,n$ or $i\sim\mathbb{P}$ is $L$-smooth, \ie, there exists an $L>0$ such that
		\begin{align}
		\|\nabla f_i (x) - \nabla f_i (y)\| \le L\|x-y\|, \quad \forall  x,y\in \mathbb{R}^d,
		\end{align}
\end{assumption}
Above two assumptions are standard in the analysis of nonconvex optimization \cite{ghadimi2016accelerated,huang2019faster,HuangNonconvex}, where Assumption \ref{assum1} guarantees the feasibility of problem (\ref{eq: P}) and Assumption \ref{assum2} imposes smoothness on the individual loss functions.

\begin{assumption}\label{assum6}
 For $\forall i \in 1,\cdots, n$ (or $i\sim \mathbb{P}$), function $f_i(x)$ is twice continuously differentiable with respect to $x$.
 There exists a positive constant $\kappa$ such that $\|\nabla^2 f_i(x)\|\le \kappa$ for $\forall$ $x$.
\end{assumption}
Note that Assumption \ref{assum6} is  standard  for SQN methods focusing on nonconvex problem \cite{wang2017stochastic}.
\begin{assumption}\label{assum3}
	There exist two positive constants $\sigma_{\mathrm{min}}$, and $\sigma_{\mathrm{max}}$ such that
\begin{align}
		\sigma_{\mathrm{min}}I \preceq H_k \preceq \sigma_{\mathrm{max}}I,
		\end{align}
where {$H_k$ is the inverse Hessian approximation matrix} and notation $A \preceq B $ with $A, B\in \mathbb{R}^{d\times d}$ means that $A-B$ is positive semidefinite.
\end{assumption}

\begin{assumption}\label{assum4}
	For any $k\ge 2$, the random variable $H_k$ ($k\ge 2$) depends only on $v_{k-1}$ and $\xi_{k}$
\begin{align}
		\mathbb{E}[H_kv_k|\xi_{k}, v_{k-1}] = H_kv_k,
		\end{align}
where the expectation is taken with respect to $|\xi_{k}|$ samples generated for calculation of $\nabla f_{\xi_k}$.
\end{assumption}
Assumptions \ref{assum3} and \ref{assum4} are commonly used for SQN methods \cite{wang2017stochastic,moritz2016linearly}, where Assumption \ref{assum3} shows that the matrix norm of $H_k$ is bounded and Assumption \ref{assum4} means although $H_k$ is generated iteratively based on historical gradient information by a random process, given $v_{k-1}$ and $\xi_{k}$ the $H_kv_k$ is determined.

\subsection{SGD Methods for Nonconvex Optimization}
Stochastic first-order optimization methods have been widely used for solving machine learning tasks. As for nonconvex optimization, a classical algorithm is the SGD \cite{ghadimi2016mini} which has an overall SFO complexity of $\mathcal{O}(\epsilon^{-4})$ to achieve an $\epsilon$-first-order stationary point. Also, a variety of SGD variants equipped with variance reduction have been proposed such as the SVRG, SAGA, and its application to federated learning \cite{zhang2020secure}. Moreover, the corresponding SFO complexity of obtaining an $\epsilon$-first-order stationary point is $\mathcal{O}(n^{2/3}\epsilon^{-2})$ \cite{reddi2016stochastic,reddi2016fast}. Recently, some algorithms with a new type of stochastic variance reduction technique have been exploited, including {SNVRG}, SARAH and SPIDER \cite{nguyen2017sarah,zhou2018stochastic,fang2018spider}, which uses more fresh gradient information to evaluate the gradient estimator. Therefore, take the {SNVRG} as an example, it has an improved SFO complexity of min$\{\mathcal{O}(n^{1/2}\epsilon^{-2}),\mathcal{O}(\epsilon^{-3})\}$ to achieve an $\epsilon$-first-order stationary point.

\subsection{SQN Methods For Nonconvex Optimization}
Newton's methods using Hessian information have rapid convergence rate (both in theory and practice) \cite{moritz2016linearly} and are popular for solving nonconvex problems \cite{kohler2017sub,zhou2018stochastic1,zhou2019stochastic}. However, time consumption of computing Hessian matrix and its inverse is extremely high. To address this problem, many quasi-Newton (QN)-based methods have been widely studied such as BFGS, L-BFGS, and the damped L-BFGS \cite{nocedal2006numerical}.
In this paper, we adopt the stochastic damped L-BFGS (SdLBFGS) \cite{wang2017stochastic} for nonconvex optimization. Let $k$ be current iteration, based on history information, SdLBFGS uses a two-loop recursion to generate a descent direction $d_k = H_kv_k$ without calculating inverse matrix $H_k$ explicitly.
\begin{algorithm}[!t]
\caption{{ Core step of stochastic damped L-BFGS} \cite{wang2017stochastic}}\label{SdLBFGS}
\begin{algorithmic}[1]
\REQUIRE {Let $k$ be current iteration. Given the stochastic gradient $v_{k-1}$ at iteration $k-1$, the samples batch $\xi_k$ at iteration $k$ and vector pairs $\{s_j, \bar{y}_j, \rho_{j}\}$ $j=k-m,\ldots,k-2$, {where $m$ is the memory size}, and $u_0=v_k$}
\STATE Calculate $s_{k-1} $, $\bar{y}_{k-1}$ and $\gamma_k$
\STATE  Calculate $\hat{y}_{k-1}$ through Eq.~\ref{2.2} and $\rho_{k-1}=(s_{k-1}^\top {\hat{y}_{k-1}})^{-1}$
  \FOR {$i=0,\ldots,\min\{m,k-1\}-1$}
  \STATE  Calculate $\mu_i=\rho_{k-i-1}u_i^\top s_{k-i-1}$
  \STATE  Calculate $u_{i+1} = u_i - \mu_i {\hat{y}_{k-i-1}}$
  \ENDFOR
  \STATE  Calculate $v_0=\gamma_k^{-1}u_p$
  \FOR {$i=0,\ldots,\min\{m,k-1\}-1$}
  \STATE  Calculate $\nu_i=\rho_{k-m+i}v_i^\top {\hat{y}_{k-m+i}}$
  \STATE  Calculate $
  \bar{v}_{i+1} = \bar{v}_i + (\mu_{m-i-1}-\nu_i)s_{k-m+i}$.
  \ENDFOR
  \ENSURE {$H_kv_k=\bar{v}_p$.}
\end{algorithmic}
\end{algorithm}
Specially, at step 1, vector pair {$\{s_{k-1}, \bar{y}_{k-1}\}$} is computed as $s_{k-1} = x_{k}-x_{k-1}$ and $\bar{y}_{k-1} = v_{k} - v_{k-1}$, and $\gamma_{k} = {\mathrm{max}}\{\frac{\bar{y}_{k-1}^\top \bar{y}_{k-1}}{s_{k-1}^\top \bar{y}_{k-1}}, \delta \}$, where $\delta$ is a positive constant. At setp 2, SdLBFGS introduces a vector {$\hat{y}_{k-1}$}
\be {\label{2.1}}
\hat{y}_{k-1} = \theta_{k-1} \bar{y}_{k-1} + (1-\theta_{k-1})H_{{k-1},0}^{-1}s_{k-1}, k\geq1,
\ee
where $H_{k,0} = \gamma_{k}^{-1}I_{d\times d} $, $k\geq0$, and $\theta_{k-1}$ is defined~as
\be \label{2.2}
\theta_{k-1}=\left\{\begin{array}{ll}{\frac{0.75 \sigma_{k-1}}{\sigma_{k-1}-s_{k-1}^{\top} \bar{y}_{k-1}},} & {\text { if } s_{k-1}^{\top} \bar{y}_{k-1}<0.25 \sigma_{k-1}} \\ {1,} & {\text { otherwise }}\end{array}\right.,
\ee
where $\sigma_{k-1} = s_{k-1}^{\top} H_{k, 0}^{-1} s_{k-1}$. Based on $\{s_{k-1}, \hat{y}_{k-1}\}$, $H_kv_k$ can be approximated through steps 3 to 10.

Importantly, SdLBFGS is a computation effective program because the whole procedure  takes only $(6m+6)d$ multiplications. Especially, the SdLBFGS with variance reduction is proposed \cite{wang2017stochastic} by incorporating SdLBFGS into SVRG. However,  its best SFO complexity to obtain an $\epsilon$-first-order stationary point is $\mathcal{O}(n^{2/3}\epsilon^{-2})$, which is not competitive to state-of-the-art stochastic first-order methods. Therefore, it is desirable to improve the SFO complexity of existing SQN methods.
\subsection{Momentum {Acceleratation} for Nonconvex Optimization}
Momentum acceleration scheme is a simple but widely used acceleration technique for optimization problem. Recently, a variety of accelerated methods have been developed for nonconvex optimization. For examples,
the stochastic gradient algorithms with momentum scheme is proposed in \cite{ghadimi2016accelerated}, which have been proved to converge as fast as gradient descent method for nonconvex problems. Li \emph{et al.} \cite{li2017convergence} explored the convergence of the algorithm proposed in \cite{yao2016efficient} under a certain local gradient dominance geometry for nonconvex optimization. Furthermore, Wang \emph{et al.} \cite{wang2018cubic} studied the convergence to a second-order stationary point under the momentum scheme. However, existing works hardly ever study the acceleration of the SQN method for nonconvex optimization. To this end, this paper focuses on accelerating SQN methods with different momentum schemes.
\section{Faster SQN Methods for Nonconvex Optimization}\label{sec: spider-sqn}
In this section, we propose a novel faster SQN method to solve the nonconvex problem (\ref{eq: P}) for finite-sum case.
\begin{algorithm}[!t]
\caption{SpiderSQN for Nonconvex Optimization}\label{SPIDER-SQN}
\begin{algorithmic}[1]
\REQUIRE {$ \left|\xi_{k}\right|, \eta, q, K \in \mathbb{N}$.}
\FOR {$k=0, 1, \ldots, K-1$}
  \IF {$\text{mod}(k, q)= 0$}
  \STATE Compute $v_{k} = \nabla f(x_k),$
  \ELSE
  \STATE Sample $\xi_k \overset{\text{Unif}}{\sim} \{1,\ldots,n\}$, and compute\\
   $v_k = \nabla f_{\xi_k} (x_k) - \nabla f_{\xi_k} (x_{k-1}) + v_{k-1}$.
  \ENDIF
  \STATE Compute $d_k = {H_k}v_k$ through SdLBFGS \cite{wang2017stochastic},
  \STATE $x_{k+1} =  x_{k} - \eta d_{k}$.
  \ENDFOR
 \STATE {\bfseries Output (in theory):} $x_\zeta$, where $\zeta \overset{\text{Unif}}{\sim} \{1,\ldots,K\}$.
 \STATE {\bfseries Output (in practice):}  $x_{K}$.
\end{algorithmic}
\end{algorithm}
\begin{algorithm}[!t]
\caption{SpiderSQN-M for Nonconvex Optimization}\label{SPIDER-SQN-M}
\begin{algorithmic}[1]
\REQUIRE {$ \left|\xi_{k}\right|, q, K \in \mathbb{N}$, $\{\beta_k\}_{k=0}^{K-1} >0$.}
\STATE Set $\alpha_k = \frac{2}{k+1}$ for $k=0,...,K$ and $\lambda_k \in [\beta_k, (1+\alpha_k)\beta_{k}]$ for $k=0,...,K-1$.
\STATE Initialize $y_0 = x_0\in \mathbb{R}^d$.
\FOR {$k=0, 1, \ldots, K-1$}
  \STATE $z_{k} = (1-\alpha_{k+1})y_{k} + \alpha_{k+1} x_{k}$,
  \IF {$\text{mod}(k, q)= 0$}
  \STATE Compute $v_{k} = \nabla f(z_k)$,
  \ELSE
  \STATE Sample $\xi_k \overset{\text{Unif}}{\sim} \{1,\ldots,n\}$, and compute\\
   $v_k = \nabla f_{\xi_k} (z_k) - \nabla f_{\xi_k} (z_{k-1}) + v_{k-1}$,
  \ENDIF
  \STATE Compute $d_k = {H_k}v_k$ through SdLBFGS \cite{wang2017stochastic},
  \STATE $x_{k+1} =  x_{k} - \lambda_{k} d_{k}$,
  \STATE $y_{k+1} = z_{k} - \beta_k d_{k}$.
  \ENDFOR
 \STATE {\bfseries Output (in theory):} $x_\zeta$, where $\zeta \overset{\text{Unif}}{\sim} \{1,\ldots,K\}$.
 \STATE {\bfseries Output (in practice):}  $x_{K}$.
\end{algorithmic}
\end{algorithm}
\subsection{Spider Stochastic Quasi-Newton Algorithm}
To improve the SFO complexity of SQN method, a new variance reduction technique SPIDER/SpiderBoost is adopted to control its intrinsic variance. The proposed SpiderSQN with improved SFO complexity is shown in Algorithm \ref{SPIDER-SQN}.

At each iteration, besides evaluating the full gradient every $q$ iterations, the stochastic gradient $v_k$ is updated as
\be \label{3.1}
v_k = \nabla f_{\xi_k} (x_k) - \nabla f_{\xi_k} (x_{k-1}) + v_{k-1},
\ee
where $\nabla f_{\xi_k}(x_k) =\frac{1}{\left|\xi_{k}\right|} \sum_{i \in \xi_{k}} \nabla f_{i}(x_k)$ and $\xi_k$ is a mini-batch where samples are uniformly sampled with replacement. It is obvious from Eq.~(\ref{3.1}), a more fresh stochastic gradient information $v_{k-1}$ is utilized to update $v_k$, and thus SpiderSQN has an improved SFO complexity compared with existing stochastic quasi-Newton methods. At step 8, $x_k$ is updated by the Hessian informative descent direction.

\subsection{Spider Stochastic Quasi-Newton with Momentum Scheme}
To improve the pratical performance of SpiderSQN, the momentum scheme is adopted for acceleration. The framework of SpiderSQN with momentum scheme (referred
as SpiderSQNM) is shown in Algorithm \ref{SPIDER-SQN-M}. The momentum scheme in Algorithm \ref{SPIDER-SQN-M} refers to steps 4, 11 and 12, where variables $x_k$ and $y_k$ are updated through the $d_k$, and $z_k$ is a convex combination of $x_k$ and $y_k$ controlled by the momentum coefficient $\alpha_k$. In this algorithm, an iteration-wise diminishing scheme is applied, where the momentum coefficient is set as $\alpha_k = \frac{2}{k+1}$.
\begin{algorithm}[!t]
\caption{SpiderSQN for Online Nonconvex Optimization}\label{Spider-SQN-online}
\begin{algorithmic}[1]
\REQUIRE {$ \left|\xi_{0}\right|, \left|\xi_{k}\right|,\eta, q, K \in \mathbb{N}$.}
\FOR {$k=0, 1, \ldots, K-1$}
  \IF {$\text{mod}(k, q)= 0$}
  \STATE Draw $|\xi_0|$ samples, and compute $v_{k} = \nabla f_{\xi_0} (z_k),$
  \ELSE
  \STATE Draw $|\xi_k|$ samples, and compute\\
   $v_k = \nabla f_{\xi_k} (z_k) - \nabla f_{\xi_k} (z_{k-1}) + v_{k-1}$.
  \ENDIF
  \STATE Compute $d_k = {H_k}v_k$ through SdLBFGS \cite{wang2017stochastic},
  \STATE $x_{k+1} =  x_{k} - \eta d_{k}$.
  \ENDFOR
 \STATE {\bfseries Output (in theory):} $x_\zeta$, where $\zeta \overset{\text{Unif}}{\sim} \{1,\ldots,K\}$.
 \STATE {\bfseries Output (in practice):}  $x_{K}$.
\end{algorithmic}
\end{algorithm}
~
\begin{algorithm}[!t]
\caption{SpiderSQN-M for Online Nonconvex Optimization}\label{Spider-SQNM-online}
\begin{algorithmic}[1]
\REQUIRE {$ \left|\xi_{0}\right|, \left|\xi_{k}\right|, q, K \in \mathbb{N}, \{\beta_k\}_{k=0}^{K-1} >0$.}
\STATE Set $\alpha_k = \frac{2}{k+1}$ for $k=0,...,K$ and $\lambda_k \in [\beta_k, (1+\alpha_k)\beta_{k}]$ for $k=0,...,K-1$.
\STATE Initialize $y_0 = x_0\in \mathbb{R}^d$.
\FOR {$k=0, 1, \ldots, K-1$}
  \STATE $z_{k} = (1-\alpha_{k+1})y_{k} + \alpha_{k+1} x_{k}$,
  \IF {$\text{mod}(k, q)= 0$}
  \STATE Draw $|\xi_0|$ samples, and compute $v_{k} = \nabla f_{\xi_0} (z_k),$
  \ELSE
  \STATE Draw $|\xi_k|$ samples, and compute\\ $v_k = \nabla f_{\xi_k} (z_k) - \nabla f_{\xi_k} (z_{k-1}) + v_{k-1}$.
  \ENDIF
  \STATE Compute $d_k = {H_k}v_k$ through SdLBFGS \cite{wang2017stochastic},
  \STATE $x_{k+1} =  x_{k} - \lambda_k d_{k}$,
  \STATE $y_{k+1} = z_{k} - \beta_k d_{k}$.
  \ENDFOR
 \STATE {\bfseries Output (in theory):} $x_\zeta$, where $\zeta \overset{\text{Unif}}{\sim} \{1,\ldots,K\}$.
 \STATE {\bfseries Output (in practice):}  $x_{K}$.
\end{algorithmic}
\end{algorithm}
\subsection{Other Momentum Acceleration Strategies}\label{momentum analysis}
The momentum scheme adopted in Algorithm \ref{SPIDER-SQN-M} is a vanilla one whose momentum coefficient $\alpha_k$ is iteration-wise diminishing. When the iteration $k$ becomes larger, $\alpha_k$ can be considerably small, leading to a limited acceleration. Thus, other momentum acceleration strategies are explored to alleviate this problem. Following are two powerful momentum schemes, where $\alpha_k$ can remain relatively large after many epochs. One is the epochwise-restart scheme, whose $\alpha_k$ is set as
\be \label{3.2}
\alpha_k = \frac{2}{\bmod (k, q)+1}, \quad k = 0, \ldots, K-1.
\ee
As the name suggests, $\alpha_k$ restarts at the beginning of each epoch.
Another effective momentum strategy is the epochwise-diminishing scheme with following momentum coefficient
\be \label{3.3}
\alpha_{k}=\frac{2}{\lceil k / q\rceil+ 1}, \quad k=0, \ldots, K-1,
\ee
where $\lceil  \cdot \rceil$ denotes the ceiling function. As defined in Eq.~(\ref{3.3}), the momentum coefficient $\alpha_k$ is a constant during a fixed epoch, and will diminish slowly as $k$ growing sharply.
To obtain the variants of SpiderSQN with above two momentum schemes, one just replace the $\alpha_k$ in Algorithm \ref{SPIDER-SQN-M} as defined.
 \section{Faster SQN Methods for Online Nonconvex Optimization}
In super large-scale learning, sample size $n$ can be considerably large or even infinite.
It is thus desirable to design algorithms with SFO complexity independent of $n$. Such algorithm are referred as online (streaming) algorithm. For this reason, we propose the online faster stochastic quasi-Newton method to solve the online problem:
\be
\min _{x\in \mathbb{R}^d} f(x) := {\mathbb{E}_{u\sim\mathbb{P}}[f_u(x)] }, \label{3.4}
\ee
where ${\mathbb{E}_{u\sim\mathbb{P}}[f_u(x)] }$ denotes a population risk over an underlying data distribution $\mathbb{P}$. Since the problem can be perceived as having infinite samples, it is impossible to evaluate the full gradient $\nabla f(x)$ by running across the whole dataset. The stochastic sampling thus is adopted as a surrogate strategy. Algorithm \ref{Spider-SQN-online} shows the detail steps of the proposed online SpiderSQN algorithm.

At steps 3 and 5 the gradient is estimated over the mini-batch samples {drawn from} the underlying distribution $\mathbb{P}$. Especially, due to the nature of the online data flow, these samples are sampled without replacement. The variant with vanilla momentum scheme is shown in Algorithm \ref{Spider-SQNM-online}. As for the counterparts with epochwise-restart momentum and  epochwise-diminishing momentum, one just replace the $\alpha_k$ in Algorithm \ref{Spider-SQNM-online} with the one defined in Eqs.~(\ref{3.2}) and (\ref{3.3}), respectively.
\begin{table*}[!t]
\centering
\caption{Total computational complexities of Algorithms~\ref{SdLBFGS} to \ref{Spider-SQNM-online} in an outer loop. Especially, the results of Algorithm~\ref{SdLBFGS} are obtained for $q$ iterations, an outer loop of Algorithms~\ref{SPIDER-SQN} to \ref{Spider-SQNM-online} includes $q$ computations of the stochastic gradient, $q$ calls of Algorithm~\ref{SdLBFGS}, and one computation of the full gradient.}
\label{complexity table}
\setlength{\tabcolsep}{1.3mm}{\small{\begin{tabular}{llllllllll}
\toprule
\multicolumn{2}{c}{Algorithm~\ref{SdLBFGS}} & \multicolumn{2}{c}{Algorithm~\ref{SPIDER-SQN}} & \multicolumn{2}{c}{Algorithm~\ref{SPIDER-SQN-M}} & \multicolumn{2}{c}{Algorithm~\ref{Spider-SQN-online}} & \multicolumn{2}{c}{Algorithm~\ref{Spider-SQNM-online}} \\ \midrule
step          & complexity                 & step         & complexity                     & step         & complexity                      & step        & complexity                             & step        & complexity                              \\
1              & $\mathcal{O}(d)$           & 3             & $\mathcal{O}(nd)$               & 4              & $\mathcal{O}(d)$                & 3            & $\mathcal{O}(\epsilon^{-2}d)$          & 4            & $\mathcal{O}(d)$                        \\
2              & $\mathcal{O}(d)$           & 5             & $\mathcal{O}(n^{1/2}d)$        & 6              & $\mathcal{O}(nd)$                & 5            & $\mathcal{O}(\epsilon^{-1}d)$          & 6            & $\mathcal{O}(\epsilon^{-2}d)$           \\
3-6            & $\mathcal{O}(md)$          & 7             & $\mathcal{O}(md)$              & 8              & $\mathcal{O}(n^{1/2}d)$         & 7            & $\mathcal{O}(md)$                      & 8            & $\mathcal{O}(\epsilon^{-1}d)$           \\
7              & $\mathcal{O}(d)$           & 8             & $\mathcal{O}(d)$               & 10             & $\mathcal{O}(md)$               & 8            & $\mathcal{O}(d)$                       & 10           & $\mathcal{O}(md)$                       \\
8-11           & $\mathcal{O}(md)$          &--               & --                               & 11-12          & $\mathcal{O}(d)$                &--              &--                                        & 11-12        & $\mathcal{O}(d)$                        \\ \hline
\textbf{total}          & $\mathcal{O}(qmd)$          & \textbf{total}           & $\mathcal{O}(nd+qmd)$           & \textbf{total}            & $\mathcal{O}(nd+qmd)$            & \textbf{total}          & $\mathcal{O}(\epsilon^{-2}d+qmd)$       & \textbf{total}          & $\mathcal{O}(\epsilon^{-2}d+qmd)$        \\ \bottomrule
\end{tabular}}}
\end{table*}
\section{Convergence Analysis}
In this section, we analyse the convergence rate of the faster stochastic quasi-Newton method and its online version. Detailed convergence analysis can be found in the Appendix.
\subsection{Convergence Analysis of Faster SQN Method}
First, the convergence properties of the four SpiderSQN-type of algorithms are presented. Let Assumptions \ref{assum1} to \ref{assum4} hold, and the following theorems are obtained.
\begin{theorem}\label{thm: ssqn}
	Apply Algorithm \ref{SPIDER-SQN} to solve the problem (\ref{eq: P}), and suppose $x_{\zeta}$ is its output. Let $q=|\xi_k|\equiv \sqrt{n}$, and $\eta \equiv \frac{(1+\sqrt{5})\sigma_{\mathrm{min}}}{2L\sigma_{\mathrm{max}}^2}$. Then, there is $x_{\zeta}$ satisfies $\mathbb{E}\|\nabla f(x_\zeta)\| \le \epsilon$ for any $\epsilon>0$ provided that the iterations number $K$ satisfies
	\begin{align}
		K \ge \mathcal{O}\bigg(\frac{f(x_0)-f^*}{\epsilon^2} \bigg).
	\end{align}
	Moreover, the total number of SFO calls is at most in the order of $\mathcal{O}(n+n^{1/2}\epsilon^{-2})$.
\end{theorem}

\begin{theorem}\label{thm: ssqnm}
Apply Algorithm \ref{SPIDER-SQN-M} to solve the problem (\ref{eq: P}), and suppose $z_{\zeta}$ is its output. Let $\alpha_k = \frac{2}{k+1}$, $q=|\xi_k|\equiv \sqrt{n}$, $\beta_k \equiv \frac{\sigma_{\mathrm{min}}}{(3+\sqrt{15})L\sigma_{\mathrm{max}}^2}$ and $\lambda_k \in [\beta_k, (1+\alpha_k)\beta_{k}]$. Then, there is $z_{\zeta}$ satisfies $\mathbb{E}\|\nabla f(z_\zeta)\| \le \epsilon$ for any $\epsilon>0$ provided that the iterations number $K$ satisfies
	\begin{align}
		K \ge \mathcal{O}\bigg(\frac{f(x_0)-f^*}{\epsilon^2} \bigg).
	\end{align}
	Moreover, the total number of SFO calls is at most in the order of $\mathcal{O}(n+n^{1/2}\epsilon^{-2})$.
\end{theorem}

\begin{theorem}\label{thm: ssqnmer}
Apply the SpiderSQN with either epochwise-restart momentum (SpiderSQNMER) or epochwise-diminishing momentum (SpiderSQNMED) to solve the problem (\ref{eq: P}), and suppose $z_{\zeta}$ is its output. Let $\alpha_k$ defined as Eqs.~(\ref{3.2}) and (\ref{3.3}) for SpiderSQNMER and SpiderSQNMED, respectively. Set $q=|\xi_k|\equiv \sqrt{n}$, $\beta_k \equiv \frac{\sigma_{\mathrm{min}}}{(3+\sqrt{15})L\sigma_{\mathrm{max}}^2}$ and $\lambda_k \in [\beta_k, (1+\alpha_k)\beta_{k}]$.
Then, for both algorithms there is $x_{\zeta}$ satisfies $\mathbb{E}\|\nabla f(x_\zeta)\| \le \epsilon$ for any $\epsilon>0$ provided that the iterations number $K$ satisfies
	\begin{align}
		K \ge \mathcal{O}\bigg(\frac{f(x_0)-f^*}{\epsilon^2} \bigg).
	\end{align}
	Moreover, the total number of SFO calls is at most in the order of $\mathcal{O}(n+n^{1/2}\epsilon^{-2})$.
\end{theorem}
\begin{remark}
There are two differences between Algorithm \ref{thm: ssqn} and Algorithm \ref{thm: ssqnm}\&\ref{thm: ssqnmer}: 1) Algorithm \ref{thm: ssqnm}\&\ref{thm: ssqnmer} introduce an extra parameter, \ie~$\alpha_k$, because of using momentum scheme; 2) the choice of $\beta_k$ in Algorithm \ref{thm: ssqnm}\&\ref{thm: ssqnmer} are different from that of $\eta$ in Algorithm \ref{thm: ssqn} (note that $\beta_k$ plays a same role as $\eta$). Algorithm \ref{thm: ssqnm}\&\ref{thm: ssqnmer} are the same except for the choice of $\alpha_k$ due to using different momentum schemes. Moreover, given required conditions in Algorithm \ref{thm: ssqn,thm: ssqnm,thm: ssqnmer}, the SFO complexity of Algorithm \ref{SPIDER-SQN} and its variants with different momentum schemes to satisfy the $\epsilon$-first-order stationary condition are $\mathcal{O}(n+n^{1/2}\epsilon^{-2})$, which matches the state-of-the-art results of first-order stochastic methods.
\end{remark}
\begin{figure*}[!t]
	\centering
	\begin{subfigure}{0.3\linewidth}
		\includegraphics[width=\linewidth]{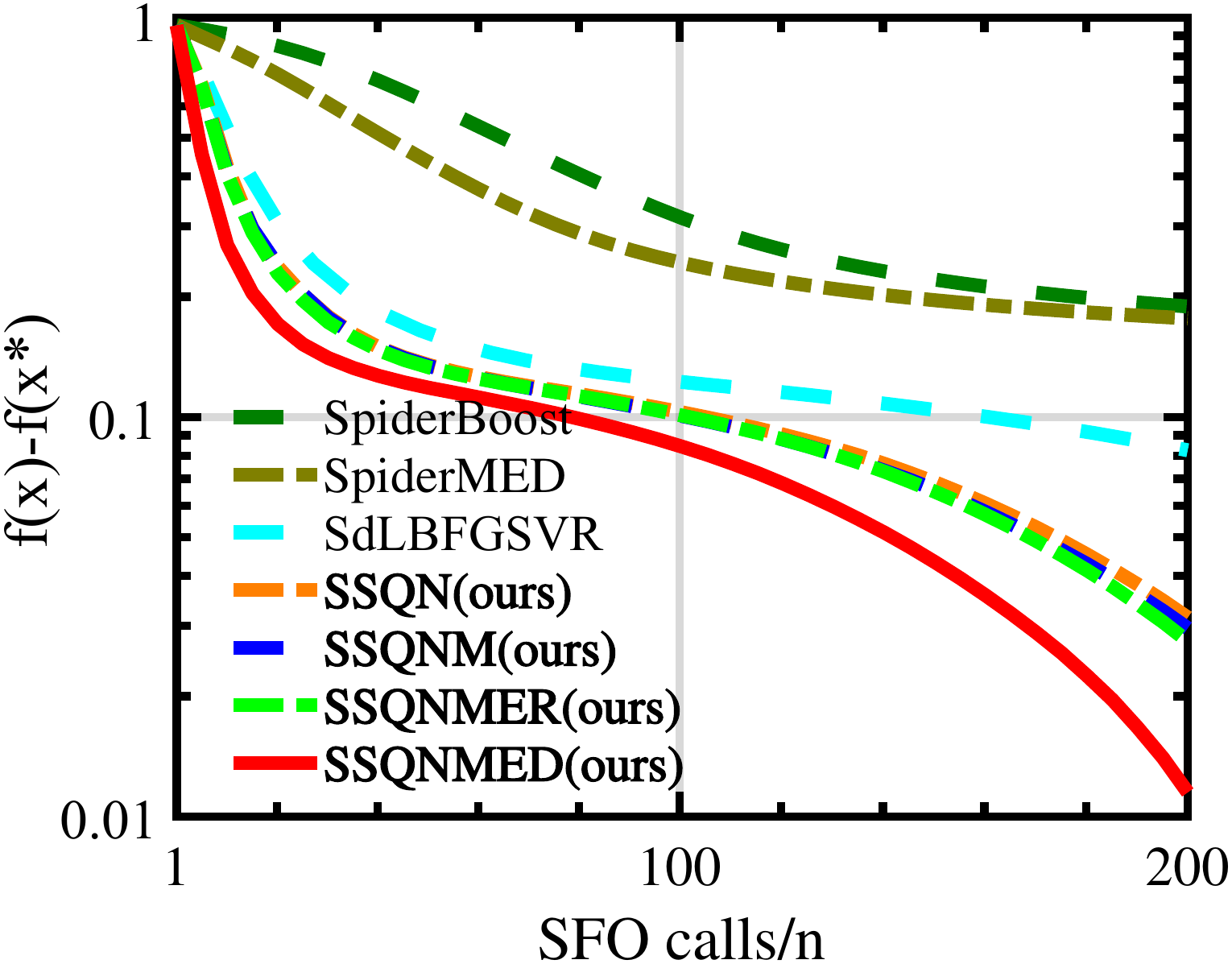}
		\caption{Data: a9a}
	\end{subfigure}
\begin{subfigure}{0.3\linewidth}
		\includegraphics[width=\linewidth]{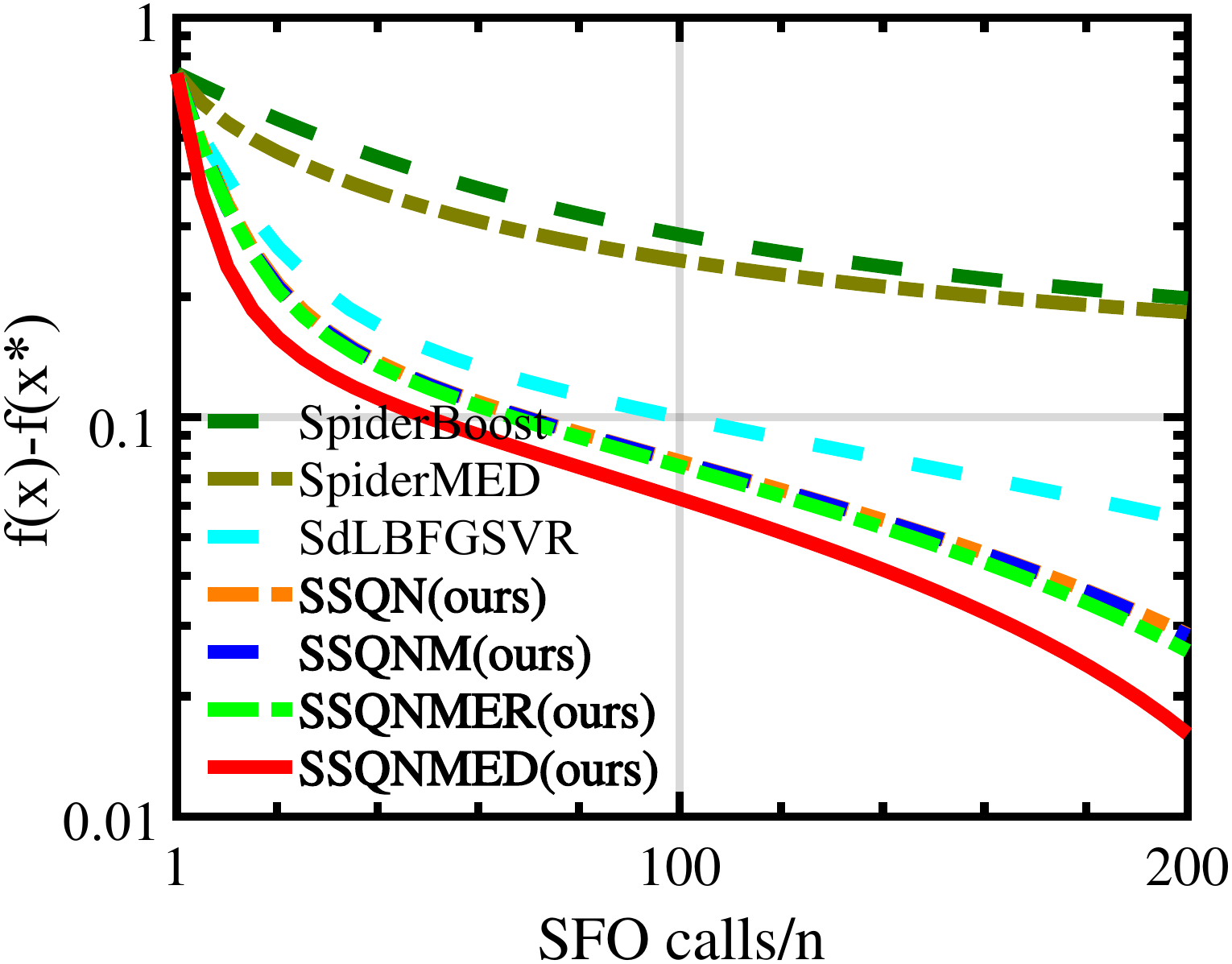}
		\caption{Data: w8a}
	\end{subfigure}
\begin{subfigure}{0.3\linewidth}
		\includegraphics[width=\linewidth]{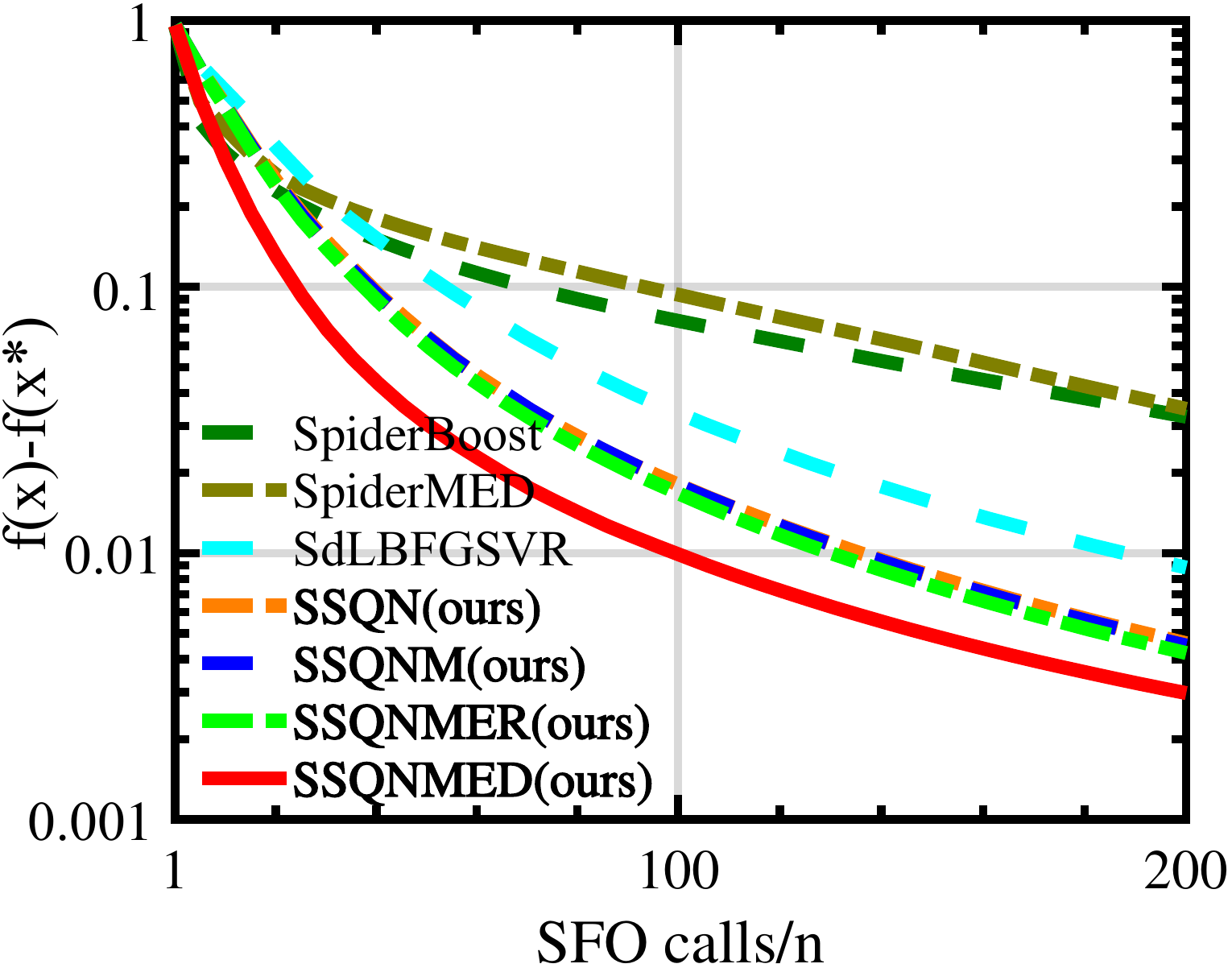}
		\caption{Data: mnist}
	\end{subfigure}%

	\begin{subfigure}{0.3\linewidth}
		\includegraphics[width=\linewidth]{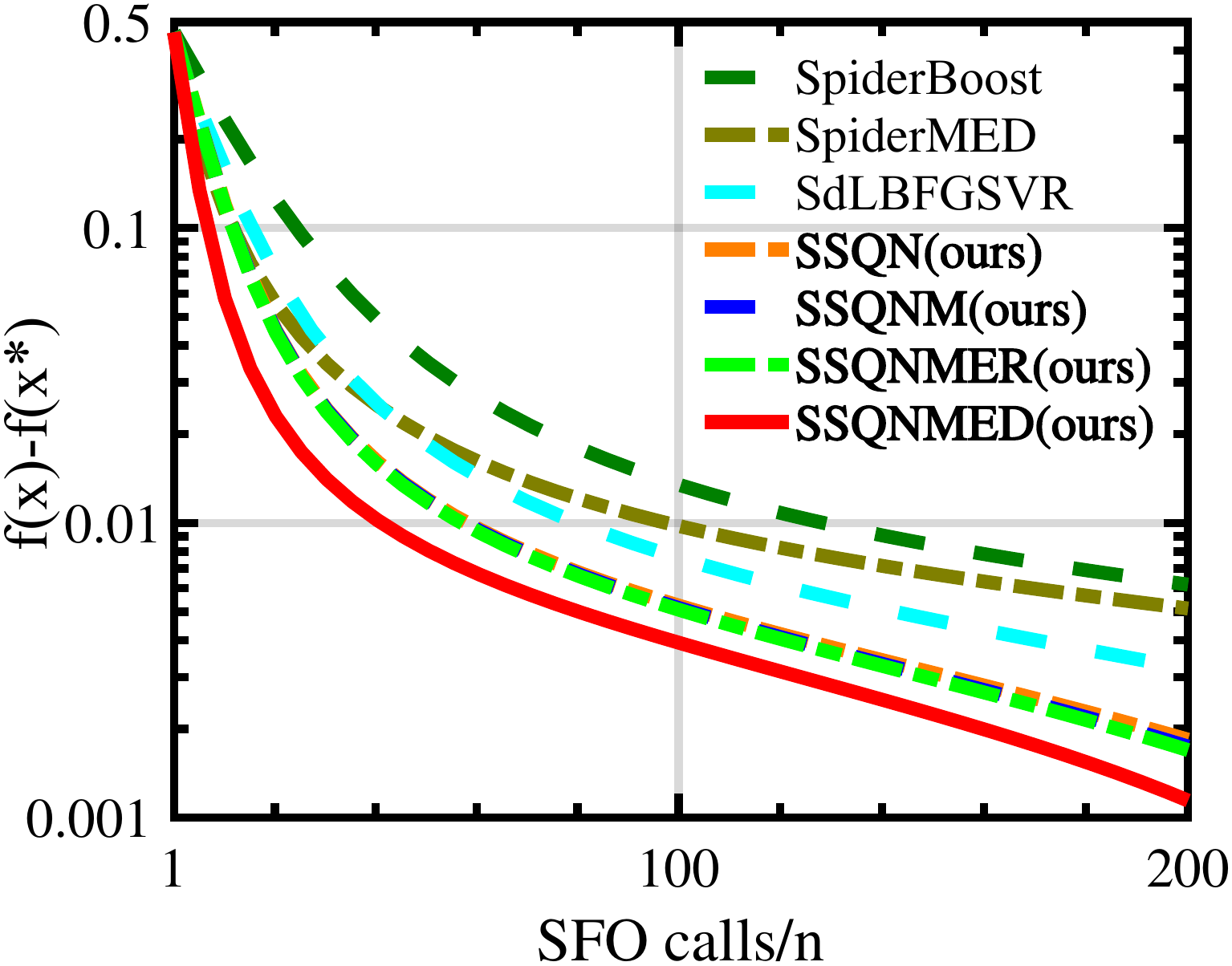}
		\caption{Data: ijcnn1}
	\end{subfigure}%
	\begin{subfigure}{0.3\linewidth}
		\includegraphics[width=\linewidth]{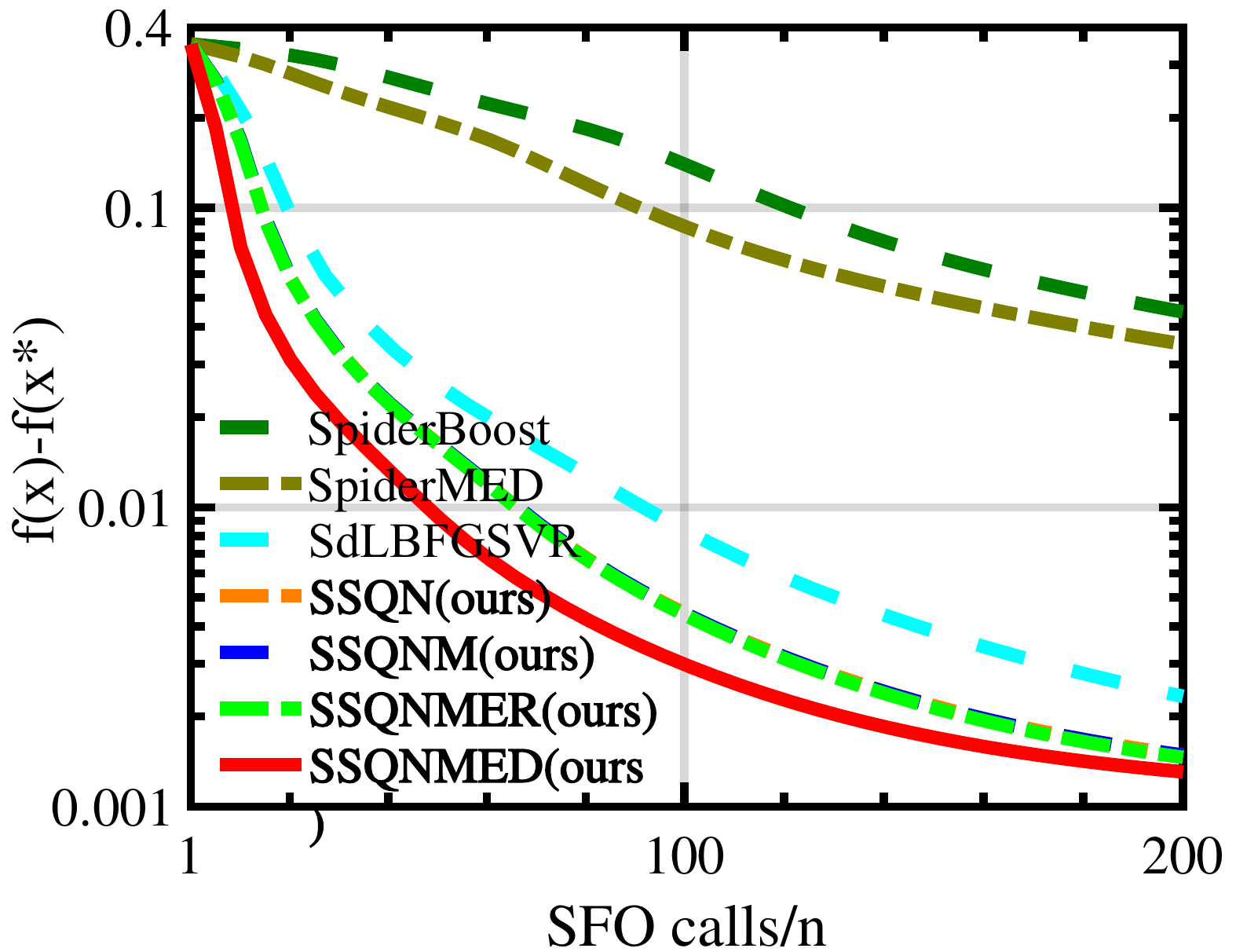}
		\caption{Data: covtype}
	\end{subfigure}%
	\begin{subfigure}{0.3\linewidth}
		\includegraphics[width=\linewidth]{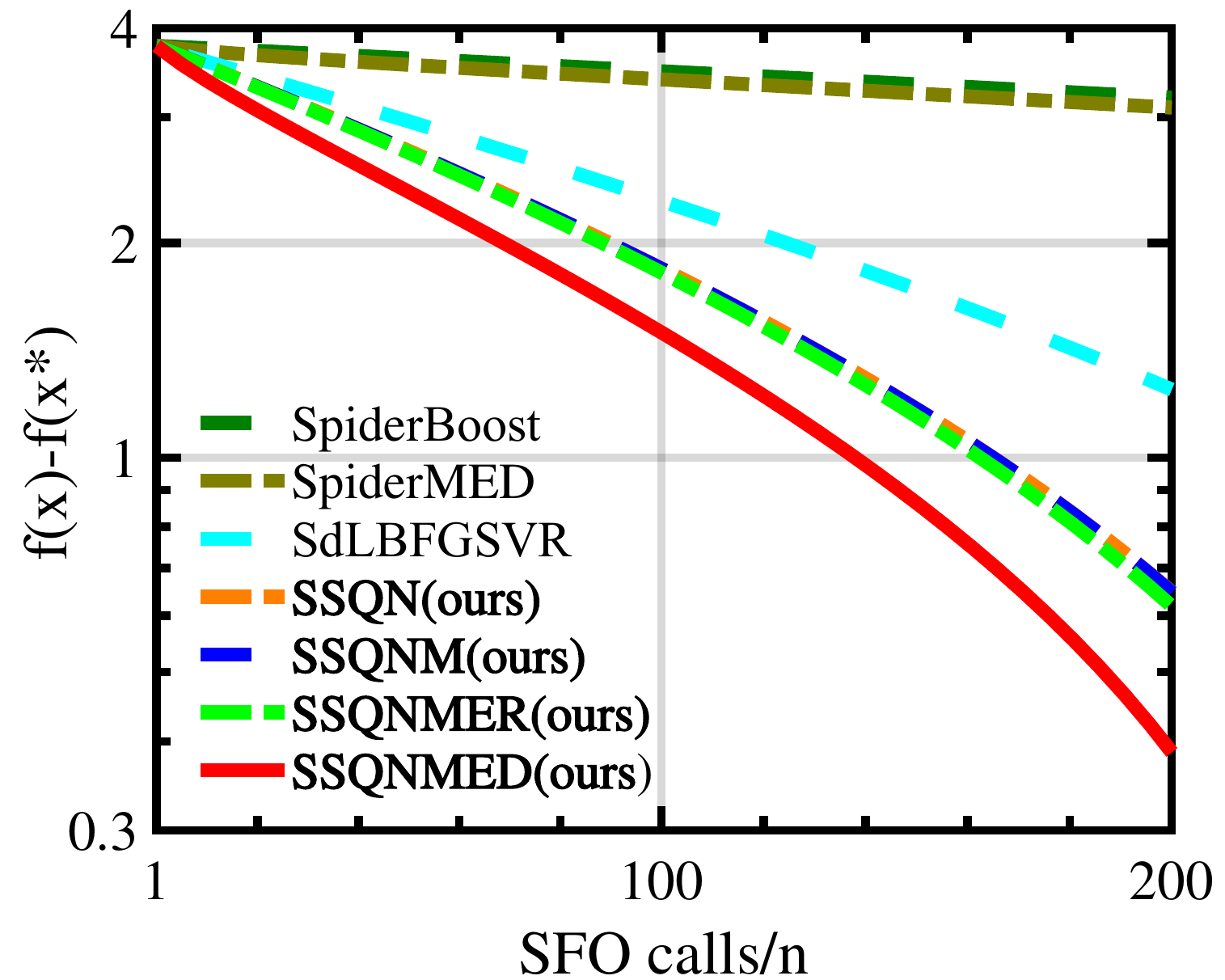}
		\caption{Data: synthetic data}
	\end{subfigure}%
	\caption{Comparison among algorithms for solving nonconvex SVM problems.}
\label{Experment1}
\end{figure*}
\subsection{Convergence Analysis of Online Faster SQN Method}
To study the SFO complexity of the online SpiderSQN-type of algorithms we let Assumptions \ref{assum1} to \ref{assum4} hold, and make an extra standard assumption (Algorithm \ref{assum5}).
{\begin{assumption}\label{assum5}
	There exists a constant $\sigma_1>0$ such that for all $x\in \mathbb{R}^d$ and all random samples $u\sim \mathbb{P}$, it holds that $\mathbb{E}_{u\sim \mathbb{P}} \|\nabla f_u(x) - \nabla f(x) \|^2 \leqslant \sigma_1^2$.
\end{assumption}
Assumption \ref{assum5} shows that the $\nabla f_u(x)$ is an unbiased estimator of $\nabla f(x)$ with bounded variance. Assumption \ref{assum5} is a standard assumption in online optimization analysis \cite{zhou2019momentum} and is for online case only.}
\begin{theorem}\label{thm: ssqn-online}
	Let additional Algorithm \ref{assum5} hold. Apply Algorithm \ref{Spider-SQN-online} to solve the online optimization problem (\ref{3.4}). Choose any desired accuracy $\epsilon>0$ and set parameters as
\begin{align}
		q=|\xi_k| = \sqrt{|\xi_0|} \equiv \sqrt{\left( \frac{\eta\sigma_{\mathrm{max}}}{ \beta^*} + 2 + \frac{ L^2 \eta^3\sigma_{\mathrm{max}}^3}{ \beta^*} \right)\frac{2\sigma_1^2}{\epsilon^2}}, \nonumber
	\end{align}
where $\beta^* = \frac{\eta\sigma_{\mathrm{min}} }{2} - \frac{L\eta^2 \sigma_{\mathrm{max}}^2}{2}- \frac{\eta^3\sigma_{\mathrm{max}}^3L^2 }{2 }$, and let $\eta \equiv \frac{(1+\sqrt{5})\sigma_{\mathrm{min}}}{2L\sigma_{\mathrm{max}}^2}$. Then, the output $x_{\zeta}$ of this algorithm satisfies $\mathbb{E}\|\nabla f(x_\zeta)\| \le \epsilon$ given that the total number of iterations $K$ satisfies
	\begin{align}
		K \ge \mathcal{O}\bigg(\frac{f(x_0)-f^*}{\epsilon^2} \bigg).
	\end{align}
	Moreover, the SFO complexity is in the order of $\mathcal{O}(\epsilon^{-3})$.
\end{theorem}

\begin{theorem}\label{thm: ssqnm-online}
	Let additional Algorithm \ref{assum5} hold. Apply online Algorithm \ref{Spider-SQNM-online} to solve the online optimization problem (\ref{3.4}). Choose any desired accuracy $\epsilon>0$ and set parameters as
\begin{align}
		\alpha_{k} = \frac{2}{k+1}, \quad q=|\xi_k| = \sqrt{|\xi_0|} \equiv \sqrt{ \frac{4(1+\beta/\beta^*)\sigma_1^2}{\epsilon^2}}, \nonumber
\end{align}
where $\beta^* = \beta (\frac{\sigma_{\mathrm{min}}}{2}-3L\beta\sigma_{\mathrm{max}}^2 - 3L^2\beta^2\sigma_{\mathrm{max}}^3)$, $\beta \equiv \frac{\sigma_{\mathrm{min}}}{(3+\sqrt{15})L\sigma_{\mathrm{max}}^2}$. Let $\beta_k=\beta$, $\lambda_k \in [\beta_k, (1+\alpha_k)\beta_{k}]$. Then, the output $z_{\zeta}$ of this algorithm satisfies $\mathbb{E}\|\nabla f(z_\zeta)\| \le \epsilon$ provided that the total number of iterations $K$ satisfies
	\begin{align}
		K \ge \mathcal{O}\bigg(\frac{f(x_0)-f^*}{\epsilon^2} \bigg).
	\end{align}
	Moreover, the SFO complexity is in the order of $\mathcal{O}(\epsilon^{-3})$.
\end{theorem}

\begin{theorem}\label{thm: ssqnmer-online}
	Let additional Algorithm \ref{assum5} hold. Apply the online SpiderSQNMER or online SpiderSQNMED to solve the problem (\ref{3.4}). Choose any desired accuracy $\epsilon>0$, let $\alpha_k$ defined as Eqs.~(\ref{3.2}) and (\ref{3.3}) for online SpiderSQNMER and online SpiderSQNMED, respectively. And set parameters as
\begin{align}
		q=|\xi_k| = \sqrt{|\xi_0|} \equiv \sqrt{ \frac{4(1+\beta/\beta^*)\sigma_1^2}{\epsilon^2}}, \nonumber
	\end{align}
where $\beta^* = \beta (\frac{\sigma_{\mathrm{min}}}{2}-3L\beta\sigma_{\mathrm{max}}^2 - 3L^2\beta^2\sigma_{\mathrm{max}}^3)$, $\beta \equiv \frac{\sigma_{\mathrm{min}}}{(3+\sqrt{15})L\sigma_{\mathrm{max}}^2}$. Let $\beta_k =\beta$, $\lambda_k \in [\beta_k, (1+\alpha_k)\beta_{k}]$. Then, the output $z_{\zeta}$ of both algorithms satisfy $\mathbb{E}\|\nabla f(z_\zeta)\| \le \epsilon$ provided that the total number of iterations $K$ satisfies
	\begin{align}
		K \ge \mathcal{O}\bigg(\frac{f(x_0)-f^*}{\epsilon^2} \bigg).
	\end{align}
	Moreover, the SFO complexity is in the order of $\mathcal{O}(\epsilon^{-3})$.
\end{theorem}
\begin{remark}
There are two differences between Algorithm \ref{thm: ssqn-online} and  Algorithm \ref{thm: ssqnm-online}\&\ref{thm: ssqnmer-online}: 1) Algorithm \ref{thm: ssqnm-online}\&\ref{thm: ssqnmer-online} introduce an extra parameter, \ie, momentum coefficient $\alpha_k$ because of using momentum scheme; 2) the choice of $\beta_k$ in Algorithm \ref{thm: ssqnm-online}\&\ref{thm: ssqnmer-online} are different from that of $\eta$ in Algorithm \ref{thm: ssqn-online} (note that $\beta_k$ plays a same role as $\eta$). Algorithm \ref{thm: ssqnm-online} and Algorithm \ref{thm: ssqnmer-online} are the same except for the choice of $\alpha_k$ due to using different momentum schemes.  Moreover, given required conditions in Algorithm \ref{thm: ssqn-online,thm: ssqnm-online,thm: ssqnmer-online}, the SFO complexity of Algorithm \ref{Spider-SQN-online} and its variants with different momentum schemes to satisfy the $\epsilon$-first-order stationary condition are $\mathcal{O}(\epsilon^{-3})$, which matches the state-of-the-art results of first-order stochastic methods.
\end{remark}
\begin{figure*}[!t]
	\centering
	\begin{subfigure}{0.3\linewidth}
		\includegraphics[width=\linewidth]{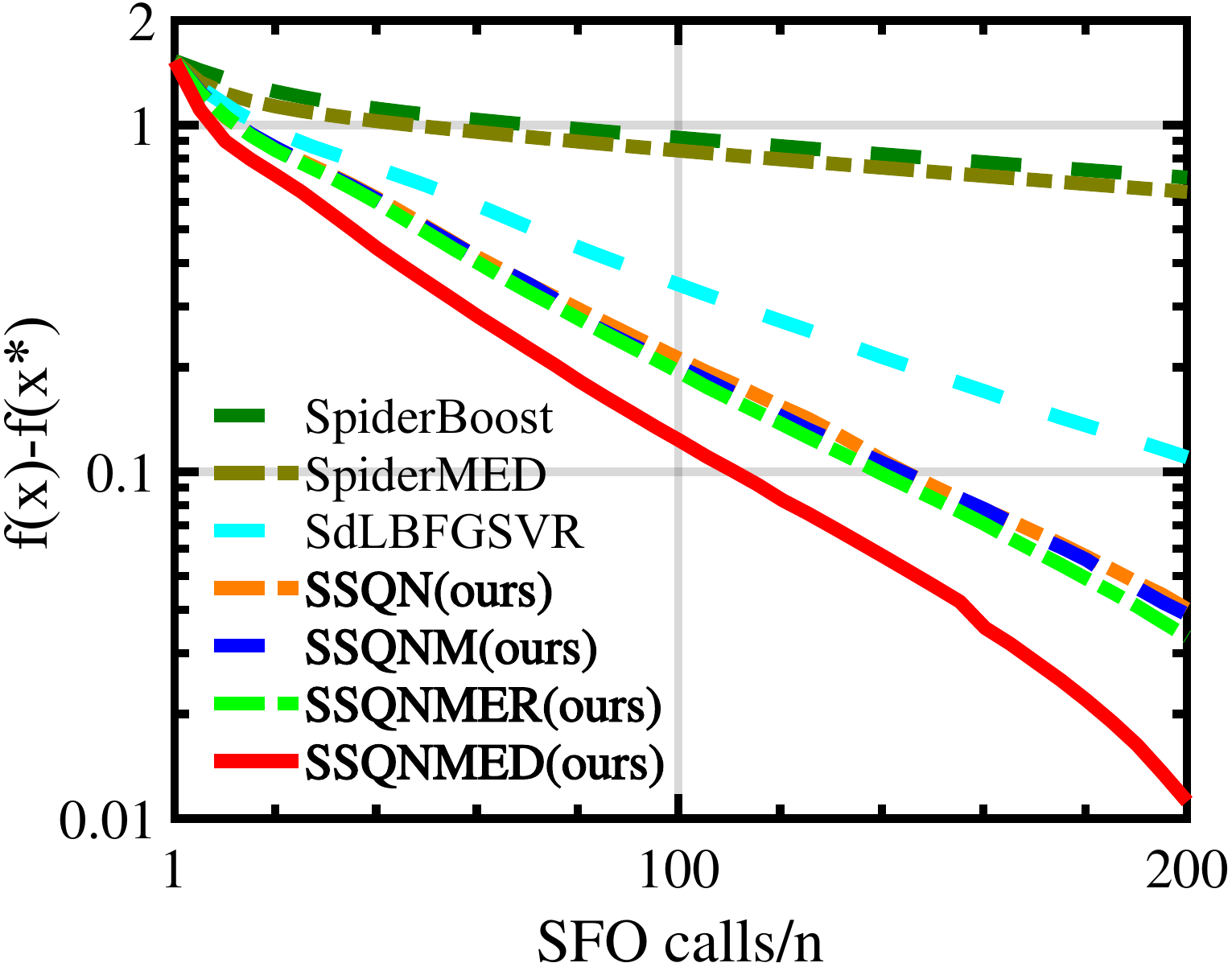}
		\caption{Data: a9a}
	\end{subfigure}
\begin{subfigure}{0.3\linewidth}
		\includegraphics[width=\linewidth]{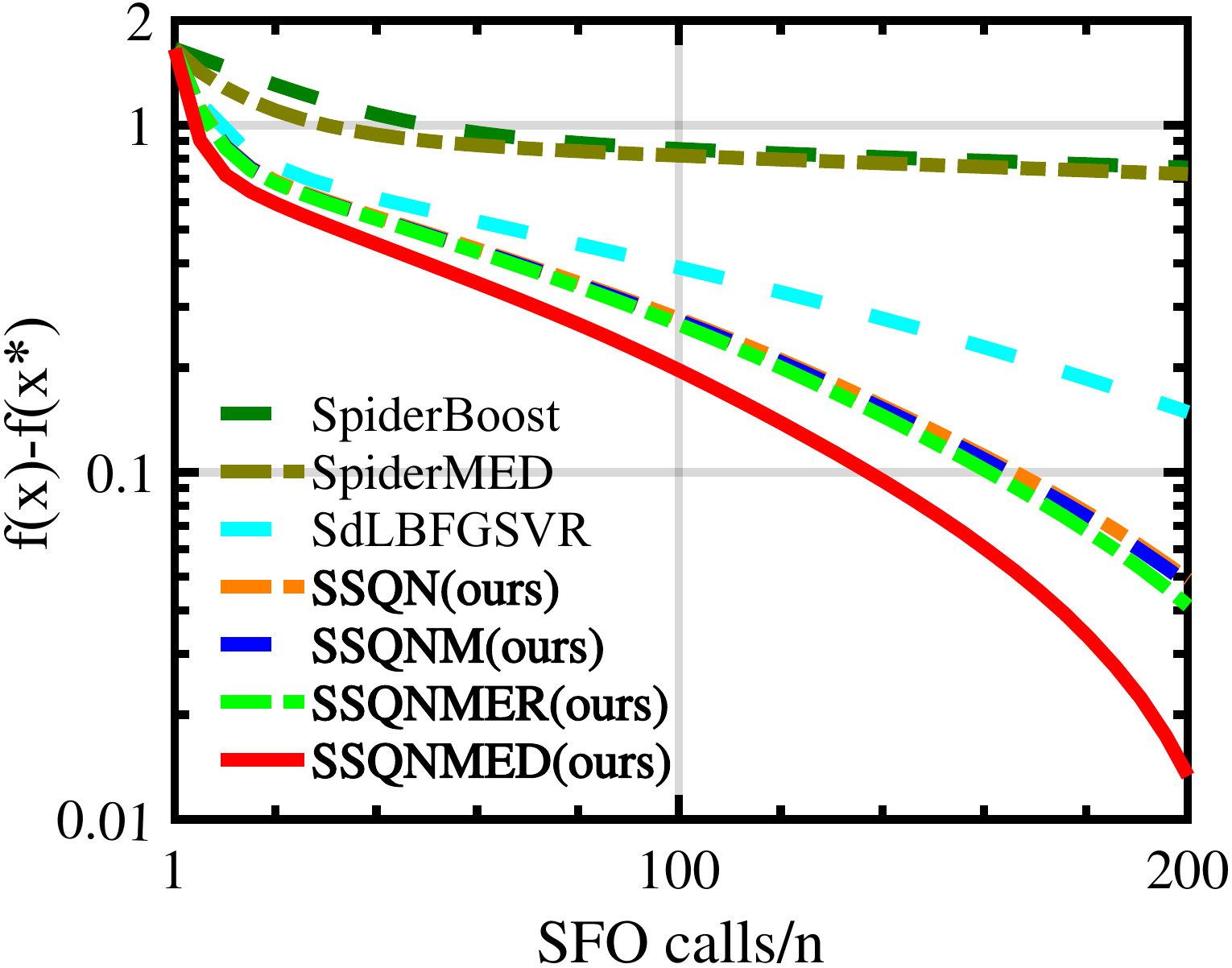}
		\caption{Data: w8a}
	\end{subfigure}
\begin{subfigure}{0.3\linewidth}
		\includegraphics[width=\linewidth]{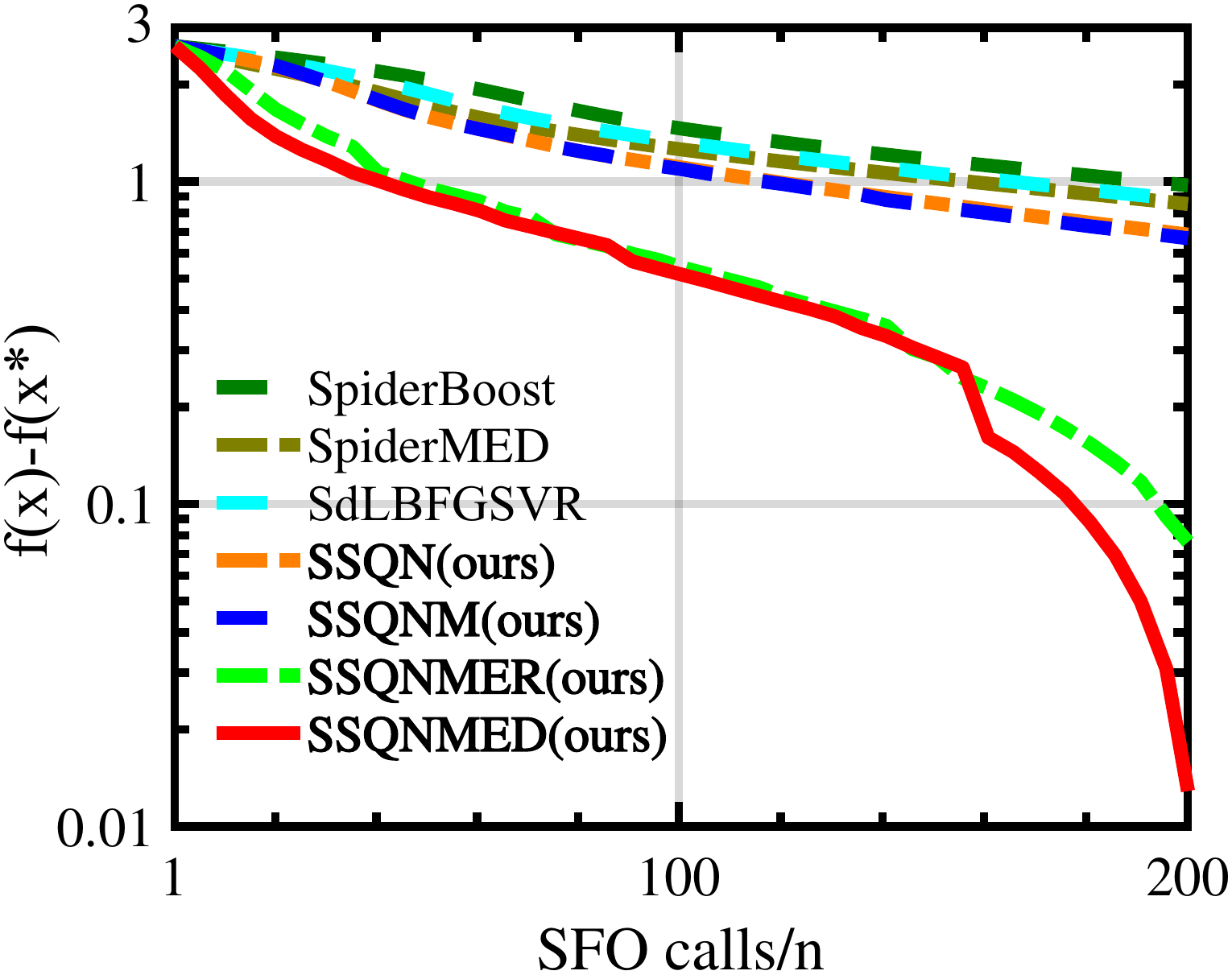}
		\caption{Data: mnist}
	\end{subfigure}%

	\begin{subfigure}{0.3\linewidth}
		\includegraphics[width=\linewidth]{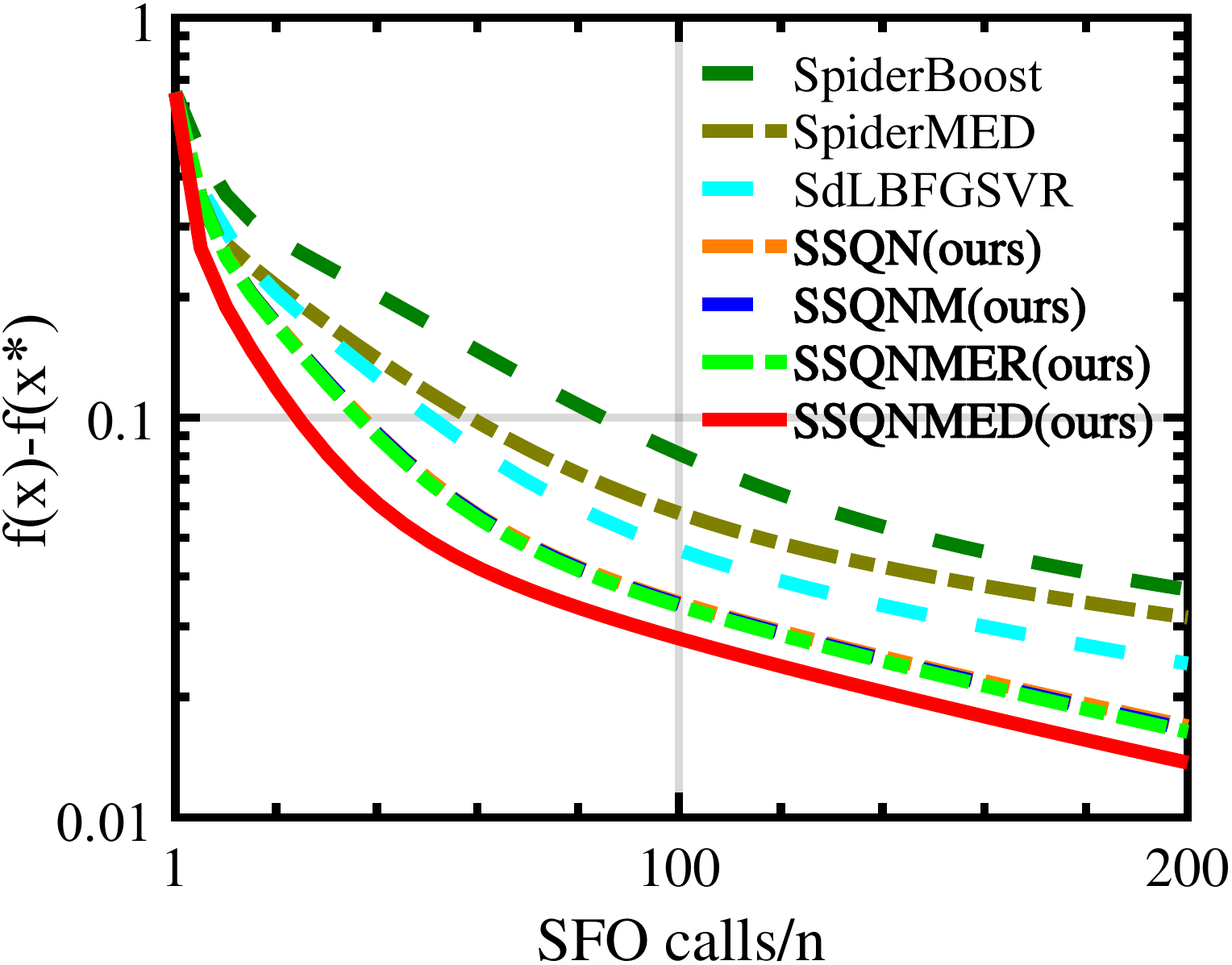}
		\caption{Data: ijcnn1}
	\end{subfigure}%
	\begin{subfigure}{0.3\linewidth}
		\includegraphics[width=\linewidth]{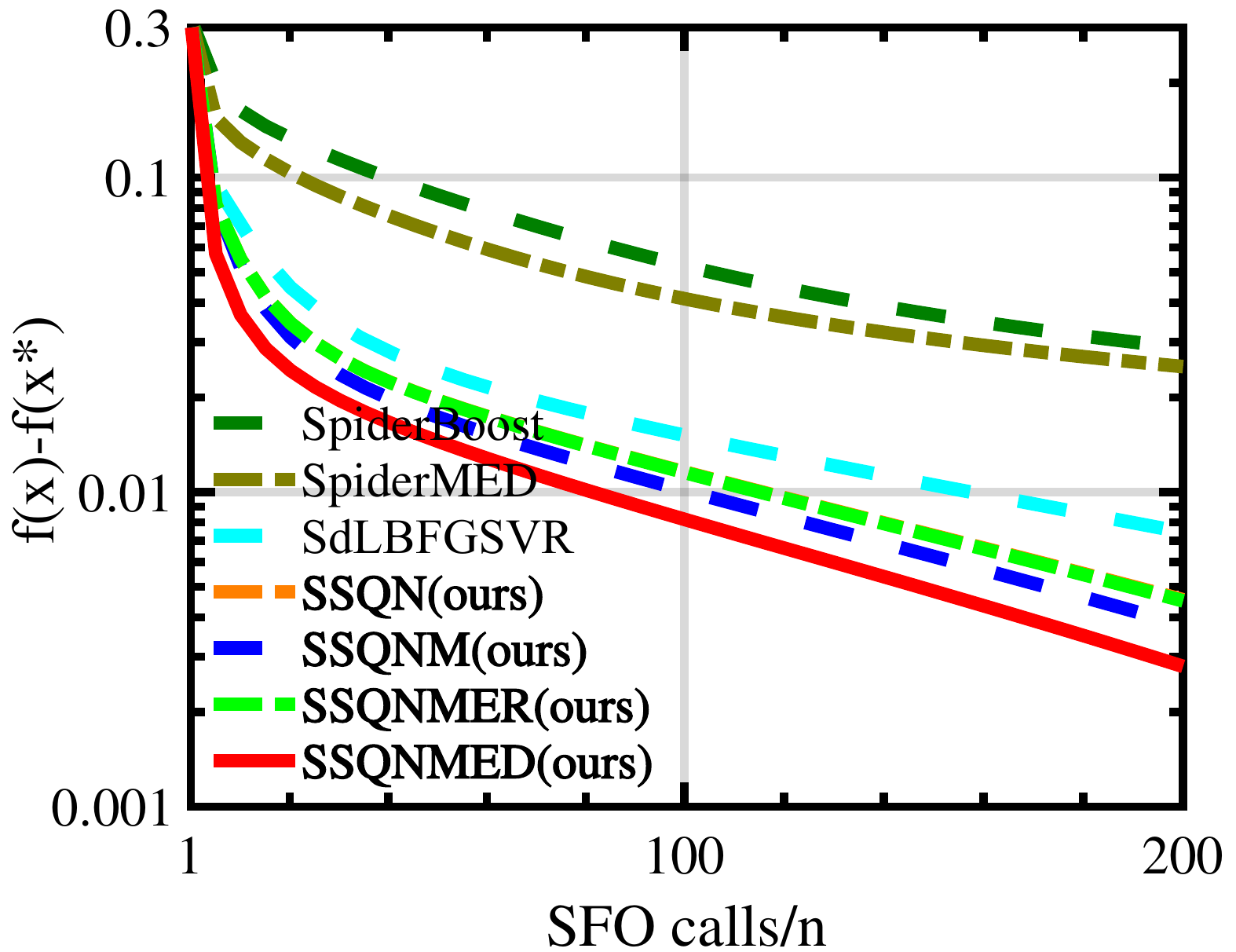}
		\caption{Data: covtype}
	\end{subfigure}%
	\begin{subfigure}{0.3\linewidth}
		\includegraphics[width=\linewidth]{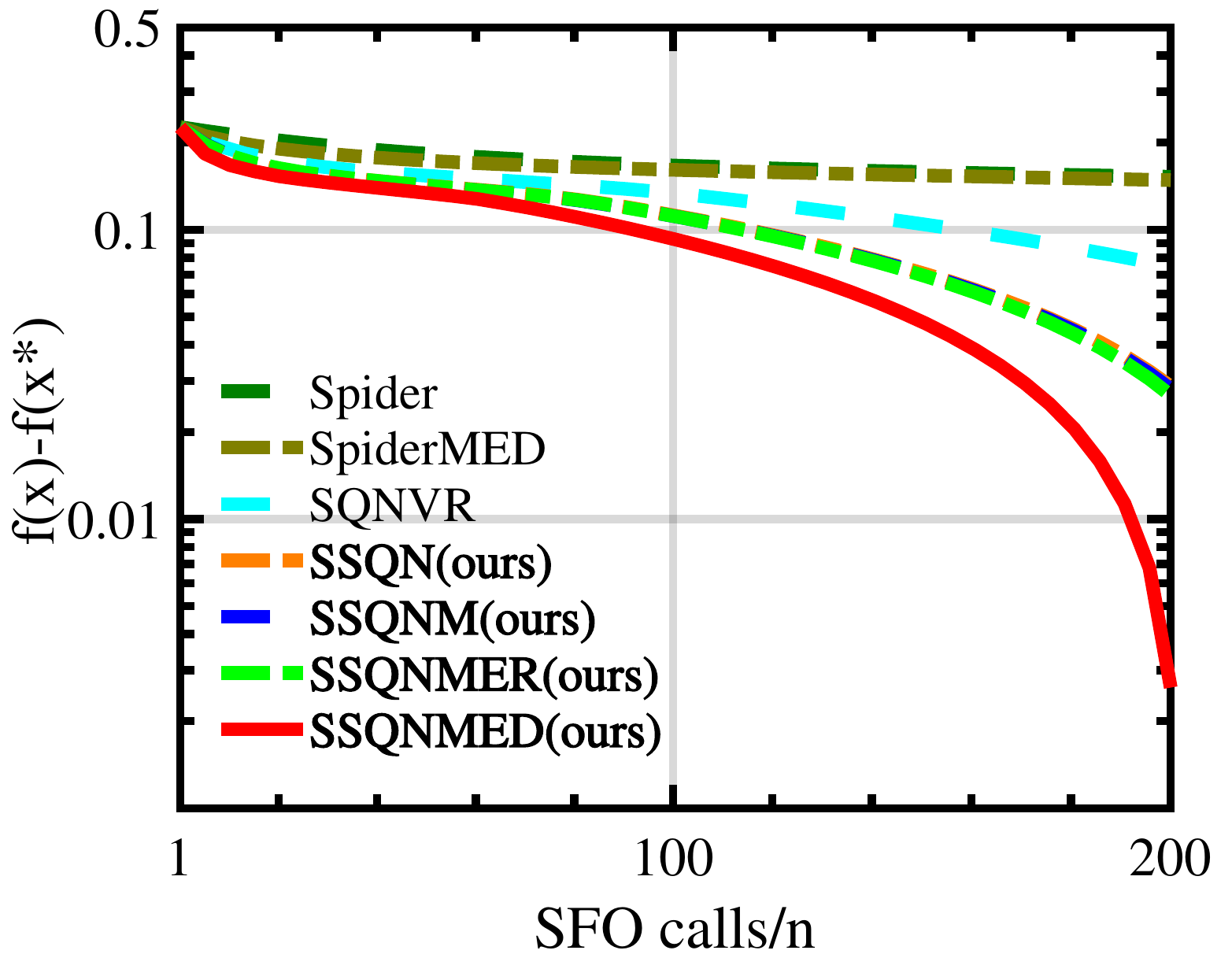}
		\caption{Data: synthetic data}
	\end{subfigure}%
	\caption{Comparison among algorithms for solving nonconvex robust linear regression problems.}
\label{Experment2}
\end{figure*}
\subsection{The Lower Bound}
We will present the optimality of our algorithms in the perspective of algorithmic lower bound result \cite{carmon2017lower}, which can be obtained by following the analyses in \cite{fang2018spider}. For the finite-sum case, given any  random algorithm ${\mathcal{A}}$ that maps functions $f: \mathbb{R}^d \to \mathbb{R}$ to a sequence of iterates in $\mathbb{R}^{d+1}$, with
\begin{align}\label{algor-need}
[\mathbf{x}^k; i_k] &=
{\mathcal{A}}^{k-1} \big(\mathbbm{\xi}, \nabla f_{i_0}(\mathbf{x}^0), \nabla f_{i_1}(\mathbf{x}^1), \ldots, \nabla f_{i_{k-1}}(\mathbf{x}^{k-1})   \big)
,\nonumber \\ & \quad k\geq 1,
\end{align}
where ${\mathcal{A}}^{k}$  denotes measure mapping into $\mathbb{R}^{d+1}$, $i_k$ is the individual function chosen by ${\mathcal{A}}$ at iteration $k$, and $\mathbbm{\xi}$ is uniform random vector from $[0,1]$.  Moreover, there is $[\mathbf{x}^0; i_0] = {\mathcal{A}}^0(\mathbbm{\xi})$,  where  ${\mathcal{A}}^0$ is a  measure mapping.
The lower bound result for solving (\ref{eq: P}) is stated in Theorem \ref{theo:lowerbdd}.
\begin{theorem}[Lower bound for SFO complexity for the finite-sum case]\cite{fang2018spider}
\label{theo:lowerbdd}
	For any $L>0$, $\Delta >0$, and {$2\leq n \leq 	\mathcal{O}\left( \Delta^2 L^2\cdot \epsilon^{-4} \right)$}, for any algorithm ${\mathcal{A}}$ satisfying \eqref{algor-need}, there exists a dimension
$d = {\mathcal{O}}\big(\Delta^2 L^2 \cdot n^2\epsilon^{-4}\big),$
and a function $f$ satisfying Assumptions~\ref{assum1}-\ref{assum5} for the finite-sum case,  such that in order to find an $\epsilon$-first-order stationary point must cost at least $
\mathcal{O}\big(L \Delta \cdot n^{1/2}\epsilon^{-2}\big)
$ stochastic gradient accesses.
\end{theorem}
Note that the condition $n\leq \mathcal{O}(\epsilon^{-4})$ in Theorem \ref{theo:lowerbdd} ensures the lower bound $\mathcal{O}(n^{1/2}\epsilon^{-2}) = \mathcal{O}(n+n^{1/2}\epsilon^{-2})$. Therefore, the upper bound in Theorem \ref{thm: ssqn} matches the lower bound in Theorem \ref{theo:lowerbdd} up to a constant factor of relevant parameters, and is thus near-optimal.
 The proof of Theorem \ref{theo:lowerbdd} provided in the Appendix utilizes a specific counterexample function that requires at least $\mathcal{O}(n^{1/2}\epsilon^{-2})$ stochastic gradient accesses, which is inspired by \cite{fang2018spider,carmon2017lower,nesterov2018lectures}.
\begin{remark}
Through setting $n = \mathcal{O}(\epsilon^{-4})$ the lower bound complexity in Theorem \ref{theo:lowerbdd} can achieve $\mathcal{O}(\epsilon^{-4})$.
It is necessary to emphasize that this does not violate the upper bound in the online case, i.e. $ \mathcal{O}(\epsilon^{-3})$ (Theorems~\ref{thm: ssqn-online}-\ref{thm: ssqnmer-online}), since the counterexample established in the lower bound depends not on  the stochastic gradient variance  $\sigma_1^2$ specified in Assumption~\ref{assum5} but the example number $n$.
To obtain the lower bound result for the online case with the additional Assumption~\ref{assum5},  one can just construct a counterexample that requires $\mathcal{O}({\epsilon}^{-3})$ stochastic gradient accesses with the knowledge of $\sigma_1^2$ instead of $n$.
\end{remark}
\subsection{Computational  Complexity}
In the following, we will analyze the time complexity of the proposed algorithms and show that the extra computation costs of computing inverse Hessian approximation matrix and using momentum acceleration are negligible.

First, we analyze the computational cost of Algorithm~\ref{SdLBFGS}. In Step 1, the computation of $\gamma_k^{-1}$ involves two inner product, which takes $2d$ multiplications. In Step 2, the computation involves two inner product and one scalar-vector product, which takes $3d$ multiplications. First recursive loop (\ie, Steps 3 to 5) involves $2m$ scalar-vector multiplications and $m$ vector inner products, which takes $3md$ multiplications. So does the second loop (\ie, Steps 8 to 10). Step 7 involving a scalar-vector product takes $d$ multiplications. Therefor, the whole procedure takes $(6m+6)d$ multiplications.

Then, we turn to Algorithm~\ref{SPIDER-SQN-M}. Step~4 involves scalar-vector products, which takes $2d$ multiplications. In Step 6, the computation of full gradient takes at least $2nd$ multiplications. In Step~8, the computation of stochastic gradients with batch-size $n^{1/2}$ takes $2n^{1/2}d$ multiplications. In Steps 10, $(6m+6)d$ multiplications are necessary  for calling Algorithm~\ref{SdLBFGS}. Steps 11 and 12 involving scalar-vector products need $d$ multiplications. Therefore, the total computational cost in an outer loop involves $ [(6m+6)q + 2n+2n^{1/2}q + 4q]d$ multiplications.

Based on above analyses, the computational cost of other algorithms can be obtained easily. For algorithms without momentum acceleration, one {needs to} omit the extra computation cost ($2d$ multiplications) of computing momentum term. As for algorithms without using  approximate Hessian information, one needs to omit the extra computational cost of calling Algorithm~\ref{SdLBFGS}.

We summarize the computational complexity of each algorithm during an outer loop with $q$ iterations (for finite-sum case there is $\cO(q)=\cO(n^{1/2})$, while for online case there is $\cO(q)=\cO(\epsilon^{-1})$) in Table~\ref{complexity table}.
As shown in Table~\ref{complexity table}, for finite-sum case, the extra computation costs of computing approximate Hessian information and using momentum acceleration take up $\frac{mn^{1/2}}{n+mn^{1/2}}$ in the whole procedure. Since $m$ usually ranges from 5 to 20 as suggested in \cite{nocedal2006numerical} and $n$ is sufficiently large in big data situation, the extra computation thus is negligible. So does the online case, when $\epsilon$ is considerably small. Note that for analyses convenience, we reasonably assume $2d$ multiplications are needed when computing a stochastic gradient for general machine learning problem.
\section{Experiments}\label{sec: exp}
In this section, to demonstrate the promising performance of the proposed algorithms, we compare our methods with some state-of-the-art stochastic quasi-Newton algorithms and stochastic first-order algorithms for nonconvex optimization. Following are brief introductions of algorithms used in our experiments.\\
  \noindent{\bf SpiderBoost} \cite{wang2018spiderboost}: SpiderBoost is a boosting version of SPIDER, which takes up a more aggressive stepsize than SPIDER and thus outperforms SPIDER in practice.\\
  \noindent{\bf SdLBFGSVR} \cite{wang2017stochastic}: SdLBFGSVR is a SQN method (more specifically, stochastic damped L-BFGS method) equipped with the SVRG variance reduction technique.\\
  \noindent{\bf SpiderMED} \cite{zhou2019momentum}: ProxSPIDER-MED \cite{zhou2019momentum} is a proximal method that uses the  epochwise-diminishing momentum scheme to improve the practical performance of SpiderBoost. Especially, ProxSPIDER-MED is the faster one among all momentum variants of SpiderBoost proposed in \cite{zhou2019momentum}. Since our paper does not touch upon nonconvex nonsmooth optimization, we adopt the ProxSPIDER-MED without proximal operator and call it SpiderMED.\\
  \noindent{\bf Our methods}: Our methods include four SpiderSQN (SSQN) type of methods, \ie, SSQN (Aslgorithm \ref{SPIDER-SQN}), SSQN with vanilla momentum scheme (SSQNM, \ie., Algorithm \ref{SPIDER-SQN-M}),  SSQN with  epochwise-restart momentum (SSQNMER) and SSQN with epochwise-diminishing momentum (SSQNMED). Note that SSQNMER and SSQNMED are proposed in section \ref{momentum analysis}.

Follow the experiment setting in \cite{zhou2019momentum}, we choose a fixed mini-batch size $256$ and the epoch length $q$ is set to $2n/256$. When implement the SdLBFGS \cite{wang2017stochastic}, we set the memory size to $m = 5$ as suggested in \cite{nocedal2006numerical}, and fix the $\sigma$  for each comparison.  Moreover, we implement experiments on synthetic data for the complement of real datasets, which are generated as \cite{wang2017stochastic}.\\
{\noindent\bf{Generating Synthetic Data:}} The training and testing points $(a, b)$ are generated in the following manner. First, we generate a sparse vector $a$ with 5\% nonzero components following the uniform distribution on $[0,1]^{n}$, and then set $b = \text{sign}({u}, a)$ for some $u \in \mathbb{R}^n$ drawn from the uniform distribution on $[-1,1]^{n}$.\\
{\noindent\bf{Descriptions of Datasets:}}
We implement all experiments on five public datasets from the LIBSVM \cite{chang2011libsvm} and a synthetic data as the complement to these public datasets is summarized in Algorithm \ref{tab1}. Especially, as for the mnist dataset we use the one-vs-rest technique to convert it to a binary class data.
\begin{table}[!t]
\centering
\caption{Descriptions of Datasets.}
\label{tab1}
\scalebox{1}{\begin{tabular}{@{}cccc@{}}
\toprule
datasets & \#samples & \#features & \#classes \\ \midrule
a9a & 32,561 & 123 & 2 \\
w8a & 64,700 & 300 & 2 \\
ijcnn1 & 141,691 & 22 & 2 \\
mnist & 60,000 & 780 & 2 \\
covtype & 581,012 & 54 & 2 \\
synthetic data & 100,000 & 5,000 & 2 \\\bottomrule
\end{tabular}}
\end{table}

\subsection{Nonconvex Support Vector Machine}
First, above algorithms are applied to solve the nonconvex support vector machine (SVM) problem with a sigmoid loss function:
\begin{align*}
\min_{x\in \mathbb{R}^d} f(x):=\frac{1}{n} \sum_{i=1}^{n} (1-\mathrm{tanh}(b_i\inner{x}{a_i})) + r \|x\|^2,
\end{align*}
where $a_i\in \mathbb{R}^d$ denotes the $i$-th sample and $b_i\in {\pm 1}$ is the corresponding label. In the experiments, the learning rate $\eta$ and regular coefficient $r$ for all algorithms are both fixed as $0.001$. Moreover, in algorithms with momentum scheme $\beta_k$ is fixed as $\eta$, and $\lambda_k$ remains the same for each comparison.

The experiment results on those four datasets are shown in Fig.~\ref{Experment1}, where $f(x)$ is the function value and $f(x^*)$ is a suitable constant for each case. First, as for datasets w8a and ijcnn1 the initial solutions to all algorithms are {drawn from} the
standard norm distribution, while for datasets a9a and mnist they take the original point.
 As Fig.~\ref{Experment1} depicts, all these stochastic quasi-Newton methods (including SdlBFGSVR and four SpiderSQN (SSQN)-type of algorithms) outperform stochastic first-order methods (including Spider and SpiderMED) by a considerably large margin, which demonstrates the promising nature of stochastic quasi-Newton methods for nonconvex optimization. And one can see that the basic algorithm SSQN converges more faster than SdLBFGSVR, which is corresponding to the theoretical result that the proposed method has a lower SFO complexity than SdLBFGSVR. Meanwhile, among the four SSQN-type of algorithms, three algorithms with different momentum schemes all have a better performance than the SSQN. Moreover, among these three algorithms, the one using epochwise-diminishing momentum (SSQNMED) achieves the best performance, while the one using the iterationwise-diminishing momentum (SSQNM) achieves the poorest.
\begin{figure*}[!t]
	\centering
	\begin{subfigure}{0.3\linewidth}
		\includegraphics[width=\linewidth]{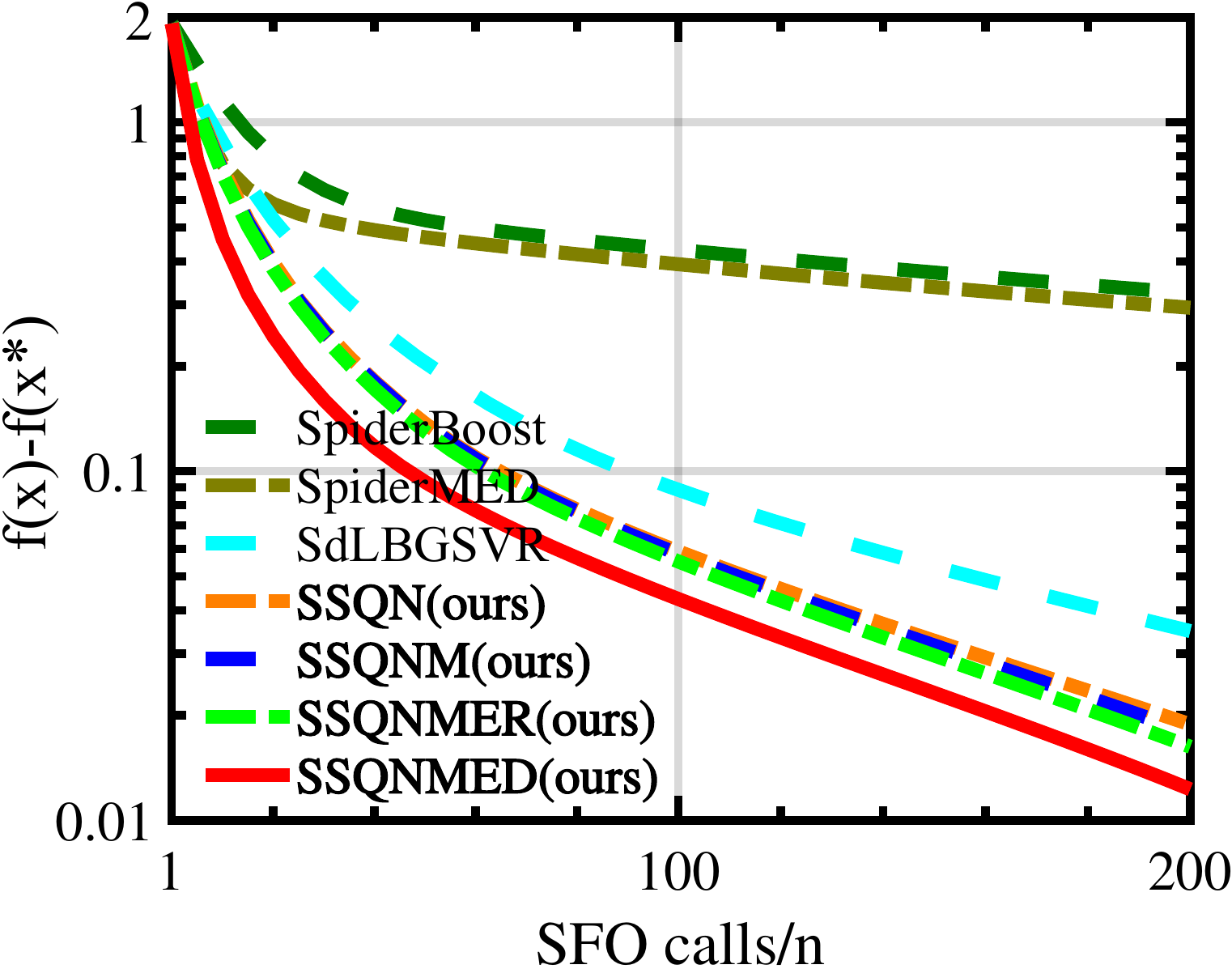}
		\caption{Data: a9a}
	\end{subfigure}
\begin{subfigure}{0.3\linewidth}
		\includegraphics[width=\linewidth]{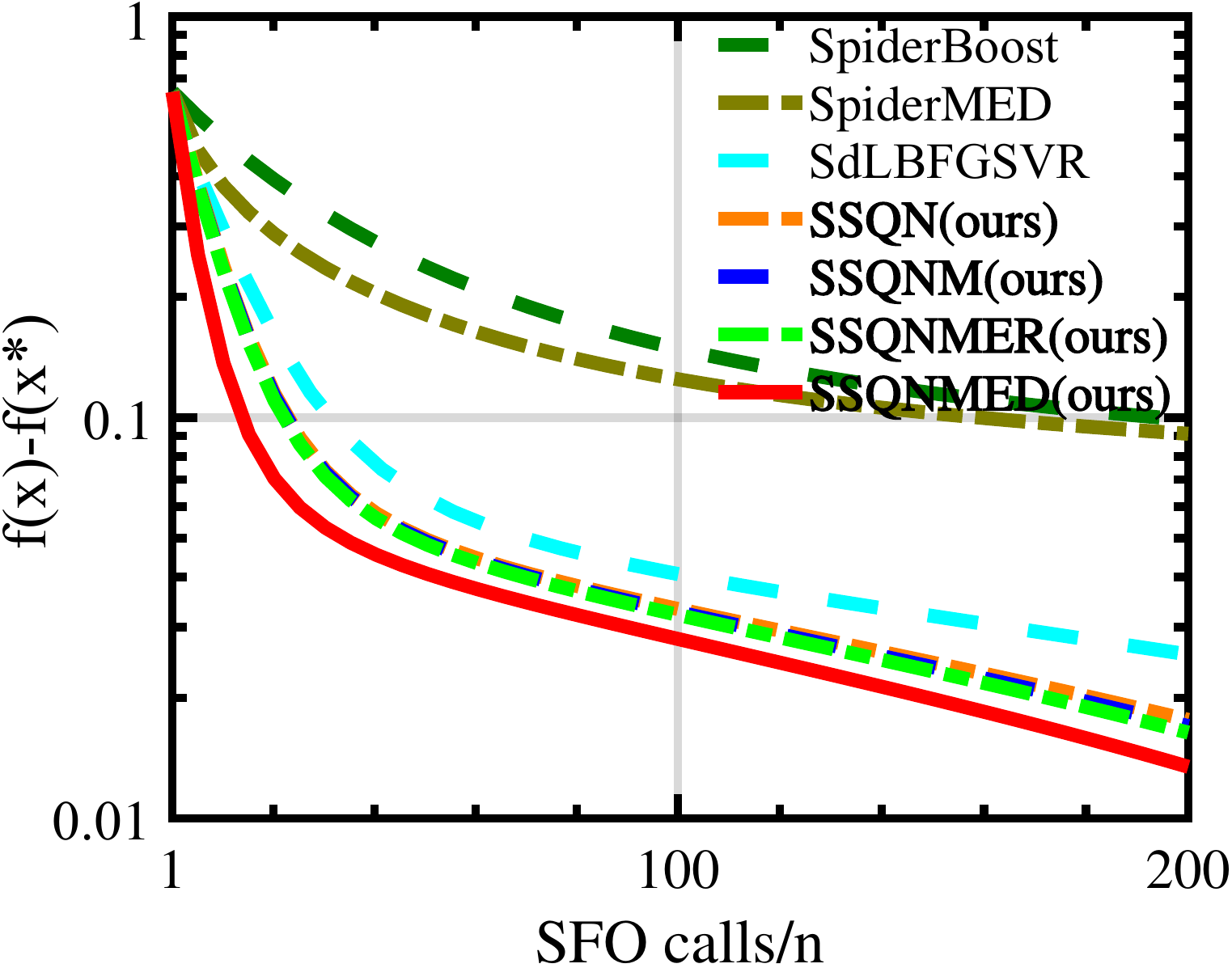}
		\caption{Data: w8a}
	\end{subfigure}
\begin{subfigure}{0.3\linewidth}
		\includegraphics[width=\linewidth]{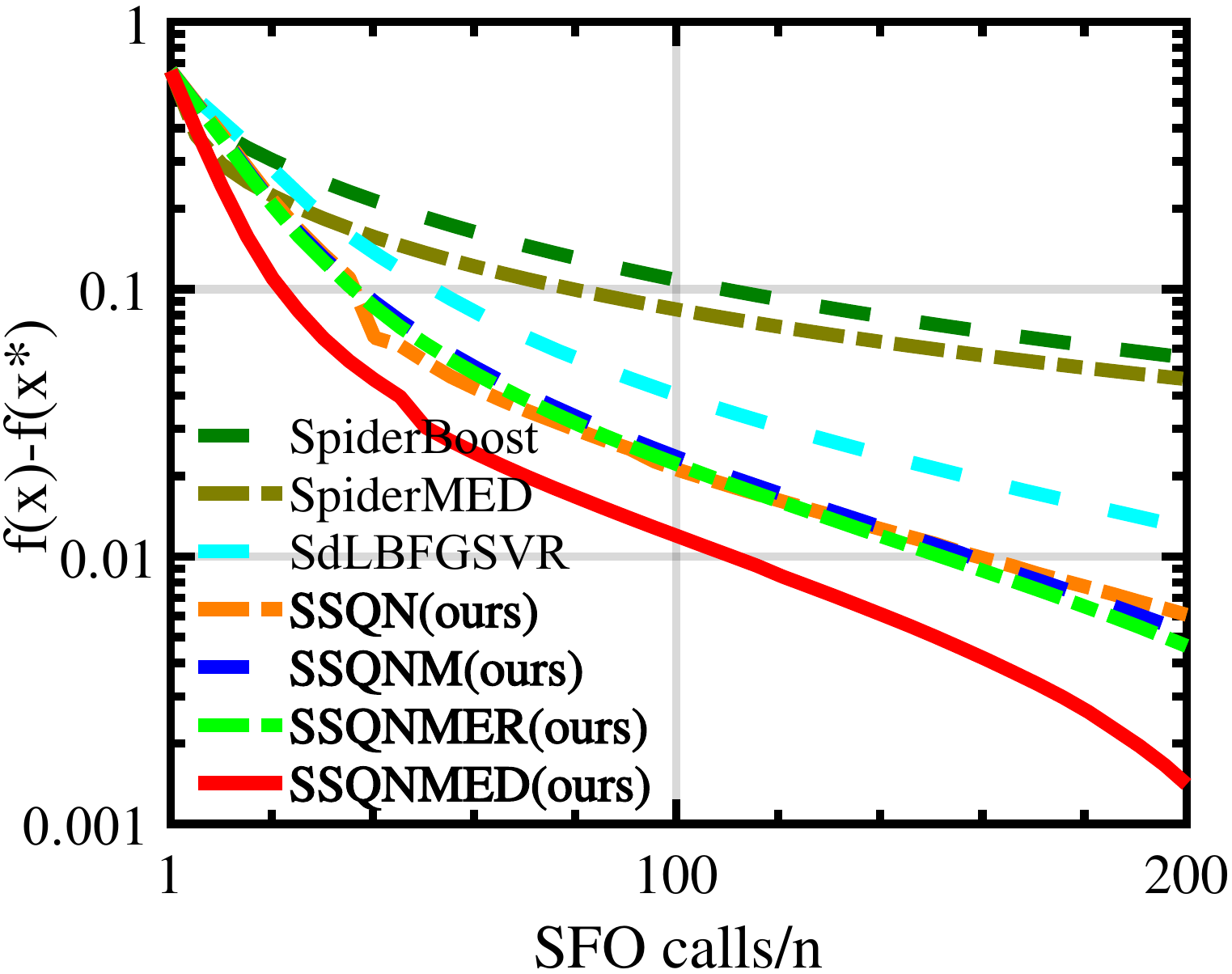}
		\caption{Data: mnist}
	\end{subfigure}%

	\begin{subfigure}{0.3\linewidth}
		\includegraphics[width=\linewidth]{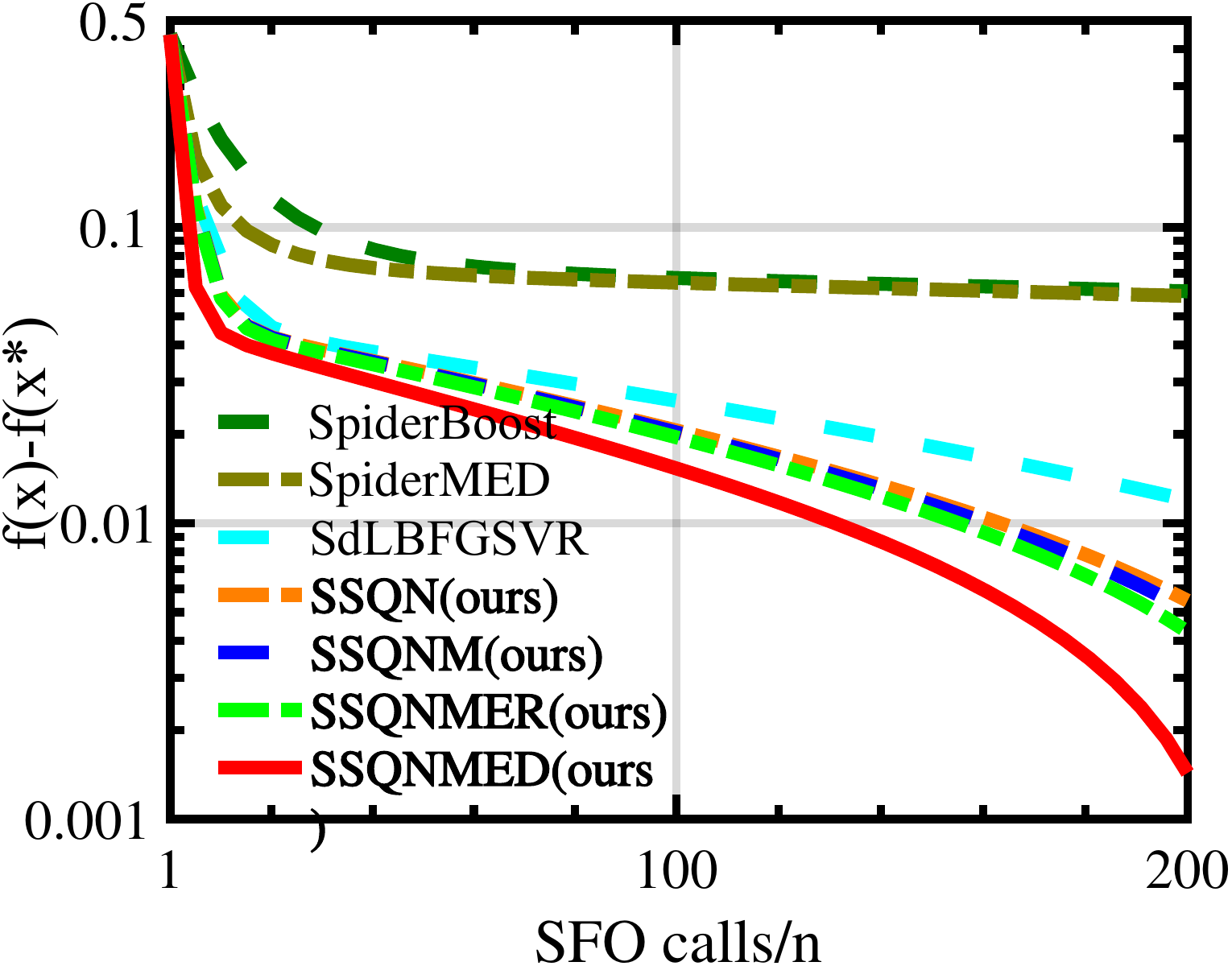}
		\caption{Data: ijcnn1}
	\end{subfigure}%
	\begin{subfigure}{0.3\linewidth}
		\includegraphics[width=\linewidth]{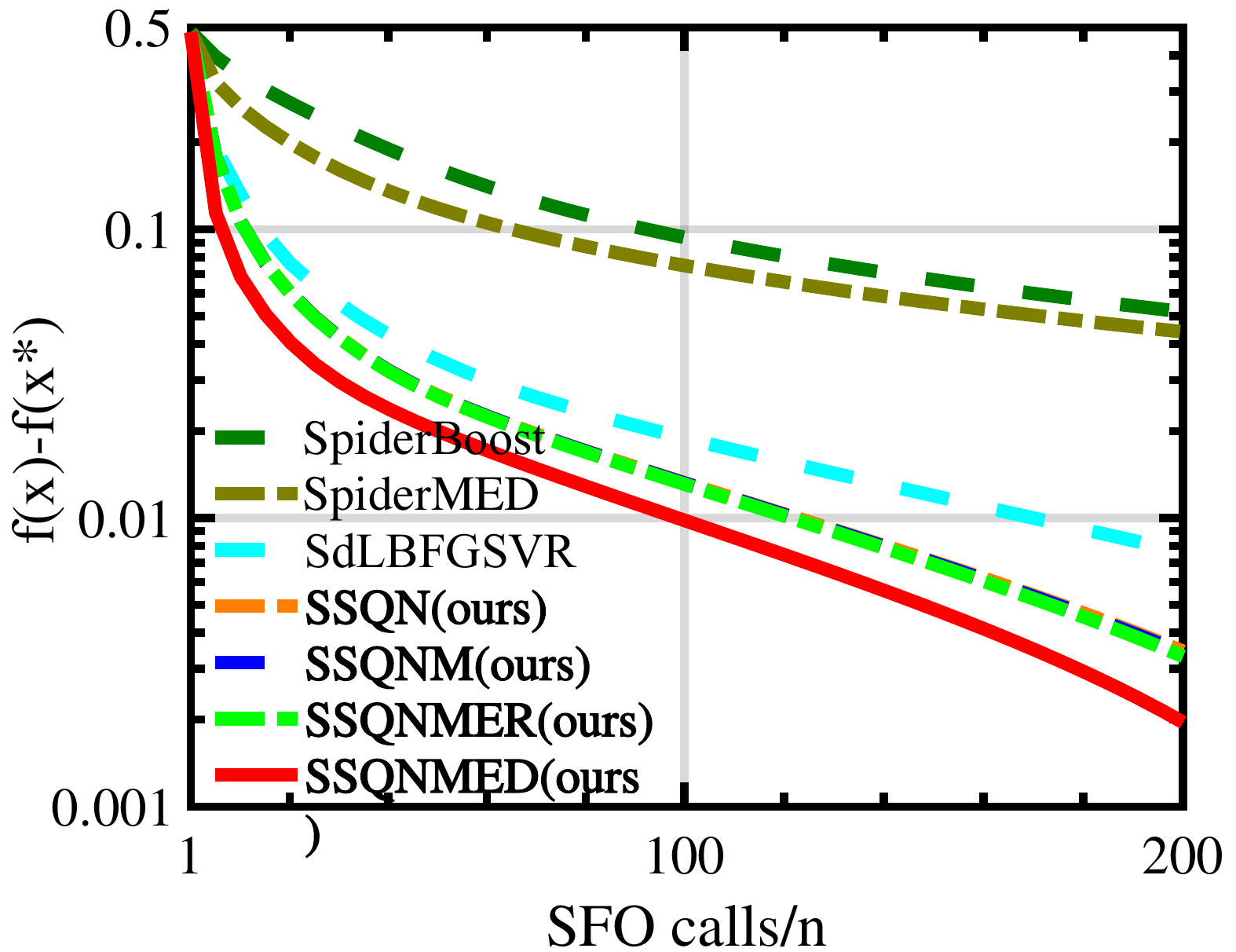}
		\caption{Data: covtype}
	\end{subfigure}%
	\begin{subfigure}{0.3\linewidth}
		\includegraphics[width=\linewidth]{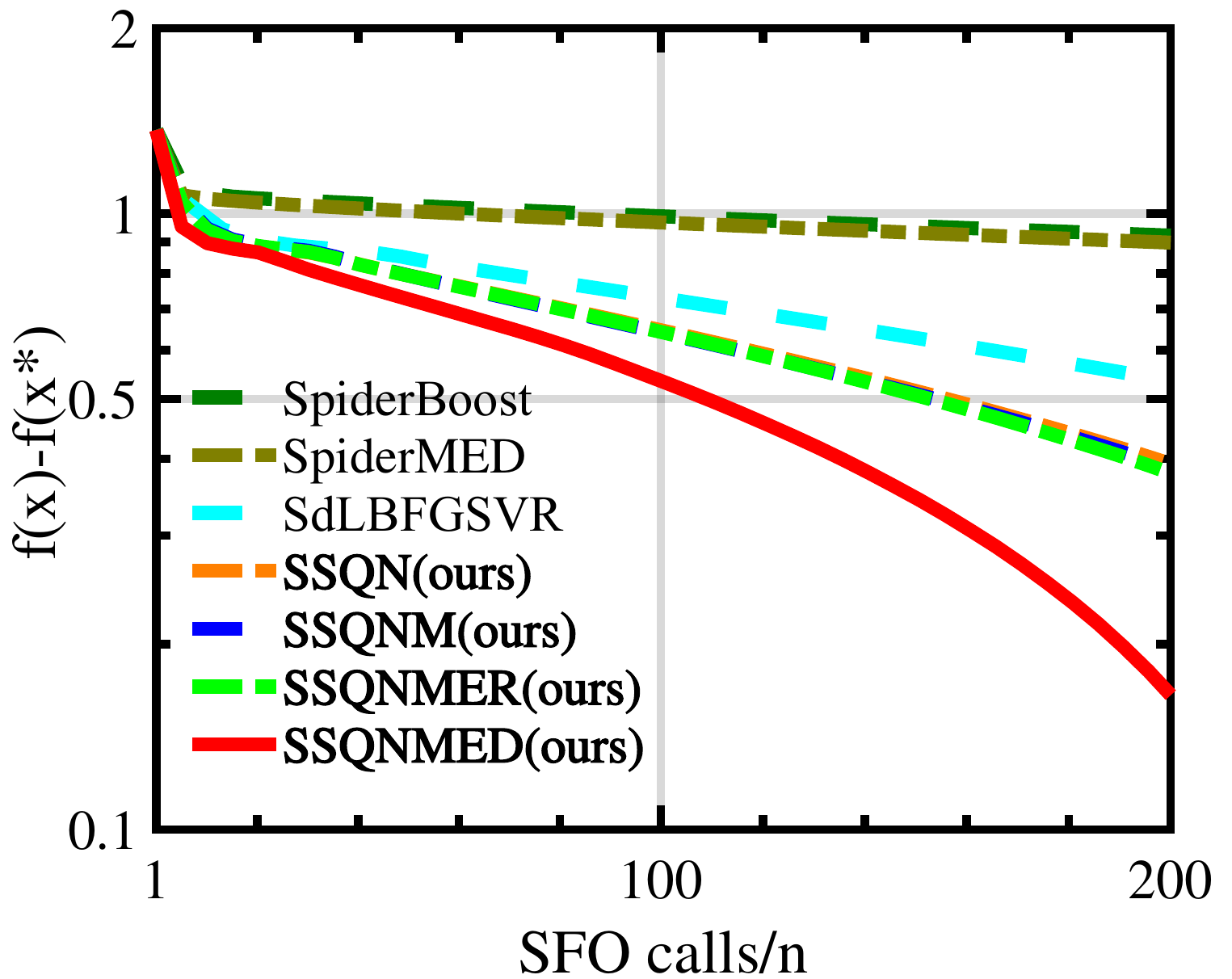}
		\caption{Data: synthetic data}
	\end{subfigure}%
	\caption{Comparison among algorithms for solving nonconvex logistic regression problems.}
\label{Experment3}
\end{figure*}

\subsection{Nonconvex Robust Linear Regression}
We consider comparing these algorithms for solving such a nonconvex robust linear regression problem:
\begin{align*}
\min_{x \in \mathbb{R}^d} f(x):=\frac{1}{n} \sum_{i=1}^{n}  \ell(b_i -\inner{x}{a_i}),
\end{align*}
where the nonconvex loss function is defined as $\ell (x) :=\log (\frac{x^2}{2} + 1)$. The experiment settings are same as those in the nonconvex SVM problem, except that the initial solutions in all cases are {drawn from} the standard norm distribution. The learning curves on the gap between $f(x)$ and $f(x^*)$ are reported in Fig.~\ref{Experment2}. As one can see from Fig.~\ref{Experment2}, the stochastic quasi-Newton methods still have a significantly better performance than the stochastic first-order methods. Also, the proposed four SSQN-type algorithms outperform the SdLBFGSVR with a considerably large margin. In most cases, SSQNMED outperforms SSQNM and SSQNMER by a large gap, except in the dataset mnist where SSQNMER and SSQNMED have similar performances and are both significantly better than that of SSQNM.
\subsection{Nonconvex Logistic Regression}
Comparisons are conducted among all algorithms for solving a nonconvex logistic regression problem:
\begin{align*}
\min_{x \in \mathbb{R}^d} f(x):=\frac{1}{n} \sum_{i=1}^{n}  \ell(b_i, \inner{x}{a_i}) + r \sum_{i=1}^{d} \frac{x_i^2}{1+x_i^2},
\end{align*}
where the loss function $\ell$ is set to be the cross-entropy loss. For this problem, the initial solutions to all algorithms on datasets w8a and a9a are {drawn from} the standard norm distribution, while experiments on datasets ijcnn1 and mnist take the original point. Other experiment settings are same as those of the nonconvex SVM problem. The learning curves on the gap between $f(x)$ and $f(x^* )$ are reported in Fig.~\ref{Experment3}. Obviously, the stochastic quasi-Newton methods outperform those stochastic first-order methods by a significantly large gap. Meanwhile, the proposed four SSQN-type of algorithms all have a better performance than the SdLBFGSVR. As for the four SSQN-type of algorithms, their performance is related to the momentum coefficient setting which means that algorithm with a larger momentum coefficient will converge faster. Moreover, in all cases the SSQNMED has the best performance among four SSQN-type algorithms, and SSQN has the worst.

\section{Conclusion}
In the paper, we presented the novel faster stochastic quasi-Newton (SpiderSQN) methods. Moreover, we proved that the SpiderSQN methods reach the best known SFO complexity of $\mathcal{O}(\min(n+n^{1/2}\epsilon^{-2},\epsilon^{-3}))$ for finding an $\epsilon$-approximated stationary point.  At the same time, we studied the lower bound of SFO complexity of the SpiderSQN methods. As presented in the theoretical results, our methods reach the near-optimal SFO complexity in solving the nonconvex problems. Moreover, we applied three different momentum schemes to SpiderSQN to further improve its practical performance.
%
\section*{Acknowledgment}
We thank the anonymous reviewers for their helpful comments. We also thank the IT Help Desk at University of Pittsburgh. Q.S. Zhang and C. Deng were supported in part by the National Natural Science Foundation of China under Grant 62071361, the National Key R\&D Program of China under Grant 2017YFE0104100, and the China Research Project under Grant 6141B07270429. F.H. Huang and H. Huang were in part supported by U.S. NSF IIS 1836945, IIS 1836938, IIS 1845666, IIS 1852606, IIS 1838627, IIS 1837956. No. 61806093.

\newpage
\onecolumn
\appendix
\section{Proof of Algorithm \ref{thm: ssqn}} \label{proof_of_ssqn}

Throughout the paper, let $n_k = \lceil k/q \rceil$ such that $(n_k-1)q \le k \le n_k q - 1$.
Note that this convergence analysis is mainly following \cite{fang2018spider}. We first present an auxiliary lemma from \cite{fang2018spider}.
\begin{lemma}[\cite{fang2018spider}, Lemma 1]\label{lemma: Zhang}
	Under Assumptions \ref{assum1} and \ref{assum2}, the SPIDER estimator satisfies for all $(n_k-1)q +1 \le k \le n_kq-1$,
	\begin{align}
	\mathbb{E} \|v_k - \nabla f(x_k)\|^2 \le \frac{L^2}{|\xi_k|} \mathbb{E}\|x_k - x_{k-1}\|^2 + \mathbb{E}\|v_{k-1} - \nabla f(x_{k-1})\|^2.
	\end{align}
\end{lemma}
Telescoping Algorithm \ref{lemma: Zhang} over $k$ from $(n_k-1)q +1$ to $k$, we obtain that
\begin{align}
\mathbb{E} \|v_k - \nabla f(x_k)\|^2 &\leq \sum_{i=(n_k-1)q}^{k-1}\frac{L^2}{|\xi_k|} \mathbb{E}\|x_{i+1} - x_{i}\|^2 + \mathbb{E}\|v_{(n_k-1)q} - \nabla f(x_{(n_k-1)q})\|^2 \nonumber \\
&\leq  \sum_{i=(n_k-1)q}^{k}\frac{L^2}{|\xi_k|} \mathbb{E}\|x_{i+1} - x_{i}\|^2 + \mathbb{E}\|v_{(n_k-1)q} - \nabla f(x_{(n_k-1)q})\|^2.  \label{eq: variance bound}
\end{align}
Note that the above inequality also holds for $k = (n_k-1)q$, which can be simply checked by
plugging $k = (n_k- 1)q$ into above inequality. As for finite-sum case, when $\textrm{mod}(k,q) =0$ there is $v_k = \nabla f(x_k)$ for all $k$ such that $\mathbb{E}\|v_{k} - \nabla f(x_{k})\|^2=0$, and then we obtain the following bound for finite-sum case
\begin{lemma}\label{lemma: dound}
	Under Assumptions \ref{assum1} and \ref{assum2}, the SPIDER estimator satisfies for all $k \in \mathbb{N}$,
\begin{align}
\mathbb{E} \|v_k - \nabla f(x_k)\|^2 \leq \sum_{i=(n_k-1)q}^{k-1}\frac{L^2}{|\xi_k|} \mathbb{E}\|x_{i+1} - x_{i}\|^2
\leq  \sum_{i=(n_k-1)q}^{k}\frac{L^2}{|\xi_k|} \mathbb{E}\|x_{i+1} - x_{i}\|^2  \label{eq: variance bound1}
\end{align}
\end{lemma}
Then, we return to the proof of Algorithm \ref{thm: ssqn}.
\begin{proof}
Consider any iteration $k$ of the algorithm. By smoothness of $f$, we obtain that
\begin{align}
f(x_k) &\overset{(i)}\le f(x_{k-1}) + \inner{\nabla f(x_{k-1})}{x_k - x_{k-1}} + \frac{L}{2}\|x_k - x_{k-1}\|^2 \nonumber \\
&= f(x_{k-1}) + \inner{\nabla f(x_{k-1})}{- \eta H_{k-1}v_{k-1}} + \frac{L\eta^2}{2}\|H_{k-1}v_{k-1}\|^2 \nonumber \\
&= f(x_{k-1}) - \eta \inner{\nabla f(x_{k-1})-v_{k-1}}{H_{k-1}v_{k-1}} -  \eta \inner{v_{k-1}}{H_{k-1}v_{k-1}} + \frac{L\eta^2}{2}\|H_{k-1}v_{k-1}\|^2 \nonumber\\
&\overset{(ii)}\le f(x_{k-1}) - \eta \inner{\nabla f(x_{k-1})-v_{k-1}}{H_{k-1}v_{k-1}} -  \eta \|v_{k-1}\|\|H_{k-1}v_{k-1}\| +\frac{L\eta^2}{2}\|H_{k-1}v_{k-1}\|^2,
\end{align}

where (i) uses the Lipschitz continuity of $\nabla f$ and (ii) follows from $\inner{a}{b}\leq\|a\|\|b\|$. Rearranging the above inequality yields that
\begin{align}
	f(x_k) &\le  f(x_{k-1}) - \eta(\|H_{k-1}\|-\frac{L\eta \|H_{k-1}\|^2}{2}) \|v_{k-1}\|^2 + \eta \|H_{k-1}\| \|\nabla f(x_{k-1})-v_{k-1}\| \|v_{k-1}\|
\nonumber \\
& \overset{(i)}\leq  f(x_{k-1}) - \eta(\|H_{k-1}\|-\frac{L\eta \|H_{k-1}\|^2}{2}) \|v_{k-1}\|^2
+ \frac{\eta \|H_{k-1}\|}{2} (\|\nabla f(x_{k-1})-v_{k-1}\|^2 + \|v_{k-1}\|^2)
\nonumber \\
& \overset{(ii)} \leq  f(x_{k-1}) - \eta(\frac{\sigma_{\mathrm{min}}}{2}-\frac{L\eta \sigma_{\mathrm{max}}^2}{2}) \|v_{k-1}\|^2
+ \frac{\eta \sigma_{\mathrm{max}}}{2}\|\nabla f(x_{k-1})-v_{k-1}\|^2
. \label{eq: 15}
\end{align}
where (i)  uses the inequality that $\inner{x}{y}\le \frac{\|x\|^2+\|y\|^2}{2}$ for $x,y\in \mathbb{R}^d$, (ii) follows from Assumption~\ref{assum3}. Taking expectation on both sides of the above inequality yields that
	\begin{align}
	\mathbb{E} &f(x_{k+1}) \nonumber \\
	 &\le \mathbb{E} f(x_k) + \frac{\eta \sigma_{\mathrm{max}}}{2} \mathbb{E}\|\nabla f(x_k)-v_k\|^2  - (\frac{\eta\sigma_{\mathrm{min}}}{2} - \frac{L\eta^2 \sigma_{\mathrm{max}}^2}{2}) \mathbb{E}\|v_k\|^2  \nonumber\\
	&\overset{(i)}{\le}  \mathbb{E} f(x_k) + \frac{\eta \sigma_{\mathrm{max}}}{2} \sum_{i=(n_k-1)q}^{k }\frac{L^2}{|\xi_k|} \mathbb{E}\|x_{i+1} - x_{i}\|^2  - (\frac{\eta\sigma_{\mathrm{min}}}{2} - \frac{L\eta^2 \sigma_{\mathrm{max}}^2}{2}) \mathbb{E}\|v_k\|^2
\nonumber\\
	&\overset{(ii)}= \mathbb{E} f(x_k) + \frac{\eta^3\sigma_{\mathrm{max}}^3}{2} \sum_{i=(n_k-1)q}^{k}\frac{L^2}{|\xi_k|} \mathbb{E}\|v_i\|^2 -  (\frac{\eta\sigma_{\mathrm{min}}}{2} - \frac{L\eta^2 \sigma_{\mathrm{max}}^2}{2}) \mathbb{E}\|v_k\|^2, \label{eq: 1}
	\end{align}
	where (i) follows from Eq.~(\ref{eq: variance bound1}), and (ii) follows from the facts that $x_{k+1} = x_k - \eta H_k v_k$ and Algorithm \ref{assum3}. Next, telescoping Eq.~(\ref{eq: 1}) over $k$ from $(n_k-1)q$ to $k$ where $k \le n_kq-1$ and noting that for $(n_k-1)q \le j \le  n_kq-1$, $n_j = n_k$ , we obtain
	\begin{align}
	\mathbb{E} &f(x_{k+1}) \nonumber \\
	 &\le \mathbb{E} f(x_{(n_k-1) q}) + \frac{\eta^3 \sigma_{\mathrm{max}}^3}{2} \sum_{j=(n_k - 1)q}^{k} \sum_{i=(n_k-1)q}^{j}\frac{L^2}{|\xi_k|} \mathbb{E}\|v_i\|^2  - (\frac{\eta \sigma_{\mathrm{min}}}{2} - \frac{L\eta^2 \sigma_{\mathrm{max}}^2}{2}) \sum_{j=(n_k - 1)q}^{k}\mathbb{E}\|v_j\|^2 \nonumber \\
	&\overset{(i)}{\le}   \mathbb{E} f(x_{(n_k-1) q}) + \frac{\eta^3 \sigma_{\mathrm{max}}^3}{2} \sum_{j=(n_k - 1)q}^{k} \sum_{i=(n_k-1)q}^{k}\frac{L^2}{|\xi_k|} \mathbb{E}\|v_i\|^2 - (\frac{\eta \sigma_{\mathrm{min}}}{2} - \frac{L\eta^2 \sigma_{\mathrm{max}}^2}{2}) \sum_{j=(n_k - 1)q}^{k}\mathbb{E}\|v_j\|^2 \nonumber \\
	&\overset{(ii)}{\le} \mathbb{E} f(x_{(n_k-1) q}) + \frac{\eta^3\sigma_{\mathrm{max}}^3L^2q}{2|\xi_k|} \sum_{i=(n_k - 1)q}^{k}   \mathbb{E}\|v_i\|^2 - (\frac{\eta\sigma_{\mathrm{min}} }{2} - \frac{L\eta^2\sigma_{\mathrm{max}}^2}{2}) \sum_{j=(n_k - 1)q}^{k}\mathbb{E}\|v_j\|^2 \nonumber \\
	&= \mathbb{E} f(x_{(n_k-1) q}) - \sum_{i=(n_k - 1)q}^{k} \left(\frac{\eta\sigma_{\mathrm{min}} }{2} - \frac{L\eta^2 \sigma_{\mathrm{max}}^2}{2}- \frac{\eta^3\sigma_{\mathrm{max}}^3L^2q }{2|\xi_k| }\right) \mathbb{E}\|v_i\|^2  \nonumber \\
	&\overset{(iii)}{=} \mathbb{E} f(x_{(n_k-1) q}) -  \sum_{i=(n_k - 1)q}^{k}  \beta^* \mathbb{E}\|v_i\|^2 , \label{eq: 2}
	\end{align}
	where (i)  extends the summation of the second term from $j$ to $k$, (ii) follows from the fact that $k  \leqslant n_kq-1$. Thus, we obtain
	\begin{align}
	\sum_{j=(n_k - 1)q}^{k} &\sum_{i=(n_k-1)q}^{k}\frac{L^2}{|\xi_k|} \mathbb{E}\|v_i\|^2 \le \frac{(k+q-n_kq +1 )L^2}{|\xi_k|} \sum_{i=(n_k-1)q}^{k}  \mathbb{E}\|v_i\|^2   \le \frac{qL^2}{|\xi_k|} \sum_{i=(n_k-1)q}^{k}  \mathbb{E}\|v_i\|^2,
	\end{align}
	and (iii) follows from $\beta^* = \frac{\eta\sigma_{\mathrm{min}} }{2} - \frac{L\eta^2 \sigma_{\mathrm{max}}^2}{2}- \frac{\eta^3\sigma_{\mathrm{max}}^3L^2q }{2|\xi_k| }$.
	
	We continue the proof by further driving
	\begin{align}
	\mathbb{E} f(x_{K}) - &\mathbb{E}f(x_0) \nonumber \\
	&= (\mathbb{E} f(x_{q}) - \mathbb{E}f(x_0)) + (\mathbb{E} f(x_{2q}) - \mathbb{E}f(x_q)) +\cdots + (	\mathbb{E} f(x_{K}) - \mathbb{E} f(x_{(n_k-1) q})) \nonumber \\
	&\overset{(i)}{\leq}    \sum_{i=0}^{q-1} \beta^* \mathbb{E}\|v_i\|^2 -    \sum_{i=q}^{2q-1}\beta^* \mathbb{E}\|v_i\|^2 - \cdots
-   \sum_{i=(n_K -1)q}^{K-1} \beta^* \mathbb{E}\|v_i\|^2 \nonumber \\
	&=   \sum_{i=0}^{K-1 }  \beta^* \mathbb{E}\|v_i\|^2 ,
	\end{align}
	where (i) follows from Eq.~(\ref{eq: 2}).
	Note that $\mathbb{E} f(x_{K}) \ge f^* \triangleq \inf_{x\in \mathbb{R}^d} f(x)$. Hence,  the above inequality implies that
	\begin{align}
	\sum_{i=0}^{K-1}  \beta^* \mathbb{E}\|v_i\|^2    \le  f(x_0) -  f^*. \label{total_sum}
	\end{align}
	We next bound $\mathbb{E}  \|\nabla f(x_\xi)\|^2$, where $\xi$ is selected uniformly at random from $\{0,\ldots,K-1\}$. Observe that
	\begin{align}
	\mathbb{E}  \|\nabla f(x_\xi)\|^2  = \mathbb{E}  \|\nabla f(x_\xi) - v_\xi + v_\xi\|^2 \leq 2\mathbb{E}  \|\nabla f(x_\xi) - v_\xi \|^2 + 2\mathbb{E}  \| v_\xi \|^2 \label{final_0}.
	\end{align}
	Next, we bound the two terms on the right hand side of the above inequality. First, note that
	\begin{align}
	\mathbb{E} \|v_\xi\|^2  = \frac{1}{K }\sum_{i=0}^{K-1} \mathbb{E} \|v_i\|^2 \le  \frac{f(x_0) -  f^*}{K\beta^* }, \label{eq: 5}
	\end{align}
	where the last inequality follows from Eq.~(\ref{total_sum}). On the other hand, note that
	\begin{align}
	\mathbb{E}  \|\nabla f(x_\xi)-v_\xi\|^2  &\overset{(i)}{\le}\mathbb{E}    \sum_{i=(n_\xi-1)q}^{\xi }\frac{L^2}{|\xi_k| } \mathbb{E}\|x_{i+1} - x_{i}\|^2 +\mathbb{E}   \sum_{i=(n_\xi-1)q}^{\xi }\frac{L^2\eta^2\sigma_{\mathrm{max}}^2}{ |\xi_k|   } \mathbb{E}\|v_i\|^2  \nonumber \\
	&\overset{(iii)}{\leq}  \mathbb{E}   \sum_{i=(n_\xi-1)q}^{\min\{(n_\xi)q   - 1, K-1\} }\frac{L^2\eta^2\sigma_{\mathrm{max}}^2}{ |\xi_k| } \mathbb{E}\|v_i\|^2  \overset{(iv)}{\leq}  \frac{q}{K } \sum_{i=0}^{K-1} \frac{L^2\eta^2\sigma_{\mathrm{max}}^2}{|\xi_k|}\mathbb{E}\|v_{i}\|^2 \nonumber \\
	&\overset{(v)}{\le}    \frac{ L^2 \eta^2\sigma_{\mathrm{max}}^2q}{K|\xi_k|\beta^*} \left( f(x_0) - f^*  \right), \label{final_2}
	\end{align}
	where (i) follows from Eqs.~(\ref{eq: variance bound}) and (\ref{eq: variance bound1}), (ii) follows from the fact that $x_{k+1} = x_k - \eta H_k v_k$ and Assumption \ref{assum3}, (iii) follows from the definition of $n_\xi$, which implies $\xi \leqslant \min \{(n_\xi) q-1,K-1  \} $, (iv) follows from the fact that the probability that $n_\xi = 1,2,\cdots,n_K $   is less than or equal to $q/(K)$, and (v) follows from Eq.~(\ref{eq: 5}).
	
	Substituting Eqs.~(\ref{eq: 5}) and (\ref{final_2}) into Eq.~(\ref{final_0}), we obtain
	\begin{align}
	\mathbb{E} \|\nabla f(x_\xi)\|^2  &\le  \frac{2\left(f(x_0) -  f^*\right)}{K\beta^* } +  \frac{2 L^2 \eta^2\sigma_{\mathrm{max}}^2q}{K|\xi_k|\beta^*} \left( f(x_0) - f^*  \right) \nonumber \\
	&= \frac{2}{K\beta^*} \left(1+   \frac{ L^2 \eta^2\sigma_{\mathrm{max}}^2q}{ |\xi_k| }  \right)\left( f(x_0) -  f^*  \right)
. \label{final1}
	\end{align}
\end{proof}

Next we set the parameters as

\begin{align}
	 S_1 = n , q =\sqrt{n} , \xi_k = \sqrt{n}, \text{ and }  \eta = \frac{c}{L\sigma_{\mathrm{max}}m} \ , \label{finite_sum_parameter_setting}
\end{align}
where $c = \sigma_{\mathrm{min}} / \sigma_{\mathrm{max}}\le1$, and $m = (1 + \sqrt {5} )/2$. Given the parameters setting of $S_1$, $q$, and $\xi_k$ the value of $m
$ is determined as follow
\begin{align}
	\beta^* &= \frac{\eta\sigma_{\mathrm{min}} }{2} - \frac{L\eta^2 \sigma_{\mathrm{max}}^2}{2}- \frac{\eta^3\sigma_{\mathrm{max}}^3L^2 }{2 } \nonumber \\
&\overset{}{=} \frac{1}{2L} ({L\eta\sigma_{\mathrm{min}} }- {L^2\eta^2 \sigma_{\mathrm{max}}^2}- {\eta^3\sigma_{\mathrm{max}}^3L^3 }) \nonumber \\
&\overset{(i)}{=} \frac{c^2}{2Lm^3} (m^2- m-c) \nonumber \\ \label{beta1}
\end{align}
where (i) follows from the definition of $\eta$ together with the problem independent parameter $c = \sigma_{\mathrm{min}} / \sigma_{\mathrm{max}}
\le1$. When $c = 1$ this reduces to the SpiderBoost algorithm with steosize $\eta$ scaled by $\sigma_{\mathrm{min}}$ (or $\sigma_{\mathrm{max}}$). Next, we should determine a suitable value of $m$ to ensure $\beta^*>0$ i.e.,
\begin{align}
	\beta^* &\overset{}{=} \frac{c^2}{2Lm^3} (m^2- m-c) >0 \nonumber \\\label{beta2}
\end{align}
it is sufficient to ensure $m^2- m-c> 0$. Thus, we obtain $m> (1 + \sqrt {1 + 4c} )/2$. In the Spider-SQN method there is $c<1$, and we can let $m = (1 + \sqrt {5} )/2$. Plugging $m = (1 + \sqrt {5} )/2$ into Eq.~(\ref{beta1}) we obtain
\begin{align}
	\beta^* = \frac{c^2}{2Lm^3} (1-c)>0. \label{finte_sum_cond}
\end{align}
therefore,  $m = (1 + \sqrt {5} )/2$ is reasonable and thus $\eta = \frac{(1+\sqrt{5})\sigma_{\mathrm{min}}}{2L\sigma_{\mathrm{max}}^2}$. Plugging Eqs.~(\ref{finite_sum_parameter_setting}) and (\ref{finte_sum_cond}) into Eq.~(\ref{final1}), we obtain that, after $K$ iterations, the output of SpiderBoost satisfies
\begin{align}
	\mathbb{E} \|\nabla f(x_\zeta)\|^2  \le \frac{2(1+\frac{c^2}{m^2})}{K \beta^*}  \left( f(x_0) -  f^*  \right)
\end{align}
To ensure $\mathbb{E} \|\nabla f(x_\zeta)\|   \leqslant \epsilon$, it is sufficient to ensure $\mathbb{E} \|\nabla f(x_\zeta)\|^2   \leqslant \epsilon^2$ (because $\left( \mathbb{E} \|\nabla f(x_\zeta)\|\right)^2\leq \mathbb{E} \|\nabla f(x_\xi)\|^2 $ due to Jensen's inequality). Thus, we need the total number $K$ of iterations  satisfies that  $ \frac{2(1+\frac{c^2}{m^2})}{K \beta^*}  \left( f(x_0) -  f^*  \right) \leq \epsilon^2$, which gives
\begin{align}
	 K =  \frac{2(1+\frac{c^2}{m^2})/\beta^*}{\epsilon^2 }  \left( f(x_0) -  f^*  \right). \label{finite_sum_number_of_iteration}
\end{align}
Then, the total SFO complexity is given by
\begin{align*}
	 \left\lceil  \frac{K}{q} \right\rceil \cdot  S_1 + K\cdot \xi_k \leqslant (K+q)\cdot \frac{S_1}{q} + K\cdot \xi_k =  K\sqrt{n} + n + K\sqrt{n}  = O(\sqrt{n} \epsilon^{-2} + n),
\end{align*}
where the last equation follows from Eq.~(\ref{finite_sum_number_of_iteration}), thus the SFO complexity of Algorithm \ref{SPIDER-SQN} is $O(\sqrt{n} \epsilon^{-2} + n)$.

\section{Proof of Algorithm \ref{thm: ssqnm}}
\subsection{Auxiliary Lemmas for Analysis of Algorithm \ref{SPIDER-SQN-M}}
Note that in algorithm utilizing momentum scheme the $\beta_k$ remains the same for all $k$, thus we use $\beta$ for notation brevity. First, we collect some auxiliary results that facilitate the analysis of Algorithm \ref{SPIDER-SQN-M}. For any $k\in \mathbb{N}$, denote $\tau(k) \in \mathbb{N}$ the unique integer such that $(\tau(k)-1)q \le k \le \tau(k) q - 1$. We also define $\Gamma_0 = 0, \Gamma_1 = 1$ and $\Gamma_k = (1-\alpha_k)\Gamma_{k-1}$ for $k=2,3,...$. Since we set $\alpha_k = \frac{2}{k+1}$, it is easy to check that $\Gamma_k = \frac{2}{k(k+1)}$. Note that this convergence analysis is mainly following \cite{zhou2019momentum}. Besides the auxiliary Algorithm \ref{lemma: Zhang} (\cite{fang2018spider}, lemma1), we prove the following auxiliary lemma.

\begin{lemma}\label{aux: 3}
	Let the sequences $\{x_k\}_k, \{y_k\}_k, \{z_k\}_k$ be generated by Algorithm \ref{SPIDER-SQN-M}. Then, the following inequalities hold
	\begin{align}
	y_k - x_k &= \Gamma_k\sum_{t=1}^k \frac{\lambda_{t-1} - \beta_{t-1}}{\Gamma_t} H_{t-1}v_{t-1},\\
	\|y_k - x_k\|^2 &\le \sigma_{\mathrm{max}}^2\Gamma_k\sum_{t=1}^k \frac{\lambda_{t-1} - \beta_{t-1}}{\alpha_t\Gamma_t} \|v_{t-1}\|^2,\\
	\|z_{k+1}-z_k\|^2 &\le 2\beta_{k}^2\sigma_{\mathrm{max}}^2\| H_kv_{k}\|^2 + 2\alpha_{k+2}^2\sigma_{\mathrm{max}}^2 \Gamma_{k+1}\sum_{t=1}^{k+1} \frac{(\lambda_{t-1} - \beta_{t-1})^2}{\alpha_t\Gamma_t} \|v_{t-1}\|^2.
	\end{align}
\end{lemma}
\begin{proof}
	We prove the first equality. By the update rule of the momentum scheme, we obtain that
	\begin{align}
	y_k - x_k &= z_{k-1} - \beta_{k-1} H_{k-1}v_{k-1} - (x_{k-1} - \lambda_{k-1} H_{k-1}v_{k-1}) \nonumber\\
	&= (1-\alpha_k) (y_{k-1} - x_{k-1}) + (\lambda_{k-1} - \beta_{k-1}) H_{k-1}v_{k-1}.
	\end{align}
	Dividing both sides by $\Gamma_k$ and noting that $\frac{1-\alpha_k}{\Gamma_k} = \Gamma_{k-1}$, we further obtain that
	\begin{align}
	\frac{y_k - x_k}{\Gamma_k} &= \frac{y_{k-1} - x_{k-1}}{\Gamma_{k-1}} + \frac{\lambda_{k-1} - \beta_{k-1}}{\Gamma_k} H_{k-1}v_{k-1}.
	\end{align}
	Telescoping the above equality over $k$ yields the first desired equality.
	
	Next, we prove the second inequality. Based on the first equality, we obtain that
	\begin{align}
	\|y_k - x_k \|^2 &= \|\Gamma_{k} \sum_{t=1}^{k} \frac{\lambda_{t-1} - \beta_{t-1}}{\Gamma_t} H_{t-1}v_{t-1}\|^2 \nonumber \\
	&= \|\Gamma_{k} \sum_{t=1}^{k} \frac{\alpha_t}{\Gamma_t} \frac{\lambda_{t-1} - \beta_{t-1}}{\alpha_t} H_{t-1}v_{t-1}\|^2 \nonumber \\
	&\overset{(i)}{\le} \Gamma_{k} \sum_{t=1}^{k} \frac{\alpha_t}{\Gamma_t} \frac{(\lambda_{t-1} - \beta_{t-1})^2}{\alpha_t^2}  \|H_{t-1}v_{t-1}\|^2 \nonumber \\
	&= \Gamma_{k} \sum_{t=1}^{k} \frac{(\lambda_{t-1} - \beta_{t-1})^2}{\Gamma_t\alpha_t}  \|H_{t-1}v_{t-1}\|^2 \\
&\overset{(ii)} {\le} \sigma_{\mathrm{max}}^2 \Gamma_{k} \sum_{t=1}^{k} \frac{(\lambda_{t-1} - \beta_{t-1})^2}{\Gamma_t\alpha_t}  \|v_{t-1}\|^2,
	\end{align}
	where (i) uses the facts that $\{\Gamma_k\}_k$ is a decreasing sequence, $\sum_{t=1}^{k} \frac{\alpha_t}{\Gamma_t} = \frac{1}{\Gamma_k}$ and Jensen's inequality, (ii) follows from the Algorithm \ref{assum3}.
	
	Finally, we prove the third inequality. By the update rule of the momentum scheme, we obtain that $z_{k+1} - z_{k} = y_{k+1} - z_k + \alpha_{k+2} (x_{k+1} - y_{k+1})$. Then, we further obtain that
	\begin{align}
	\|z_{k+1} - z_{k}\| &\le \|y_{k+1} - z_k\| + \alpha_{k+2} \|x_{k+1} - y_{k+1}\| \nonumber \\
	&\le \beta_{k} \|H_{k}v_{k}\| + \alpha_{k+2} \sqrt{\|x_{k+1} - y_{k+1}\|^2} \nonumber \\
	&\le \beta_{k} \|H_{k}v_{k}\| + \alpha_{k+2} \sqrt{\Gamma_{k+1} \sum_{t=1}^{k+1} \frac{(\lambda_{t-1} - \beta_{t-1})^2}{\Gamma_t\alpha_t}  \|H_{t-1}v_{t-1}\|^2} \nonumber \\
	\end{align}
	The desired result follows by taking the square on both sides of the above inequality and using the facts that $(a+b)^2 \le 2a^2+2b^2$ and $\|H_k\|$ is upper bounded by $\sigma_{\mathrm{max}}$.
\end{proof}

\subsection{Proof of Algorithm \ref{thm: ssqnm}}

Consider any iteration $k$ of the algorithm. By smoothness of $f$, we obtain that
\begin{align}
f(x_k) &\le f(x_{k-1}) + \inner{\nabla f(x_{k-1})}{x_k - x_{k-1}} + \frac{L}{2}\|x_k - x_{k-1}\|^2 \nonumber \\
&= f(x_{k-1}) + \inner{\nabla f(x_{k-1})}{- \lambda_{k-1} H_{k-1}v_{k-1}} + \frac{L\lambda_{k-1}^2}{2}\|H_{k-1}v_{k-1}\|^2 \nonumber \\
&= f(x_{k-1}) - \lambda_{k-1} \inner{\nabla f(x_{k-1})-v_{k-1}}{H_{k-1}v_{k-1}} -  \lambda_{k-1} \inner{v_{k-1}}{H_{k-1}v_{k-1}} + \frac{L\lambda_{k-1}^2}{2}\|H_{k-1}v_{k-1}\|^2 \nonumber\\
&\overset{(i)}\le f(x_{k-1}) - \lambda_{k-1} \inner{\nabla f(x_{k-1})-v_{k-1}}{H_{k-1}v_{k-1}} -  \lambda_{k-1} \|v_{k-1}\|\|H_{k-1}v_{k-1}\| +\frac{L\lambda_{k-1}^2}{2}\|H_{k-1}v_{k-1}\|^2,
\end{align}
where (i) follows from Cauchy-Swartz inequality. Rearranging the above inequality and using Cauchy-Swartz inequality yields that
\begin{align}
	f(x_k) \le  f(x_{k-1}) - \lambda_{k-1}(\|H_{k-1}\|-\frac{L\lambda_{k-1} \|H_{k-1}\|^2}{2}) \|v_{k-1}\|^2 + \lambda_{k-1} \|H_{k-1}\| \|\nabla f(x_{k-1})-v_{k-1}\| \|v_{k-1}\|. \label{eq: 15}
\end{align}

Note that
\begin{align}
\|\nabla f(x_{k-1}) - v_{k-1}\| &\le \|\nabla f(x_{k-1}) - \nabla f(z_{k-1})\| + \|\nabla f(z_{k-1}) - v_{k-1}\| \nonumber \\
&\overset{(i)}{\le} L\|x_{k-1} - z_{k-1}\| + \|\nabla f(z_{k-1}) - v_{k-1}\| \nonumber \\
&\overset{(ii)}{\le} L(1-\alpha_k) \|y_{k-1} - x_{k-1}\| + \|\nabla f(z_{k-1}) - v_{k-1}\| ,
\end{align}
where (i) uses the Lipschitz continuity of $\nabla f$ and (ii) follows from the update rule of the momentum scheme. Substituting the above inequality into Eq.~(\ref{eq: 15}) yields that
\begin{align}
f(x_k) &\le f(x_{k-1}) - \lambda_{k-1}(\|H_{k-1}\|-\frac{L\lambda_{k-1} \|H_{k-1}\|^2}{2}) \|v_{k-1}\|^2
+ L\lambda_{k-1}(1-\alpha_k) \|H_{k-1}\|\|v_{k-1}\|\|y_{k-1} - x_{k-1}\| \nonumber\\
&\quad+ \lambda_{k-1}\|H_{k-1}\| \|v_{k-1}\|\|\nabla f(z_{k-1}) - v_{k-1}\|
\nonumber \\
&\le f(x_{k-1}) - \lambda_{k-1}(\|H_{k-1}\|-\frac{L\lambda_{k-1} \|H_{k-1}\|^2}{2}) \|v_{k-1}\|^2 + \frac{L\lambda_{k-1}^2\|H_{k-1}\|^2}{2} \|v_{k-1}\|^2 + \frac{L(1-\alpha_k)^2}{2}\|y_{k-1} - x_{k-1}\|^2 \nonumber\\
&\qquad+ \frac{\lambda_{k-1}\|H_{k-1}\|}{2} \|v_{k-1}\|^2 + \frac{\lambda_{k-1}\|H_{k-1}\|}{2}\|\nabla f(z_{k-1}) - v_{k-1}\|^2
\nonumber \\
&\overset{(i)}\leq f(x_{k-1}) - \lambda_{k-1}(\frac{\sigma_{\mathrm{min}}}{2}-\frac{2L\lambda_{k-1} \sigma_{\mathrm{max}}^2}{2}) \|v_{k-1}\|^2 + \frac{L(1-\alpha_k)^2}{2}\|y_{k-1} - x_{k-1}\|^2 \nonumber \\
 &\quad+ \frac{\lambda_{k-1}\sigma_{\mathrm{max}}}{2}\|\nabla f(z_{k-1}) - v_{k-1}\|^2
 \nonumber\\
&\overset{(ii)} \le  f(x_{k-1}) - \lambda_{k-1}(\frac{\sigma_{\mathrm{min}}}{2}-\frac{2L\lambda_{k-1} \sigma_{\mathrm{max}}^2}{2}) \|v_{k-1}\|^2 + \frac{L\Gamma_{k-1}}{2}\sum_{t=1}^{k-1} \frac{\lambda_{t-1} - \beta_{t-1}}{\alpha_t\Gamma_t} \sigma_{\mathrm{max}}^2\| v_{t-1}\|^2 \nonumber\\
&\quad+ \frac{\lambda_{k-1}\sigma_{\mathrm{max}}}{2}\|\nabla f(z_{k-1}) - v_{k-1}\|^2,
\end{align}
where (i) follows from and the Assumption \ref{assum3}, (ii) uses item 2 of Algorithm \ref{aux: 3} and the fact that $0<\alpha_k <1$.
Telescoping the above inequality over $k$ from $1$ to $K$ yields that
\begin{align}
f(x_{K}) &\le f(x_{0}) - \sum_{k=0}^{K-1} \lambda_{k}(\frac{\sigma_{\mathrm{min}}}{2}-\frac{2L\lambda_{k} \sigma_{\mathrm{max}}^2}{2}) \|v_{k}\|^2 + \sum_{k=0}^{K-1} \frac{L\Gamma_{k}}{2}\sum_{t=1}^{k-1} \frac{\lambda_{t-1} - \beta_{t-1}}{\alpha_t\Gamma_t} \sigma_{\mathrm{max}}^2\| v_{t-1}\|^2 \nonumber \\
&\quad + \sum_{k=0}^{K-1} \frac{\lambda_k\sigma_{\mathrm{max}}}{2}\|\nabla f(z_{k}) - v_k\|^2
\nonumber \\
&= f(x_{0}) - \sum_{k=0}^{K-1} \lambda_{k}(\frac{\sigma_{\mathrm{min}}}{2}-\frac{2L\lambda_{k} \sigma_{\mathrm{max}}^2}{2}) \|v_{k}\|^2 + \frac{L \sigma_{\mathrm{max}}^2}{2} \sum_{k=0}^{K-1} \sum_{t=1}^{k-1} \frac{\lambda_{t-1} - \beta_{t-1}}{\alpha_t\Gamma_t} \| v_{t-1}\|^2(\sum_{t=k}^{K-1} \Gamma_{t})  \nonumber\\
&\quad+ \sum_{k=0}^{K-1} \frac{\lambda_k\sigma_{\mathrm{max}}}{2}\|\nabla f(z_{k}) - v_k\|^2, \label{eq15}
\end{align}
where we have exchanged the order of summation in the second equality. Furthermore, note that $\sum_{t=k}^{K-1} \Gamma_{t} = 2\sum_{t=k}^{K-1} \frac{1}{t} - \frac{1}{t+1} \le \frac{2}{k}$. Then, substituting this bound into the above inequality and taking expectation on both sides yield that
\begin{align}
\mathbb{E}[f(x_{K})] &\le f(x_{0}) - \sum_{k=0}^{K-1} \lambda_{k}(\frac{\sigma_{\mathrm{min}}}{2}-\frac{2L\lambda_{k} \sigma_{\mathrm{max}}^2}{2}) \mathbb{E}\| v_{k}\|^2 + \frac{L \sigma_{\mathrm{max}}^2}{2} \sum_{k=0}^{K-1} \frac{2(\lambda_k - \beta_k)^2}{k\Gamma_{k+1}\alpha_{k+1}} \mathbb{E}\| v_{k}\|^2  \nonumber\\
&\quad+ \sum_{k=0}^{K-1} \frac{\lambda_k\sigma_{\mathrm{max}}}{2} \mathbb{E}\|\nabla f(z_{k}) - v_k\|^2. \label{eq: 16}
\end{align}

Next, we bound the term $\mathbb{E}\|\nabla f(z_{k}) - v_k\|^2$ in the above inequality. By Algorithm \ref{aux: 3} we obtain that
\begin{align}
\mathbb{E}\|\nabla f(z_{k}) - v_k\|^2 &\le \sum_{i=(\tau(k)-1)q}^{k-1}\frac{L^2}{|\xi_i|} \mathbb{E}\|z_{i+1} - z_{i}\|^2 \nonumber \\
&\le \sum_{i=(\tau(k)-1)q}^{k-1}\frac{L^2 \sigma_{\mathrm{max}}^2}{|\xi_i|} \big[2\beta_{i}^2\| v_{i}\|^2 + 2\alpha_{i+2}^2 \Gamma_{i+1}\sum_{t=0}^{i} \frac{(\lambda_{t} - \beta_{t})^2}{\alpha_t\Gamma_t} \|v_{t}\|^2], \label{eq: 17}
\end{align}
where the last inequality uses item 3 of Algorithm \ref{aux: 3}. Substituting Eq.~(\ref{eq: 17}) into Eq.~(\ref{eq: 16}) and simplifying yield that
\begin{align}
\mathbb{E}[f(x_{K})] &\le f(x_{0}) - \sum_{k=0}^{K-1} \Big[\lambda_{k}(\frac{\sigma_{\mathrm{min}}}{2}-\frac{2L\lambda_{k} \sigma_{\mathrm{max}}^2}{2})  - \frac{L \sigma_{\mathrm{max}}^2(\lambda_k - \beta_k)^2}{k\Gamma_{k+1}\alpha_{k+1}}\Big] \mathbb{E}\|v_{k}\|^2 \nonumber\\
&\quad + \underbrace{\sum_{k=0}^{K-1} \frac{\lambda_k \sigma_{\mathrm{max}}^3}{2} \mathbb{E}\bigg[\sum_{i=(\tau(k)-1)q}^{k-1}\frac{L^2}{|\xi_i|} \bigg[2\beta_i^2\| v_{i}\|^2 + 2\alpha_{i+2}^2 \Gamma_{i+1}\sum_{t=0}^i \frac{(\lambda_t - \beta_t)^2}{\alpha_{t+1}\Gamma_{t+1}} \|v_{t}\|^2 \bigg] \bigg]}_{T}. \label{eq: 18}
\end{align}
Before we proceed the proof, we first specify the choices of all the parameters. Specifically, we choose a constant mini-batch size $|\xi_k|\equiv |\xi|$, a constant $q=|\xi|$, a constant $\beta_k \equiv \beta>0$, $\lambda_k \in [\beta, (1+\alpha_{k+1})\beta]$. Based on these parameter settings, the term $T$ in the above inequality can be bounded as follows.
\begin{align}
T &\overset{(i)}{\le} \sum_{k=0}^{K-1} \frac{\lambda_k\sigma_{\mathrm{max}}^3}{2} \mathbb{E}\bigg[\sum_{i=(\tau(k)-1)q}^{\tau(k)q-1}\frac{L^2}{|\xi_i|} \bigg[2\beta_i^2\|v_{i}\|^2 + 2\alpha_{i+2}^2 \Gamma_{i+1}\sum_{t=0}^{k-1} \frac{(\lambda_t - \beta_t)^2}{\alpha_{t+1}\Gamma_{t+1}} \| v_{t}\|^2 \bigg] \bigg]
\nonumber\\
&\overset{(ii)}{\le} \sum_{k=0}^{K-1} \frac{\lambda_k L^2q\beta^2\sigma_{\mathrm{max}}^3}{|\xi|}\mathbb{E}\|v_{k}\|^2 + \sum_{k=0}^{K-1} \frac{2\lambda_kL^2\sigma_{\mathrm{max}}^3}{|\xi|[(\tau(k)-1)q+1]^3} \sum_{t=0}^{k-1}\frac{(\lambda_t - \beta_t)^2}{\alpha_{t+1}\Gamma_{t+1}}\mathbb{E}\|v_{t}\|^2
\nonumber\\
&\overset{(iii)}{\le} \sum_{k=0}^{K-1} \lambda_k L^2\beta^2\sigma_{\mathrm{max}}^3 \mathbb{E}\| v_{k}\|^2 + \frac{2L^2\beta^2\sigma_{\mathrm{max}}^3 }{|\xi|} \sum_{k=0}^{K-1}  \frac{\alpha_{k+1}}{\Gamma_{k+1}} \mathbb{E}\| v_{k}\|^2 (\sum_{t=k}^{K-1} \frac{\lambda_k}{[(\tau(t)-1)q+1]^3})
\nonumber\\
&\overset{(iv)}{\le} \sum_{k=0}^{K-1} \lambda_k L^2\beta^2\sigma_{\mathrm{max}}^3 \mathbb{E}\| v_{k}\|^2 + \frac{4L^2\beta^3\sigma_{\mathrm{max}}^3}{|\xi|} \sum_{k=0}^{K-1} (k+1) \mathbb{E}\| v_{k}\|^2 (\sum_{t=(\tau(k)-1)q}^{\tau(K)q} \frac{1}{[(\tau(t)-1)q+1]^3})
\nonumber\\
&= \sum_{k=0}^{K-1} \lambda_k L^2\beta^2\sigma_{\mathrm{max}}^3 \mathbb{E}\|G v_{k}\|^2 + \frac{4L^2\beta^3\sigma_{\mathrm{max}}^3}{|\xi|} \sum_{k=0}^{K-1} (k+1) \mathbb{E}\| v_{k}\|^2 (\sum_{t=\tau(k)-1}^{\tau(K)} \frac{q}{(tq+1)^3})
\nonumber\\
&\le \sum_{k=0}^{K-1} \lambda_k L^2\beta^2\sigma_{\mathrm{max}}^3 \mathbb{E}\|v_{k}\|^2 + \frac{2L^2\beta^3\sigma_{\mathrm{max}}^3}{q} \sum_{k=0}^{K-1} (k+1) \mathbb{E}\| v_{k}\|^2 \frac{1}{[(\tau(k)-1)q+1]^2}
\nonumber \\
&\overset{(v)}{\le} \sum_{k=0}^{K-1} \lambda_k L^2\beta^2 \sigma_{\mathrm{max}}^3 \mathbb{E}\| v_{k}\|^2 + 2L^2\beta^3\sigma_{\mathrm{max}}^3 \sum_{k=0}^{K-1} \mathbb{E}\|v_{k}\|^2 \frac{\tau(k)}{[(\tau(k)-1)q+1]^2} \nonumber\\
&\le \sum_{k=0}^{K-1} \lambda_k L^2\beta^2\sigma_{\mathrm{max}}^3 \mathbb{E}\|v_{k}\|^2 + 2L^2\beta^3\sigma_{\mathrm{max}}^3 \sum_{k=0}^{K-1} \mathbb{E}\| v_{k}\|^2, \label{eq: 19}
\end{align}
where (i) follows from the facts that $i\le k-1$ and $k-1\le \tau(k)q-1$, (ii) uses the fact that $\sum_{i=(\tau(k)-1)q}^{\tau(k)q-1} \alpha_{i+2}^2\Gamma_{i+1} \le \frac{2}{(\tau(k)-1)q+1)^3}$, (iii) uses the parameter settings $q=|\xi|$ and $\lambda_t-\beta_t \le \alpha_t\beta$, (iv) uses the facts that $\lambda_k\le 2\beta$ and $(\tau(k)-1)q \le k\le \tau(k)q$ and (v) uses the fact that $k \le \tau(k)q-1$. Substituting the above inequality into Eq.~(\ref{eq: 18}) and simplifying, we obtain that
\begin{align}
\mathbb{E}[f(x_{K})] &\le f(x_{0}) - \sum_{k=0}^{K-1} \Big[\lambda_k (\frac{\sigma_{\mathrm{min}}}{2}-\frac{2L\lambda_{k} \sigma_{\mathrm{max}}^2}{2} - L^2\beta^2\sigma_{\mathrm{max}}^3) - \frac{L(\lambda_k - \beta_k)^2\sigma_{\mathrm{max}}^2}{k\Gamma_{k+1}\alpha_{k+1}} -2L^2\beta^3 \sigma_{\mathrm{max}}^3\Big] \mathbb{E}\| v_{k}\|^2 \label{eq: 31} \\
&\le f(x_{0}) - \sum_{k=0}^{K-1} \Big[\beta (\frac{\sigma_{\mathrm{min}}}{2}-{2L\beta  \sigma_{\mathrm{max}}^2} - L^2\beta^2\sigma_{\mathrm{max}}^3) - L\beta^2\sigma_{\mathrm{max}}^2 -2L^2\beta^3\sigma_{\mathrm{max}}^3 \Big] \mathbb{E}\|v_{k}\|^2
\nonumber \\
&= f(x_{0}) - \sum_{k=0}^{K-1} \Big[\beta (\frac{\sigma_{\mathrm{min}}}{2}-{3L\beta  \sigma_{\mathrm{max}}^2} - 3L^2\beta^2\sigma_{\mathrm{max}}^3) \Big] \mathbb{E}\|v_{k}\|^2.
 \label{eq: 32}
\end{align}
Let $\beta^* = \beta (\frac{\sigma_{\mathrm{min}}}{2}-3L\beta\sigma_{\mathrm{max}}^2 - 3L^2\beta^2\sigma_{\mathrm{max}}^3)$. Following the analysis of Eq.~(\ref{beta1}), we choose $\beta = \frac{c}{L(3+\sqrt{15})\sigma_{\mathrm{max}}}$, where $c={\sigma_{\mathrm{min}} }/{\sigma_{\mathrm{max}}}< 1$ and then there is
\begin{align}
  \beta^* &= \frac{3c^2}{Lm^3}(1-c)\overset{(i)}> 0
\end{align}
where $m = 3+\sqrt{15}$ and (i) follows the definition of $\beta$. the above inequality further implies that
\begin{align}
\mathbb{E} [f(x_{K})] &\le f(x_{0}) - \sum_{k=0}^{K-1} \beta^*\mathbb{E}\| v_{k}\|^2 . \label{eq: 27}
\end{align}
Then, it follows that $\frac{1}{K}\sum_{k=0}^{K-1}\mathbb{E}\|v_{k}\|^2 \le (f(x_0) - f^*)/(K\beta^*)$. Next, we bound the term $\mathbb{E}  \| \nabla f(z_{\zeta})\|^2$, where $\zeta$ is selected uniformly at random from $\{0,\ldots,K-1\}$. Observe that
\begin{align}
\mathbb{E}  \| \nabla f(z_{\zeta})\|^2  = \mathbb{E}  \| \nabla f(z_{\zeta})-v_{\zeta} + v_{\zeta}\|^2  \overset{(i)}{\le} 2\mathbb{E}  \|\nabla f(z_{\zeta})-v_{\zeta}\|^2 + 2\mathbb{E}  \| v_{\zeta}\|^2,  \label{eq: 29}
\end{align}
where (i) uses the fact $(a+b)^2 \le 2a^2 +2b^2$. Next, we bound the two terms on the right hand side of the above inequality separately. First, note that
\begin{align}
  \mathbb{E} \|v_\zeta\|^2  = \frac{1}{K }\sum_{k=0}^{K-1} \mathbb{E} \|
v_k\|^2 \le \frac{(f(x_0) - f^*)}{K\beta^*}.
\end{align}

Second, note that Eq.~(\ref{eq: 17}) implies that
\begin{align}
\mathbb{E}  \|\nabla f(z_\zeta)-v_\zeta\|^2  &\le \mathbb{E}\sum_{i=(\tau(\zeta)-1)q}^{\zeta-1}\frac{L^2\sigma_{\mathrm{max}}^2}{|\xi_i|} \big[2\beta_i^2\|v_i\|^2 + 2\alpha_{i+2}^2 \Gamma_{i+1}\sum_{t=0}^{i} \frac{(\lambda_{t} - \beta_{t})^2}{\alpha_{t+1}\Gamma_{t+1}} \| v_t\|^2 \big] \nonumber\\
&\le \frac{2L^2\beta^2\sigma_{\mathrm{max}}^2}{|\xi|} \mathbb{E} \bigg(\sum_{i=(\tau(\zeta)-1)q}^{\tau(\zeta)q-1} \|
 v_i\|^2\bigg) + \frac{L^2\sigma_{\mathrm{max}}^2}{|\xi|} \mathbb{E} \bigg(\sum_{i=(\tau(\zeta)-1)q}^{\zeta-1} 2\alpha_{i+2}^2 \Gamma_{i+1}\sum_{t=0}^{i} \frac{(\lambda_{t} - \beta_{t})^2}{\alpha_{t+1}\Gamma_{t+1}} \| v_t\|^2 \bigg) \nonumber\\
&\le \frac{2L^2\beta^2\sigma_{\mathrm{max}}^2}{|\xi|} \frac{1}{K} \sum_{\zeta=0}^{K-1} \bigg(\sum_{i=(\tau(\zeta)-1)q}^{\tau(\zeta)q-1} \mathbb{E}\| v_i\|^2\bigg) \nonumber\\
&\quad+ \frac{L^2 \beta^2\sigma_{\mathrm{max}}^2}{|\xi|} \frac{1}{K} \sum_{\zeta=0}^{K-1} \bigg(\sum_{i=(\tau(\zeta)-1)q}^{\tau(\zeta)q-1} 2\alpha_{i+2}^2 \Gamma_{i+1}\sum_{t=0}^{\zeta-1} (t+1) \mathbb{E}\|v_t\|^2 \bigg) \nonumber\\
&\le \frac{2L^2\beta^2\sigma_{\mathrm{max}}^2 q}{|\xi|} \frac{1}{K} \sum_{\zeta=0}^{K-1} \mathbb{E}\|v_\zeta\|^2 + \frac{L^2 \beta^2\sigma_{\mathrm{max}}^2}{|\xi|} \frac{1}{K} \sum_{\zeta=0}^{K-1} \bigg(\frac{4}{[(\tau(\zeta)-1)q+1]^3}\sum_{t=0}^{\zeta-1} (t+1) \mathbb{E}\|v_t\|^2 \bigg) \nonumber\\
&\le 2L^2\beta^2\sigma_{\mathrm{max}}^2 \bigg(\frac{1}{K} \sum_{\zeta=0}^{K-1} \mathbb{E}\|v_\zeta\|^2\bigg) + \frac{L^2 \beta^2\sigma_{\mathrm{max}}^2}{|\xi|} \frac{1}{K} \sum_{\zeta=0}^{K-1} (\zeta+1) \mathbb{E}\|v_\zeta\|^2 \sum_{t=\zeta}^{K-1} \frac{4}{[(\tau(t)-1)q+1]^3} \nonumber\\
&\le 2L^2\beta^2\sigma_{\mathrm{max}}^2 \bigg(\frac{1}{K} \sum_{\zeta=0}^{K-1} \mathbb{E}\|v_\zeta\|^2\bigg) + L^2 \beta^2\sigma_{\mathrm{max}}^2 \frac{1}{K} \sum_{\zeta=0}^{K-1} \mathbb{E}\|v_\zeta\|^2 \frac{2\tau(\zeta)}{[(\tau(\zeta)-1)q+1]^2} \nonumber\\
&\le 3L^2\beta^2\sigma_{\mathrm{max}}^2 \frac{(f(x_0) - f^*)}{K\beta^*}, \label{eq: proof4}
\end{align}
where we have used the fact that $\zeta$ is sampled uniformly from $0,...,K-1$ at random.

Combining the above three inequalities we have
\begin{align}
\mathbb{E}  \| \nabla f(z_{\zeta})\|^2  &= \mathbb{E}  \| \nabla f(z_{\zeta})-v_{\zeta} + v_{\zeta}\|^2 \nonumber\\
&\overset{(i)}{\le} 2\mathbb{E}  \|\nabla f(z_{\zeta})-v_{\zeta}\|^2 + 2\mathbb{E}  \| v_{\zeta}\|^2 \nonumber\\
&\le \frac{(6L^2\beta^2\sigma_{\mathrm{max}}^2 + 2)}{K\beta^*} {(f(x_0) - f^*)}. \nonumber
\\ \label{eq: final2}
\end{align}
To ensure $\mathbb{E}  \| \nabla f(z_{\zeta})\| \le \epsilon $, it is sufficient to ensure $\mathbb{E}  \| \nabla f(z_{\zeta})\|^2 \le \epsilon^2$ ( since  $(\mathbb{E}  \| \nabla f(z_{\zeta})\|)^2 \le \mathbb{E}  \| \nabla f(z_{\zeta})\|^2 $, due to Jensen's inequality.) Therefore, we need the total number $K$ of iterations satisfies that and note that $\frac{(6L^2\beta^2\sigma_{\mathrm{max}}^2 + 2)}{K\beta^*} {(f(x_0) - f^*)} \le \epsilon^2$, which gives
\begin{align}
K= \frac{(6L^2\beta^2\sigma_{\mathrm{max}}^2 + 2)}{\beta^*} \frac{(f(x_0) - f^*)}{\epsilon^2}  . \label{eq: 20}
\end{align}
And then, the total SFO complexity is given by
\begin{align*}
(K+q)\frac{n}{q} + K|\xi| \le O(n+\sqrt{n}\epsilon^{-2}).
\end{align*}
Thus the SFO complexity of the Algorithm \ref{SPIDER-SQN-M} is $O(n+\sqrt{n}\epsilon^{-2})$ corresponding to Algorithm \ref{thm: ssqnm}.

\section{Proof of Algorithm \ref{thm: ssqnmer}}
The convergence proof of Algorithm \ref{thm: ssqnmer}, including both SpiderSQNMER and SpiderSQNMED , follows from that of Algorithm \ref{thm: ssqnm}, and therefore we only describe the key steps to adapt the proof.

We first prove the result of SpiderSQNMED. Under the epochwise-diminishing momentum scheme, the momentum coefficient is set to be $\alpha_{k} = \frac{2}{\ceil[]{k/q}+1}$. Consequently, we have $\Gamma_{k} = \frac{2}{\ceil[]{k/q}(\ceil[]{k/q}+1)}$. First, one can check that Eq.~(\ref{eq15}) still holds, and now we have $\sum_{t=k}^{K-1} \Gamma_{t} \le \frac{2}{\ceil[]{k/q}}$. Then, we follow the steps that bound the accumulation error term $T$ in Eq.~(\ref{eq: 18}). In the derivation of (ii), we now have that $\sum_{i=(\tau(k)-1)q}^{\tau(k)q-1} \alpha_{i+2}^2\Gamma_{i+1} \le \frac{2}{\tau(k)^3}$. Substituting this new bound into (ii) and noting that in (iii) we now have $\frac{\alpha_{k+1}}{\Gamma_{k+1}} = (\ceil[]{k/q}+1)$, one can follow the subsequent steps and show that the upper bound for T in Eq.~(\ref{eq: 19}) still holds. Moreover, in Eq.~(\ref{eq: 31}) we should replace $\frac{L(\lambda_k - \beta_k)^2}{k\Gamma_{k+1}\alpha_{k+1}}$ with $\frac{L(\lambda_k - \beta_k)^2}{\ceil[]{k/q}\Gamma_{k+1}\alpha_{k+1}}$, and consequently Eq.~(\ref{eq: 32}) is still valid. Then, one can follow the same analysis and show that Eq.~(\ref{eq: 27}) is still valid.
In summary, given the same parameters as for SpiderSQNM the convergence rate and the corresponding oracle complexity of SpiderSQNMED remain in the same order as SpiderSQNM, that is, $O(n+\sqrt{n}\epsilon^{-2})$ given the parameters as Algorithm \ref{thm: ssqnmer}.

The convergence proof of SpiderSQNMER follows from that of SpiderSQNM. The core idea is to apply the result of SpiderSQNM to each restart period.
Specifically, consider the iterations $k=0,1,...,q-2$. Firstly, we can rewrite Eq.~(\ref{eq: proof4}) as
\begin{align}
\mathbb{E}  \|\nabla f(z_\zeta)-v_\zeta\|^2
&\le 3L^2\beta^2\sigma_{\mathrm{max}}^2 \frac{(f(x_0) - f^*)}{K\beta^*}, \nonumber \\
&=O(\frac{(f(x_0) - f^*)}{K}).\label{eq: proof41}
\end{align}
As no restart is performed within these iterations, we can apply the result in Eq.~(\ref{eq: proof41}) (note that $f^*$ is the relaxation of $f(x_K)$) obtained from the analysis of Algorithm \ref{thm: ssqnm} and conclude that
\begin{align}
\mathbb{E}  \|\nabla f(z_{\zeta})\|^2 \le O\bigg(\frac{(f(x_0) - \mathbb{E}[f(x_{q-1})])}{q-1} \bigg), ~\text{where}~ \zeta \overset{\text{Unif}}{\sim}\{0,...,q-2\}.
\end{align}
Due to the periodic restart, the above bound also holds similarly for the iterations $k=tq,tq+1,...,(t+1)q-2$ for any $t\in \mathbb{N}$, which yields that
\begin{align}
\mathbb{E}  \| \nabla f(z_{\zeta})\|^2 \le O\bigg(\frac{(f(x_{tq}) - \mathbb{E}[f(x_{(t+1)q-1})])}{q-1} \bigg), ~\text{where}~ \zeta \overset{\text{Unif}}{\sim}\{tq,...,(t+1)q-2\}. \label{eq: 21}
\end{align}
Next, consider running the algorithm with restart for iterations $k=0,...,K-1$, and the output index $\zeta$ is selected from $\{k: 0\le k\le K-1, \textrm{mod}(k,q-1)\ne0 \}$ uniformly at random. Let $T = \ceil[\Big]{\frac{K}{q-1}}$. Then, we can obtain the following estimate
\begin{align}
\mathbb{E}  \| \nabla f(z_{\zeta})\|^2 &\le \frac{1}{K-T} \sum_{t=0}^{T} \sum_{k=tq}^{(t+1)q-2} \mathbb{E}  \| \nabla f(z_{k})\|^2 \nonumber\\
&\overset{(i)}{\le} O \bigg(\frac{1}{K-T} \sum_{t=0}^{T} \mathbb{E}(f(x_{tq}) - f(x_{(t+1)q-1})) \bigg) \nonumber\\
&\overset{(ii)}{\le} O \bigg(\frac{(f(x_{0}) - f^*)}{K} \bigg), \nonumber
\end{align}
where (i) uses the results inductively derived from Eq.~(\ref{eq: 21}) and (ii) uses the fact that $x_{(t+1)q-1} = x_{(t+1)q}$ due to restart.

Therefore, it follows that $\mathbb{E}  \|\nabla f(z_{\zeta})\| \le \epsilon$ whenever $K \ge O(\frac{(f(x_{0}) - f^*)}{\epsilon^2})$, and the total number of stochastic gradient calls is in the order of $O(n+\sqrt{n}\epsilon^{-2})$ given the parameters as Algorithm \ref{thm: ssqnmer}.
\section{Proof of Algorithm \ref{thm: ssqn-online}} \label{proof_of_ssqn-online}
As for online case when $\textrm{mod}(k,q)=0$, the Algorithm \ref{Spider-SQN-online} samples $\xi_0$ data points to estimate the gradient, and we obtain the following variance bound based on Algorithm \ref{assum5}.
\begin{align}
	\mathbb{E} \|v_k - \nabla f(x_k)\|^2 &= \mathbb{E} \bigg\| \frac{1}{|\xi_1|}\sum_{i = 1}^{|\xi_1|} \nabla \ell_{u_i}(x_k) - \nabla f(x_k) \bigg\|^2  \le \frac{1}{|\xi_1|^2} \sum_{i =1}^{|\xi_1|} \mathbb{E}\left\|   \nabla \ell_{u_i}(x_k) - \nabla f(x_k)\right\|^2 \le \frac{\sigma_1^2}{|\xi_0|}.
\end{align}

Through telescoping \ref{lemma: Zhang} and using the above bound, we obtain the following lemma.
\begin{lemma}\label{aux: 5}
	Under Assumptions \ref{assum1}, \ref{assum2} and \ref{assum5} , the estimation of gradient $v_k$ constructed by Algorithm \ref{Spider-SQN-online} satisfies that for all $k\in \mathbb{N}$,
	\begin{align}
	\mathbb{E} \|v_k - \nabla f(z_k)\|^2 \le \sum_{i=(\tau(k)-1)q}^{k-1}\frac{L^2}{|\xi_i|} \mathbb{E}\|z_{i+1} - z_{i}\|^2 + \frac{\sigma_1^2}{|\xi_0|}. \label{eq: var online}
	\end{align}
\end{lemma}

Then we can begin the proof of Algorithm \ref{thm: ssqn-online} by applying Algorithm \ref{aux: 5} to step (i) at Eq.~(\ref{eq: 1}), and we can get
\begin{align}
	\mathbb{E} &f(x_{k+1}) \nonumber \\
	 &\le \mathbb{E} f(x_k) + \frac{\eta \sigma_{\mathrm{max}}}{2} \mathbb{E}\|\nabla f(x_k)-v_k\|^2  - (\frac{\eta\sigma_{\mathrm{min}}}{2} - \frac{L\eta^2 \sigma_{\mathrm{max}}^2}{2}) \mathbb{E}\|v_k\|^2  \nonumber\\
	&\overset{(i)}{\le}  \mathbb{E} f(x_k) + \frac{\eta \sigma_{\mathrm{max}}}{2} \sum_{i=(n_k-1)q}^{k }\frac{L^2}{|\xi_k|} \mathbb{E}\|x_{i+1} - x_{i}\|^2 + \frac{\eta \sigma_{\mathrm{max}}}{2}\mathbb{E}\| v_{(n_k-1)q}-\nabla f(x_{(n_k-1)q})\| \nonumber\\
&- (\frac{\eta\sigma_{\mathrm{min}}}{2} - \frac{L\eta^2 \sigma_{\mathrm{max}}^2}{2}) \mathbb{E}\|v_k\|^2  \nonumber\\
	&\overset{(ii)}{=} \mathbb{E} f(x_k) + \frac{\eta^3\sigma_{\mathrm{max}}^3}{2} \sum_{i=(n_k-1)q}^{k}\frac{L^2}{|\xi_k|} \mathbb{E}\|v_i\|^2 + \frac{\eta \sigma_{\mathrm{max}}}{2} \frac{\sigma_1^2}{|\xi_0|}-  (\frac{\eta\sigma_{\mathrm{min}}}{2} - \frac{L\eta^2 \sigma_{\mathrm{max}}^2}{2}) \mathbb{E}\|v_k\|^2. \label{eq: 22}
\end{align}
 Then, one can follow the same analysis and obtain:
\begin{align}
	\mathbb{E} \|\nabla f(x_\zeta)\|^2  &\le \frac{2}{\beta^*} \left(1+   \frac{ L^2 \eta^2\sigma_{\mathrm{max}}^2q}{ |\xi_k| }  \right)\frac{\left( f(x_0) -  f^*  \right)}{K} + \left( \frac{\eta\sigma_{\mathrm{max}}}{ \beta^*} + 2 + \frac{ L^2 \eta^3\sigma_{\mathrm{max}}^3q}{  |\xi_k| \beta^*} \right)\frac{\sigma_1^2}{|\xi_0|}. \label{final3}	
\end{align}
To make the right hand side be smaller than $\epsilon^2$, $K\ge \frac{2}{\beta^*} \left(1+ \frac{ L^2 \eta^2\sigma_{\mathrm{max}}^2q}{ |\xi_k| }  \right)\frac{2\left( f(x_0) -  f^*  \right)}{\epsilon^2} $, $|\xi_0|\ge \left( \frac{\eta\sigma_{\mathrm{max}}}{ \beta^*} + 2 + \frac{ L^2 \eta^3\sigma_{\mathrm{max}}^3q}{ |\xi_k|\beta^*} \right)\frac{2\sigma_1^2}{\epsilon^2}$ is necessary. Let
\begin{align}
	  q =|\xi_k| = \sqrt{|\xi_0|}, \eta \equiv \frac{(1+\sqrt{5})\sigma_{\mathrm{min}}}{2L\sigma_{\mathrm{max}}^2} , \label{online_parameter_setting}
\end{align}
where $|\xi_0|$ is set as $|\xi_0|= \left( \frac{\eta\sigma_{\mathrm{max}}}{ \beta^*} + 2 + \frac{ L^2 \eta^3\sigma_{\mathrm{max}}^3}{ \beta^*} \right)\frac{2\sigma_1^2}{\epsilon^2}$. This proves the desired iteration complexity, and the total number of stochastic gradient oracle calls is at most $(K+q)\frac{|\xi_0|}{q} + K|\xi_k|$. With the parameters setting, we obtain the total SFO complexity as $O(\epsilon^{-3})$.

\section{Proof of Algorithm \ref{thm: ssqnm-online}}
Firstly, one can check that Eq.~(\ref{eq: 16}) still holds in the online case. And then, one can apply Algorithm \ref{aux: 5} to Eq.~(\ref{eq: 17}) and follow the proof of Eq.~(\ref{eq: 27}). One can check that there is an additional term $\sum_{k=0}^{K-1} \frac{\lambda_{k}\sigma_1^2}{2|\xi_1|}$ in the online case, and we obtain the following bound.
\begin{align}
\mathbb{E} [f(x_{K})] &\le f(x_{0}) - \sum_{k=0}^{K-1} {\beta^*}\mathbb{E}\| v_{k}\|^2 + \sum_{k=0}^{K-1} \frac{\lambda_{k}\sigma_1^2}{2|\xi_0|} \nonumber\\
&\le f(x_{0}) - \sum_{k=0}^{K-1} {\beta^*}\mathbb{E}\| v_{k}\|^2 + \frac{K\beta\sigma_1^2}{|\xi_0|}. \label{eq: 28}
\end{align}
Then, it follows that $\frac{1}{K}\sum_{k=0}^{K-1}\mathbb{E}\|v_{k}\|^2 \le (f(x_0) - f^*)/(K\beta^*) + \frac{\beta \sigma_1^2}{\beta^*|\xi_0|}$. One can check that Eq.~(\ref{eq: 29}) still holds, and we only need to update the bound for the term $\mathbb{E}  \|\nabla f(z_{\zeta}) - v_{\zeta}\|^2$ as follows
\begin{align}
	\mathbb{E}  \|\nabla f(z_{\zeta}) - v_{\zeta}\|^2 \le 3L^2\beta^2 \frac{16(f(x_0) - f^*)}{K\beta^*} + \frac{\sigma_1^2}{|\xi_0|}. \label{eq: 40}
\end{align}
Then, we finally obtain that
\begin{align}
\mathbb{E}  \| \nabla f(z_{\zeta})\|^2  &\le \frac{6L^2\beta^2+2}{\beta^*} \frac{(f(x_0) - f^*)}{K} +2(1+\frac{\beta}{\beta^*})\frac{\sigma_1^2}{|\xi_0|} . \label{eq: final4}
\end{align}
To make the right hand side be smaller than $\epsilon^2$, we can set $K\ge \frac{2(6L^2\beta^2+2)(f(x_0) - f^*)}{\beta^*\epsilon^2}$, $|\xi_0|\ge \frac{4(1+\beta/\beta^*)\sigma_1^2}{\epsilon^2}$,
and let
\begin{align}
	  q =\xi_k = \sqrt{|\xi_0|}, \beta_k \equiv \frac{\sigma_{\mathrm{min}}}{(3+\sqrt{15})L\sigma_{\mathrm{max}}^2} \label{eq:5} ,
\end{align}
where  $|\xi_0|$ is set as $|\xi_0|= \frac{4(1+\beta/\beta^*)\sigma_1^2}{\epsilon^2}$. The total number of stochastic gradient oracle calls is at most $(K+q)\frac{|\xi_0|}{q} + K|\xi_k|$. By parameters setting as Eq.~(\ref{eq:5}) we obtain the total SFO complexity as $O(\epsilon^{-3})$.

\section{Proof of Algorithm \ref{thm: ssqnmer-online}}
The convergence proof of Algorithm \ref{thm: ssqnmer-online}, including both online SpiderSQNMER and online SpiderSQNMED, follows from that of Theorem~\ref{thm: ssqnmer}. Especially, one just consider the additional variance bounded by $\sigma_1$ and therefore we only describe the key steps to adapt the proof.

We first prove the result of online SpiderSQNMED. Under the epochwise-diminishing momentum scheme, the momentum coefficient is set to be $\alpha_{k} = \frac{2}{\ceil[]{k/q}+1}$. Consequently, we have $\Gamma_{k} = \frac{2}{\ceil[]{k/q}(\ceil[]{k/q}+1)}$. First, one can check that Eq.~(\ref{eq: 15}) still holds, and now we have $\sum_{t=k}^{K-1} \Gamma_{t} \le \frac{2}{\ceil[]{k/q}}$. Then, we follow the steps that bound the accumulation error term $T$ in Eq.~(\ref{eq: 18}). In the derivation of (ii), we now have that $\sum_{i=(\tau(k)-1)q}^{\tau(k)q-1} \alpha_{i+2}^2\Gamma_{i+1} \le \frac{2}{\tau(k)^3}$. Substituting this new bound into (ii) and noting that in (iii) we now have $\frac{\alpha_{k+1}}{\Gamma_{k+1}} = (\ceil[]{k/q}+1)$, one can follow the subsequent steps and show that the upper bound for T in Eq.~(\ref{eq: 19}) still holds. Moreover, in Eq.~(\ref{eq: 31}) we should replace $\frac{L(\lambda_k - \beta_k)^2}{k\Gamma_{k+1}\alpha_{k+1}}$ with $\frac{L(\lambda_k - \beta_k)^2}{\ceil[]{k/q}\Gamma_{k+1}\alpha_{k+1}}$, and consequently Eq.~(\ref{eq: 32}) is still valid.
Then, one can check that Eq.~(\ref{eq: final4}) that is

\begin{align}
\mathbb{E}  \| \nabla f(z_{\zeta})\|^2  &\le \frac{6L^2\beta^2+2}{\beta^*} \frac{(f(x_0) - f^*)}{K} +2(1+\frac{\beta}{\beta^*})\frac{\sigma_1^2}{|\xi_0|},  \label{eqfinal6}
\end{align}
is still valid. To make the right hand side of above equation be smaller than $\epsilon^2$, we can set $K\ge \frac{2(6L^2\beta^2+2)(f(x_0) - f^*)}{\beta^*\epsilon^2}$, $|\xi_0|\ge \frac{4(1+\beta/\beta^*)\sigma_1^2}{\epsilon^2}$,
and let
\begin{align}
	  q =\xi_k = \sqrt{|\xi_0|},  \beta_k \equiv \frac{\sigma_{\mathrm{min}}}{(3+\sqrt{15})L\sigma_{\mathrm{max}}^2} ,
\end{align}
where $|\xi_0|$ is set to $|\xi_0|= \frac{4(1+\beta/\beta^*)\sigma_1^2}{\epsilon^2}$.The total number of stochastic gradient oracle calls is at most $(K+q)\frac{|\xi_0|}{q} + K|\xi_k|$. By setting $q=|\xi_k|=\sqrt{|\xi_0|}$, and we obtain the total SFO complexity as $O(\epsilon^{-3})$.

In summary, given the same parameters as for SpiderSQNM the convergence rate and the corresponding oracle complexity of SpiderSQNMED remain in the same order as SpiderSQNM that is $O(n+\sqrt{n}\epsilon^{-2})$. One can follow the same analysis as Algorithm \ref{thm: ssqnm-online} and. The convergence proof of online SpiderSQNMER follows from that of online SpiderSQNM. The core idea is to apply the result of online SpiderSQNM to each restart period.
Specifically, consider the iterations $k=0,1,...,q-2$. Firstly, we can rewrite Eq.~(\ref{eq: final4}) as
\begin{align}
\mathbb{E}  \| \nabla f(z_{\zeta})\|^2  &\le \frac{6L^2\beta^2+2}{\beta^*} \frac{(f(x_0) - f^*)}{K} +2(1+\frac{\beta}{\beta^*})\frac{\sigma_1^2}{|\xi_0|} \nonumber\\
&=O(\frac{(f(x_0) - f^*)}{K} + \frac{1}{|\xi_0|}) \label{eq: final6}.
\end{align}
As no restart is performed within these iterations, we can apply the result in Eq.~(\ref{eq: proof41}) (note that $f^*$ is the relaxation of $f(x_K)$) obtained from the analysis of Algorithm \ref{thm: ssqnm} and conclude that
\begin{align}
\mathbb{E}  \|\nabla f(z_{\zeta})\|^2 \le O\bigg(\frac{(f(x_0) - \mathbb{E}[f(x_{q-1})])}{q-1} + \frac{1}{|\xi_0|}\bigg), ~\text{where}~ \zeta \overset{\text{Unif}}{\sim}\{0,...,q-2\}.
\end{align}
Due to the periodic restart, the above bound also holds similarly for the iterations $k=tq,tq+1,...,(t+1)q-2$ for any $t\in \mathbb{N}$, which yields that
\begin{align}
\mathbb{E}  \| \nabla f(z_{\zeta})\|^2 \le O\bigg(\frac{(f(x_{tq}) - \mathbb{E}[f(x_{(t+1)q-1})])}{q-1}+ \frac{1}{|\xi_0|} \bigg), ~\text{where}~ \zeta \overset{\text{Unif}}{\sim}\{tq,...,(t+1)q-2\}. \label{eq21}
\end{align}
Next, consider running the algorithm with restart for iterations $k=0,...,K-1$, and the output index $\zeta$ is selected from $\{k: 0\le k\le K-1, \textrm{mod}(k,q-1)\ne0 \}$ uniformly at random. Let $T = \ceil[\Big]{\frac{K}{q-1}}$. Then, we can obtain the following estimate
\begin{align}
\mathbb{E}  \| \nabla f(z_{\zeta})\|^2 &\le \frac{1}{K-T} \sum_{t=0}^{T} \sum_{k=tq}^{(t+1)q-2} \mathbb{E}  \| \nabla f(z_{k})\|^2 \nonumber\\
&\overset{(i)}{\le} O \bigg(\frac{1}{K-T} \sum_{t=0}^{T} (\mathbb{E}(f(x_{tq}) - f(x_{(t+1)q-1})+ \frac{q-1}{|\xi_0|}) )\bigg) \nonumber\\
&\overset{(ii)}{\le} O \bigg(\frac{(f(x_{0}) - f^*)}{K} + \frac{1}{|\xi_0|} \bigg), \nonumber
\end{align}
where (i) uses the results inductively derived from Eq.~(\ref{eq21}) and (ii) uses the fact that $x_{(t+1)q-1} = x_{(t+1)q}$ due to restart.
To make the right hand side be smaller than $\epsilon^2$, we can set $K\ge \frac{2(6L^2\beta^2+2)(f(x_0) - f^*)}{\beta^*\epsilon^2}$, $|\xi_0|\ge \frac{4(1+\beta/\beta^*)\sigma_1^2}{\epsilon^2}$,
and let
\begin{align}
	  q =|\xi_k| = \sqrt{|\xi_0|}, \beta_k \equiv \frac{\sigma_{\mathrm{min}}}{(3+\sqrt{15})L\sigma_{\mathrm{max}}^2} \label{eq5} ,
\end{align}
where  $|\xi_0|$ is set as $|\xi_0|= \frac{4(1+\beta/\beta^*)\sigma_1^2}{\epsilon^2}$. The total number of stochastic gradient oracle calls is at most $(K+q)\frac{|\xi_0|}{q} + K|\xi_k|$. By parameters setting as Eq.~(\ref{eq5}) we obtain the total SFO complexity as $O(\epsilon^{-3})$.

\subsection{Proof of Theorem  for Lower Bound}
When do convergence analyses, we only use the first-order information, as defined in \cite{carmon2017lower}, our method is a first-order method. Therefore, the proof can be a direct extension of \cite{carmon2017lower,fang2018spider}.
Before drilling into the proof of Theorem \ref{theo:lowerbdd}, it is necessary  for us to introduce the hard instance $\tf_M$  {with $M\geq 1$} constructed by \cite{carmon2017lower}.
	\begin{eqnarray}
	\hf_M(\x) -\Psi(1)\Phi(x_1)
 +\sum_{i=2}^M \left[\Psi(-x_{i-1})\Phi(-x_i) - \Psi(x_{i-1})\Phi(x_i) \right],
	\end{eqnarray}
where the component functions are
\begin{eqnarray}
\Psi(x) =\left\{
\begin{array}{ll}
0&x\leq \frac{1}{2}\\
\exp\left(1- \frac{1}{(2x-1)^2}\right)&x> \frac{1}{2}
\end{array}
\right.
\end{eqnarray}
and
\begin{eqnarray}
\Phi(x)= \sqrt{e} \int_{-\infty}^{x} e^{-\frac{t^2}{2}},
\end{eqnarray}
where $x_i$ denote the value of $i$-th coordinate of $\x$, with $i\in  [d]$.
$\hf_M(\x)$ constructed by \cite{carmon2017lower} is a zero-chain function, that is for every $i\in [d]$, $\nabla_i f(\x) =0$ whenever $x_{i-1} = x_i =x_{i+1}$.  Therefore, any deterministic algorithm can just recover ``one'' dimension in each iteration \cite{carmon2017lower}.  Moreover, it satisfies that : If $|x_i|\leq 1$ for any $i \leq M$,
\begin{eqnarray}
\left\| \nabla \hf_M(\x)\right\| \geq 1.
\end{eqnarray}

Then to handle random algorithms, \cite{carmon2017lower} further consider the following extensions:
\begin{eqnarray}
\tf_{M,\B^M}(\x)  = \hf_M\left((\B^M)^\bT \rho(\x)\right)
+\frac{1}{10}\|\x\|^2 =\hf_M\left( \left\langle \mathbf{b}^{(1)}, \rho(\x)  \right\rangle, \ldots, \left\langle\mathbf{b}^{(M)}, \rho(\x) \right\rangle \right)+\frac{1}{10}\|\x\|^2,
\end{eqnarray}
where $\rho(\x) = \frac{\x}{\sqrt{1+ \| \x \|^2/R^2}}$ and $R = 230 \sqrt{M}$, $\B^M$ is chosen uniformly at random from the space of orthogonal matrices $\cO(d,M) =\{\D \in \RR^{d\times M}\vert \D^\top \D = I_M \ \}$.
The function $\tf_{M,\B}(\x)$ satisfies the following{:}
\begin{enumerate}[]
	\item
	\begin{eqnarray}\label{bound}
	\tf_{M,\B^M}(\mathbf{0}) - \inf_{\x} \tf_{M,\B^M}(\x) \leq 12 M.
	\end{eqnarray}
	\item $\tf_{M,\B^M}(\x)$ has  constant $l$ (independent of $M$ and $d$) Lipschitz continuous gradient.
	\item if $d\geq 52 \cdot 230^2 M^2 \log (\frac{2M^2}{p})$, for any algorithm $\cA$ {solving} \ref{eq: P} (finite-sum case)  with $n=1$, and $f(\x) =\tf_{M,\B^M}(\x)$, then with probability $1-p$,
	\begin{eqnarray}\label{bound1}
	\left\|\nabla \tf_{M,\B^M}(\x^k) \right\| \geq \frac{1}{2}, \quad  \text{for every } k\leq M.
	\end{eqnarray}
\end{enumerate}
The above properties found by \cite{carmon2017lower} is very technical.  One can refer to \cite{carmon2017lower} for more details.

\begin{proof}[Proof of Theorem \ref{theo:lowerbdd}]
	Our lower bound theorem proof is as follows. Following the proof in \cite{fang2018spider}, we further take the number of individual function $n$ into account which is slightly different from Theorem 2 in \cite{carmon2017lower}.
	Set
	\begin{eqnarray}
	f_i(\x) = \frac{ln^{1/2}\epsilon^2}{L}\tf_{M,\B_i^M}(\C_i^\bT\x/b)= \frac{ln^{1/2}\epsilon^2}{L} \left(\hf_{M}\left( (\B_i^M)^\bT\rho(\C^\bT_i\x/b)\right) +\frac{1}{10}\left\|\C^\bT_i \x/b\right\|^2\right),
	\end{eqnarray}
	and
	\begin{eqnarray}
	f(\x) = \frac{1}{n}\sum_{i=1}^n f_i(\x).
	\end{eqnarray}
	where $\B^{nM} =  [\B^M_1, \ldots,\B^M_n]$  is chosen uniformly at random from the space of orthogonal matrices $\cO(d,M) =\{\D \in \RR^{(d/n)\times (nM)}\vert \D^\top \D = I_{(nM)}
\ \}$, with each $\B^M_i \in \{\D \in \RR^{(d/n)\times (M)}\vert \D^\top \D = I_{(M)} \ \}$, $i\in [n]$,  $\C =  [\C_1, \ldots,\C_n]$   is an arbitrary  orthogonal matrices $\cO(d,M) =\{\D \in \RR^{d\times d}\vert \D^\top \D = I_{d} \ \}$, with each $\C^M_i \in \{\D \in \RR^{(d)\times (d/n)}\vert \D^\top \D = I_{(d/n)} \ \}$, $i\in [n]$.  $M =  \frac{\Delta L}{12ln^{1/2}\epsilon^2}$,  with {$n\leq \frac{144 \Delta^2 L^2}{l^2 \epsilon^4}$ (to ensure $M\geq 1$)}, $b = \frac{l\epsilon}{L}$, and $R = \sqrt{230M}$. We first verify that $f(\x)$ satisfies Assumption~\ref{assum1}.
	For  Assumption~\ref{assum1},  from \eqref{bound}, we have
	$$ f(\mathbf{0}) -\inf_{\x\in \RR^d} f(\x)\leq \frac{1}{n}\sum_{i=1}^n ( f_i(\mathbf{0}) -\inf_{\x\in \RR^d} f_i(\x)) \leq \frac{ln^{1/2}\epsilon^2}{L} 12M=  \frac{ln^{1/2}\epsilon^2}{L}\frac{12\Delta L}{12ln^{1/2}\epsilon^2} =\Delta\footnote{\text{If $\x^0 \neq \mathbf{0}$, we can simply translate the counter example as $f'(\x) = f(\x - \x^0)$, then $ f'(\x^0) -\inf_{\x\in \RR^d} f'(\x)\leq \Delta$.  }}.$$
	For  Assumption \ref{assum2},  for any {$i$}, using the $\tf_{M,\B_i^M}$ has $l$-Lipschitz continuous gradient, we have
	\begin{eqnarray}
	\left\|   \nabla \tf_{M,\B_i^M}(\C_i^\bT\x/b)  - \nabla \tf_{M,\B_i^M}(\C_i^\bT\y/b)    \right\|^2 \leq  l^2\left\|\C_i^\bT (\x -\y)/b \right\|^2,
	\end{eqnarray}
	Because $\| \nabla f_i(\x)  -\nabla f_i(\y)\|^2 = \left\|\frac{ln^{1/2}\epsilon^2}{Lb} \C_i\left(  \nabla \tf_{M,\B^M_i}(\C_i^\bT\x/b)  - \nabla \tf_{M,\B^M_i}(\C_i^\bT\y/b)\right)\right\|^2$, and using $\C_i^\top\C_i = I_{d/n}$, we have
	\begin{eqnarray}
	\left\| \nabla f_i(\x)  -\nabla f_i(\y) \right\|^2 \leq \left(\frac{ln^{1/2}\epsilon^2}{L}\right)^2 \frac{l^2}{b^4} \left\| \C_i^\bT (\x -\y)  \right\|^2 = n L^2 \left\|\C_i^\bT (\x -\y) \right\|^2,
	\end{eqnarray}
	where we use $b = \frac{l\epsilon}{L}$.
	Summing $i =1, \ldots, n$ and using each $\C_i$ are orthogonal matrix, we have
	\begin{eqnarray}
	\E \| \nabla f_i(\x)  -\nabla f_i(\y) \|^2 \leq L^2\| \x -\y\|^2.
	\end{eqnarray}
	Then with
	$$d\geq  2\max(9n^3M^2,12n^2MR^2)\log \left(\frac{2n^3M^2}{p}\right) + n^2M  \sim \cO\left(\frac{n^2 \Delta^2 L^2}{\epsilon^{4}}\log\left(\frac{n^2 \Delta^2 L^2}{\epsilon^{4}p}\right)\right),$$
	from Lemma 2 of \cite{carmon2017lower} (or Lemma~12 in \cite{fang2018spider}, also refer to Lemma \ref{lowerbound}  in this paper),  with probability at least $1-p$, after $T = \frac{nM}{2}$ iterations (at the end of iteration $T-1$), for all $I^{T-1}_i$ with $i\in [d]$,   if $I^{T-1}_i< M$,  then for any $j_i \in \{I^{T-1}_i+1, \ldots, M \}$,  we have $ \left\langle \mathbf{b}_{i,j_i},\rho(\C_i^\bT\x/b)   \right\rangle \leq \frac{1}{2}$,  where $I^{T-1}_i$  denotes that the algorithm $\cA$ has called individual function $i$ with $I^{T-1}_i$ times ($\sum_{i=1}^n I^{T-1}_i = T$) at the end of iteration $T-1$,  and $\mathbf{b}_{i,j}$ denotes the $j$-th column of $\B^M_{i}$.  However, from \eqref{bound1}, if $ \left\langle\mathbf{b}_{i,j_i},\rho(\C_i^\bT\x/b)   \right\rangle \leq \frac{1}{2}$, we will have   $\|\nabla \tf_{M,\B_i^M}(\C^\bT_i\x/b)\|\geq\frac{1}{2} $.  So $f_i$ can be solved only after $M$ times calling it.
	
	From the above analysis, for any algorithm $\cA$, after running $	T=\frac{nM}{2} = \frac{\Delta Ln^{1/2}}{24l\epsilon^2}$ iterations, at least $\frac{n}{2}$ functions cannot be solved (the worst case is when  $\cA$ exactly solves $\frac{n}{2}$ functions), so
	\begin{eqnarray}
	&&\left\|\nabla f(\x^{nM/2}) \right\|^2 =  \frac{1}{n^2} \left\|\sum_{i  \text{ not solved}} \frac{ln^{1/2}\epsilon^2}{Lb} \C_i  \nabla \tf_{M,\B_i^M}(\C_i^\bT\x^{nM/2}/b) \right\|^2\notag\\
	&& \overset{a}{=} \frac{1}{n^2} \sum_{i \text{ not solved}}\left\| n^{1/2}  \epsilon \nabla \tf_{M,\B_i^M}(\C_i^\bT\x^{nM/2}/b) \right\|^2\overset{\eqref{bound1}}{\geq} \frac{\epsilon^2}{8},
	\end{eqnarray}
	where in $\overset{a}=$, we use $\C_i^\top\C_j=\mathbf{0}_{d/n}$, when $i\neq j$, and $\C_i^\top\C_i= I_{d/n}$.	
\end{proof}

\begin{lemma}\label{lowerbound}
Let $\{\x\}_{0:T}$ with $T = \frac{nM}{2} $ is informed by a certain algorithm in the form \eqref{algor-need}.  Then when $d\geq 2\max(9n^3M^2,12n^3MR^2)\log (\frac{2n^2M^2}{p}) + n^2M$,  with probability $1-p$, at each iteration $0\leq t\leq T$, $\x^t$ can only recover one  coordinate.
\end{lemma}
\begin{proof}
	The proof is essentially same to \cite{carmon2017lower} and \cite{fang2018spider}.   We give a proof here. Before the poof, we give the following definitions:
	  \begin{enumerate}
	  	\item Let $i^t$ denotes that at iteration $t$, the algorithm choses the $i^t$-th individual function.
	  	\item Let $I^t_i$ denotes the total  times that individual function with index $i$ has been called before iteration $k$. We have  $I^0_i=0$ with $i\in [n]$, $i\neq i^t$, and $I^0_{ i^0}=1$. And  for $t\geq 1$,
	  	\begin{eqnarray}
	   I^t_i=\left\{
	   \begin{aligned}
	  I^{t-1}_i +1, & & \quad i = i_t. \\
	  I^{t-1}_i,&  & \quad \text{otherwise}. \\
	   \end{aligned}
	   \right.
	  	\end{eqnarray}
	   \item Let $\y^t_i =   \rho(\C_i^\bT \x^t) = \frac{\C_i^\bT \x^t}{\sqrt{R^2+ \|\C_i^\bT \x^t\|^2}}$ with $i \in [n]$. We have $\y^t_i\in \RR^{d/n}$ and $\|\y^t_i\|\leq R$.
	  \item Set $\bcV^t_i$ be the set that $ \left(\bigcup_{i=1}^n\left\{ \mathbf{b}_{i, 1}, \cdots  \mathbf{b}_{i, \min(M,I^t_i)}\right\}\right) \bigcup \left\{\y^0_i,\y^1_i, \cdots, \y^t_i\right\}$, where  $\mathbf{b}_{i,j}$ denotes the $j$-th column of $\B^M_{i}$.

	  \item Set $\bcU^t_i$ be the set of  $ \left\{  \mathbf{b}_{i, \min(M,I^{t-1}_i+1)}, \cdots,  \mathbf{b}_{i, M }\right\}$  with $i\in [n]$. $\bcU^t = \bigcup_{i=1}^n \bcU^t_i$. And set $\tilde{\bcU}^t_i =\left\{\mathbf{b}_{i, \min(M,1)}, \cdots, \mathbf{b}_{i, \min(M,I^{t-1}_i)}\right\}$. $\tilde{\bcU}^t = \bigcup_{i=1}^n \tilde{\bcU}^t_i$.
	  \item Let $\cP_i^t\in \mathcal{R}^{(d/n) \times (d/n)}$ denote the projection operator to the span of $\uu\in \bcV_i^t$. And let
	  $\cP^{t\bot}_i$ denote its orthogonal complement.	
\end{enumerate}
Because $\cA^t$ performs measurable mapping, the above terms are all measurable on  $\bxi$ and $\B^{nM}$, where $\bxi$ is the random vector in $\cA$.
It is clear that if for all $0\leq t\leq T$ and $i \in [n]$, we have
\begin{eqnarray}\label{lowerend}
\left| \left\langle \u, \y^t_i    \right\rangle \right| < \frac{1}{2},  \quad  \text{for all~ } \u \in \bcU^t_i.
\end{eqnarray}
then  at each iteration,  we can only recover one index, which is our destination. To prove that \eqref{lowerend} holds with probability at least $1-p$,  we consider a more hard event $\bcG^t$ as
\begin{eqnarray}\label{lowerbound G}
\bcG^t  = \left\{     \left|\left\langle \u, \cP^{(t-1) \bot}_i \y^t_i    \right\rangle\right|\leq  a \|  \cP^{(t-1) \bot}_i \y_i^t\| \mid  \u \in \bcU^t ~(\text{not~} \bcU^t_i) , ~ i\in [n]\right\}, \quad t\geq 1,
\end{eqnarray}
with $a = \min \left( \frac{1}{3(T+1)}, \frac{1}{ 2(1+\sqrt{3T})R}   \right)$.
And $G^{\leq t} = \bigcap_{j=0}^t \bcG^j$.

We first show that if $\bcG^{\leq T}$ happens, then \eqref{lowerend} holds for all $0\leq t\leq T$.      For  $0\leq t\leq T$, and  $i \in [n]$, if $\bcU^t_i = \varnothing$, \eqref{lowerend} is right; otherwise for any $\u \in \bcU^t_i$, we have
\begin{eqnarray}\label{lowerbound uy}
&&\left| \left\langle \u, \y^t_i    \right\rangle \right|\notag\\
&\leq&\left| \left\langle\u, \cP^{(t-1)\bot}_i \y^t_i \right\rangle          \right| + \left| \left\langle\u, \cP^{(t-1)}_i \y^t_i \right\rangle          \right|\notag\\
&\leq& a \| \cP_i^{(t-1)\bot} \y^t_i \|+  \left| \left\langle\u, \cP^{t-1}_i \y^t_i \right\rangle          \right|\leq   a R+  R\left\| \cP^{t-1}_i \uu        \right\|,
\end{eqnarray}
where in the last inequality, we use $ \| \cP^{(t-1)\bot}_i \y^t_i \|\leq \|  \y^{(t-1)}_i \|\leq R$.

If $t=0$, we have $\cP_i^{t-1} = \mathbf{0}_{d/n\times d/n}$, then $\left\| \cP_i^{t-1} \uu         \right\| = 0$, so \eqref{lowerend} holds. When $t\geq 1$, suppose at $t-1$,  $\bcG^{\leq t}$ happens then  \eqref{lowerend} holds for all $0$ to $t-1$. Then we need to prove that $\|\cP^{t-1}_i \uu  \|\leq b = \sqrt{3T}a$ with $\uu\in \bcU^t_i$ and $i\in [n]$. Instead,  we prove   a stronger results: $\|\cP^{t-1}_i \uu  \|\leq b = \sqrt{3T}a$ with all   $\uu\in \bcU^t$ and $i\in [n]$. Again,   When $t = 0$, we have $\|\cP^{t-1}_i \uu  \|= 0$, so it is right, when $t\geq 1$,
by  Graham-Schmidt procedure on  $\y^0_{i},  \mathbf{b}_{i_0, \min(I^0_{i^0}, M)}, \cdots,  \y^{t-1}_{i}, \mathbf{b}_{i_{t-1}, \min(I^{t-1}_{i^{t-1}}, M)}$, we have
\begin{eqnarray}\label{lowerbound3}
\left\| \cP^{t-1}_i \uu        \right\|^2 = \sum_{z=0}^{t-1}\left| \left\langle\frac{ \cP^{(z-1)\bot}_i \y^z_i}{\|\cP^{(z-1)\bot}_i \y^z_i \|}, \u     \right\rangle   \right|^2+ \sum_{z=0, ~I^{z}_{i^{z}}\leq M}^{t-1}\left| \left\langle\frac{ \hat{\cP}^{(z-1)\bot}_i \mathbf{b}_{i_{z}, I^z_{i^z}}}{\|\hat{\cP}^{(z-1)\bot}_i \mathbf{b}_{i_{z}, I^z_{i^z}} \|}, \u     \right\rangle   \right|^2,
\end{eqnarray}
where
 $$ \hat{\cP}^{(z-1)}_i  = \cP^{(z-1)}_i + \frac{\left(\cP^{(z-1)\bot}_i \y^z_i\right)\left(\cP^{(z-1)\bot}_i \y^z_i\right)^\bT}{\left\|\cP^{(z-1)\bot}_i \y^z_i\right\|^2}.     $$
Using $\mathbf{b}_{i_{z}, I^z_{i^z}} \bot \uu$ for all $\uu \in \bcU^t$, we have
 \begin{eqnarray}\label{lowerbound2}
&&\left|\left\langle\hat{\cP}^{(z-1)\bot}_i \mathbf{b}_{i_{z}, I_{i^z}^z}, \u \right\rangle\right|\\
&=&\left|0-\left\langle\hat{\cP}^{(z-1)}_i \mathbf{b}_{i_{z}, I_{i^z}^z}, \u \right\rangle\right|\notag\\
&\leq& \left|\left\langle\cP^{(z-1)}_i \mathbf{b}_{i_{z}, I^z_{i^z}}, \u \right\rangle\right|+ \left|  \left\langle \frac{\cP^{(z-1)\bot}_i   \y^z_i}{\|\cP^{(z-1)\bot}_i   \y^z_i \|}, \mathbf{b}_{i_{z}, I^z_{i^z}} \right\rangle \left\langle  \frac{\cP^{(z-1)\bot}_i   \y^z_i}{\|\cP^{(z-1)\bot}_i   \y^z_i \|}, \uu\right\rangle      \right|.\notag
 \end{eqnarray}

For the first term in the right hand of \eqref{lowerbound2}, by induction, we have
\begin{eqnarray}
\left|\left\langle\cP^{(z-1)}_i \mathbf{b}_{i_{z}, I^z_{i^z}}, \u \right\rangle\right| = \left|\left\langle\cP^{(z-1)}_i \mathbf{b}_{i_{z}, I^z_{i^z}},\cP^{(z-1)}_i   \u \right\rangle\right|\leq b^2.
\end{eqnarray}
For the second  term  in the right hand of \eqref{lowerbound2}, by assumption \eqref{lowerbound G}, we have
\begin{eqnarray}
 \left|  \left\langle \frac{\cP^{(z-1)\bot}_i   \y^z_i}{\|\cP^{(z-1)\bot}_i   \y^z_i \|},  \mathbf{b}_{i_{z}, I^z_{i^z}}\right\rangle \left\langle  \frac{\cP^{(z-1)\bot}_i   \y^z_i}{\|\cP^{(z-1)\bot}_i   \y^z_i \|}, \uu\right\rangle      \right|\leq a^2.
\end{eqnarray}

Also,  we have
 \begin{eqnarray}\label{lowerbound4}
&&\left\|\hat{\cP}^{(z-1)\bot}_i \mathbf{b}_{i_{z}, I^z_{i^z}} \right\|^2\\
&=& \| \mathbf{b}_{i_{z}, I^z_{i^z}}\|^2 - \left\|\hat{\cP}^{(z-1)}_i\mathbf{b}_{i_{z}, I^z_{i^z}} \right\|^2\notag\\
&=&  \| \mathbf{b}_{i_{z}, I^z_{i^z}}\|^2 - \left\|\cP^{(z-1)}_i \mathbf{b}_{i_{z}, I^z_{i^z}} \right\|^2 -\left|  \left\langle \frac{\cP^{(z-1)\bot}_i   \y^z_i}{\|\cP^{(z-1)\bot}_i   \y^z_i \|},  \mathbf{b}_{i_{z}, I^z_{i^z}}\right\rangle       \right|^2\notag\\
&\geq&1 - b^2 - a^2.\notag
\end{eqnarray}
 Substituting \eqref{lowerbound2}  and \eqref{lowerbound4} into \eqref{lowerbound3}, for all $\u\in \bcU^t$, we have
 \begin{eqnarray}
\left\| \cP^{t-1}_i \uu        \right\|^2 &\leq& t a^2 + t \frac{ (a^2+b^2)^2}{ 1- (a^2 +b^2)}\notag\\
&\overset{a^2 +b^2 \leq  (3T+1)a^2 \leq a}{\leq}& T a^2+ T \frac{a^2}{ 1-a}\overset{a\leq 1/2}\leq 3T a^2 = b^2.
 \end{eqnarray}

 Thus for \eqref{lowerbound uy},  $t\geq 1$, because $\u \in \bcU^t_i\subseteq \bcU^t $,  we have
 \begin{eqnarray}
\left| \left\langle \u, \y^t_i    \right\rangle \right|\leq (a+b)R \overset{a\leq \frac{1}{ 2(1+\sqrt{3T})R}}{\leq}\leq \frac{1}{2}.
 \end{eqnarray}
 This shows that  if $\bcG^{\leq T}$ happens,  \eqref{lowerend} holds for all $0\leq t\leq T$.  Then we prove that $\PP(\bcG^{\leq T})\geq 1- p$. We have
 \begin{eqnarray}
\PP\left((\bcG^{\leq T})^c\right) &=& \sum_{t=0}^T \PP\left((\bcG^{\leq t})^c\mid \bcG^{< t} \right) .
  \end{eqnarray}
We give the following definition:
\begin{enumerate}
	\item  Denote $\hat{i}^t$ be the sequence of $i_{0:t-1}$.  Let $\hat{\cS}^t$ be the  set that  contains all possible ways of $\hat{i}^t$ ($|\hat{\cS}^t |\leq n^t $).
	\item Let $\tilde{\UU}^j_{\hat{i}^t} = [\mathbf{b}_{j, 1}, \cdots,  \mathbf{b}_{j, \min(M,I^{t-1}_j)}]$ with $j\in [n]$, and $\tilde{\UU}_{\hat{i}^t} =[\tilde{\UU}^1_{\hat{i}^t},\cdots,\tilde{\UU}^n_{\hat{i}^t}]$. $\tilde{\UU}_{\hat{i}^t}$ is  analogous to $\tilde{\bcU^t}$, but is a matrix.
	\item Let $\UU^j_{\hat{i}^t} = [ \mathbf{b}_{j, \min(M,I^t_j)};\cdots;\mathbf{b}_{j, M}] $ with $j\in [n]$, and $\UU_{\hat{i}^t} =[\UU^1_{\hat{i}^t},\cdots,\UU^n_{\hat{i}^t}]$. $\UU_{\hat{i}^t}$ is  analogous to $\bcU^t$, but is a matrix.  Let $\bar{\UU} = [\tilde{\UU}_{\hat{i}^t}, \UU_{\hat{i}^t}  ] $.
	\end{enumerate}
We have that
 \begin{eqnarray}
 &&\PP\left((\bcG^{\leq t})^c\mid \bcG^{< t} \right)\\
&=&\sum_{\hat{i}^t_0\in  \hat{\cS}^t} \E_{\bxi, \UU_{\hat{i}^t_0}} \left(\PP\left((\bcG^{\leq t})^c\mid \bcG^{< t}, \hat{i}^t=\hat{i}^t_0, \bxi,   \UU_{\hat{i}^t_0} \right)\PP\left(\hat{i}^t = \hat{i}^t_0 \mid  \bcG^{< t},   \bxi,   \UU_{\hat{i}^t_0} \right)\right).\notag
 \end{eqnarray}
For  $\sum_{\hat{i}^t_0\in  \hat{\cS}^t} \E_{\bxi, \UU_{\hat{i}^t_0}}\PP\left(\hat{i}^t = \hat{i}^t_0 \mid  \bcG^{< t},   \bxi,   \UU_{\hat{i}^t_0} \right) =  \sum_{\hat{i}^t_0\in  \hat{\cS}^t} \PP\left(\hat{i}^t = \hat{i}^t_0\mid  \bcG^{< t} \right) = 1$,
 in the rest, we show that the probability $\PP\left((\bcG^{\leq t})^c\mid \bcG^{< t}, \hat{i}^t= \hat{i}^t_0, \bxi = \bxi_0,  \tilde{\UU}_{\hat{i}^t_0} = \tilde{\UU}_0, \right)$ for all  $\xi_0 ,\tilde{\UU}_0$ is small. By union bound, we have
\begin{eqnarray}
&&\PP\left((\bcG^{\leq t})^c\mid \bcG^{< t}, \hat{i}^t= \hat{i}^t_0, \bxi = \bxi_0,   \tilde{\UU}_{\hat{i}^t_0} = \tilde{\UU}_0 \right)\\
&\leq& \sum_{i=1}^n\sum_{\uu\in \bcU^t}\PP \left(\left\langle\uu, \cP_i^{(t-1)\bot}\y^t_i \right\rangle\geq a \| \cP_i^{(t-1)\bot}\y^t_i\|\mid \bcG^{< t}, \hat{i}^t= \hat{i}^t_0, \bxi = \bxi_0,   \tilde{\UU}_{\hat{i}^t_0} = \tilde{\UU}_0\right) \notag.
\end{eqnarray}
Note that $\hat{i}^t_0$ is a constant. Because given $\bxi$ and $\tilde{\UU}_{\hat{i}^t_0}$, under $G^{\leq t}$,  both  $\cP_i^{(t-1)}$ and $\y^t_i$ are known.  We prove
 	\begin{eqnarray}
\!\!\!\! \!\!\!\! \!\!\!\! 	
\PP\left( \UU_{\hat{i}^t_0} =  \UU_0\mid \bcG^{< t}, \hat{i}^t= \hat{i}^t_0, \bxi = \bxi_0,   \tilde{\UU}_{\hat{i}^t_0} = \tilde{\UU}_0\right) =  	\PP\left(\UU_{\hat{i}^t_0} =  \Z_i\UU_0\mid \bcG^{< t}, \hat{i}^t=\hat{i}^t_0, \bxi = \bxi_0,   \tilde{\UU}_{\hat{i}^t_0} = \tilde{\UU}_0\right),
 	\end{eqnarray}
 	where $\Z_i\in \RR^{d/n\times d/n}$, $\Z_i^\bT\Z_i = \I_d$, and $\Z_i \u  = \u = \Z^\bT_i\u$ for all $\u\in \bcV^{t-1}_i $. In  this way,   $\frac{\cP_i^{(t-1)\bot}\u}{\|\cP_i^{(t-1)\bot}\u\|}$  has uniformed distribution on the unit space.  To prove it, we have
 \begin{eqnarray}
 	&&\PP\left( \UU_{\hat{i}^t_0} = \UU_0 \mid \bcG^{< t}, \hat{i}^t=\hat{i}^t_0, \bxi = \bxi_0,   \tilde{\UU}_{\hat{i}^t_0} = \tilde{\UU}_0\right)\notag\\
 	 &=&  \frac{\PP(\UU_{\hat{i}^t_0} = \UU_0, \bcG^{< t}, \hat{i}^t=\hat{i}^t_0, \bxi = \bxi_0,   \tilde{\UU}_{\hat{i}^t_0} = \tilde{\UU}_0)}{\PP( \bcG^{< t}, \hat{i}^t=\hat{i}^t_0,  \bxi = \bxi_0,   \tilde{\UU}_{\hat{i}^t_0} = \tilde{\UU}_0)}\notag\\
 	 &=& \frac{\PP( \bcG^{< t}, \hat{i}^t=\hat{i}^t_0 \mid \bxi = \bxi_0, \UU_{\hat{i}^t_0} =\UU_0,  \tilde{\UU}_{\hat{i}^t_0} = \tilde{\UU}_0) p(\bxi = \bxi_0, \UU_{\hat{i}^t_0} =\UU_0,  \tilde{\UU}_{\hat{i}^t_0} = \tilde{\UU}_0)}{\PP( \bcG^{< t}, \hat{i}^t=\hat{i}^t_0, \bxi = \bxi_0,   \tilde{\UU}_{\hat{i}^t_0} = \tilde{\UU})},
 \end{eqnarray}
 And
  \begin{eqnarray}
 	&&\PP\left(\UU_{\hat{i}^t_0} = \Z_i \UU_0 \mid \bcG^{< t}, \hat{i}^t=\hat{i}^t_0, \bxi = \bxi_0,   \tilde{\UU}_{\hat{i}_0} = \tilde{\UU}_0\right)\notag\\
 	&=& \frac{\PP( \bcG^{< t}, \hat{i}^t=\hat{i}^t_0 \mid \bxi = \bxi_0, \UU_{\hat{i}^t_0} =\UU_0,  \tilde{\UU}_{\hat{i}^t_0} = \Z_i \tilde{\UU}_0) p(\bxi = \bxi_0, \UU_{\hat{i}^t_0} =\Z_i \UU_0,  \tilde{\UU}_{\hat{i}^t_0} = \tilde{\UU}_0)}{\PP( \bcG^{< t}, \hat{i}^t=\hat{i}^t_0, \bxi = \bxi_0,   \tilde{\UU}_{\hat{i}^t_0} = \tilde{\UU}_0)}
  \end{eqnarray}
For $\bxi$ and $\bar{\UU}$ are independent. And $p(\bar{\UU}) = p(\Z_i \bar{\UU})$, we have $p(\bxi = \bxi_0, \UU_{\hat{i}^t_0} =\UU_0,  \tilde{\UU}_{\hat{i}^t_0} = \tilde{\UU}_0) =p(\bxi = \bxi_0, \UU_{\hat{i}^t_0} =\Z_i \UU_0,  \tilde{\UU}_{\hat{i}^t_0} = \tilde{\UU}_0) $. Then
  we  prove that if $ \bcG^{< t}$ and $\hat{i}^t=\hat{i}^t_0$ happens under  $\UU_{\hat{i}^t_0} =\UU_0, \bxi = \bxi_0,   \tilde{\UU}_{\hat{i}^t_0} = \tilde{\UU}_0$, if and only if $ \bcG^{< t}$ and $\hat{i}^t=\hat{i}^t_0$ happen under  $\UU_{\hat{i}^t_0} =\Z_i \UU_0, \bxi = \bxi_0,   \tilde{\UU}_{\hat{i}^t_0} = \tilde{\UU}_0$.

  Suppose at iteration $l-1$ with $l\leq t$, we have the result. At iteration $l$, suppose  $ \bcG^{<l}$ and $\hat{i}^l = \hat{i}^l_0$ happen, given    $\UU_{\hat{i}^t_0} =\UU_0, \bxi = \bxi_0,   \tilde{\UU}_{\hat{i}^t_0} = \tilde{\UU}_0$. Let $\x'$ and $(\hat{i}')^j$ are generated by  $\bxi = \bxi_0,  \UU_{\hat{i}^t_0} =\Z_i \UU_0, \tilde{\UU}_{\hat{i}^t_0} = \tilde{\UU}_0$. Because $\bcG^{<l}$ happens, thus at each iteration, we can only recover one index until $l-1$.  Then $(\x')^j=\x^j$ and $(\hat{i}')^j = \hat{i}^j$. with $j\leq l$.  By induction, we only need to prove that $\bcG^{l-1'}$ will happen.   Let $\uu \in \bcU^{l-1}$, and $i\in [n]$, we have
 \begin{eqnarray}
  \left|\left\langle \Z_i\u, \frac{\cP^{(l-2) \bot}_i \y^{l-1}_i    }{\|  \cP^{(l-2) \bot}_i\y^{l-1}_i\|}\right\rangle\right|=  \left|\left\langle \u, \Z_i^\bT\frac{\cP^{(l-2) \bot}_i \y^{l-1}_i    }{\|  \cP^{(l-2) \bot}_i \y_i^{l-1}\|}\right\rangle\right|\overset{a}{=} \left|\left\langle \u, \frac{\cP^{(l-2) \bot}_i \y^{l-1}_i    }{\|  \cP^{(l-2) \bot}_i \y^{l-1}_i\|}\right\rangle\right|,
 \end{eqnarray}
where in $\overset{a}=$, we use $\cP^{(l-2) \bot}_i \y^{l-1}_i$ is in the span of $\bcV^l_i\subseteq \bcV^{t-1}_i $.  This shows that  if $ \bcG^{< t}$ and $\hat{i}^t=\hat{i}^t_0$ happen under  $\UU_{\hat{i}^t_0} =\UU_0, \bxi = \bxi_0,   \tilde{\UU}_{\hat{i}^t_0} = \tilde{\UU}_0$, then $\bcG^{< t}$ and $\hat{i}^t=\hat{i}^t$ happen under  $\UU_{\hat{i}^t_0} =\Z_i \UU_0, \bxi = \bxi_0,   \tilde{\UU}_{\hat{i}^t_0} = \tilde{\UU}_0$. In the same way, we can prove the necessity.
Thus for any $\uu\in \UU^t$, if $\| \cP_i^{(t-1)\bot}\y^t_i\|\neq 0$ (otherwise, $ \left|\left\langle \u, \cP^{(t-1) \bot}_i \y^t_i    \right\rangle\right|\leq  a \|  \cP^{(t-1) \bot}_i \y_i^t\|$ holds), we have
\begin{eqnarray}
&&\PP \left(\left\langle\uu, \frac{\cP_i^{(t-1)\bot}\y^t_i }{ \| \cP_i^{(t-1)\bot}\y^t_i\|}\right\rangle\geq a \mid \bcG^{< t}, \hat{i}^t= \hat{i}^t_0, \bxi = \bxi_0,   \tilde{\UU}_{\hat{i}^t_0} = \tilde{\UU}_0\right)\notag\\
&\overset{a}\leq&\PP \left(\left\langle\frac{\cP_i^{(t-1)\bot}\uu }{ \| \cP_i^{(t-1)\bot}\uu\|}, \frac{\cP_i^{(t-1)\bot}\y^t_i }{ \| \cP_i^{(t-1)\bot}\y^t_i\|}\right\rangle\geq a \mid \bcG^{< t}, \hat{i}^t= \hat{i}^t_0, \bxi = \bxi_0,   \tilde{\UU}_{\hat{i}^t_0} = \tilde{\UU}_0\right)\notag\\
&\overset{b}{\leq}& 2e^{\frac{-a^2 (d/n -2T)}{2}},
\end{eqnarray}
where in $\overset{a}{\leq}$, we use $\| \cP_i^{(t-1)\bot}\uu\|\leq 1$; and in $\overset{b}\leq$, we use $\frac{\cP_i^{(t-1)\bot}\y^t_i }{ \| \cP_i^{(t-1)\bot}\y^t_i\|}$ is a known unit vector and $\frac{\cP_i^{(t-1)\bot}\u}{\|\cP_i^{(t-1)\bot}\u\|}$  has uniformed distribution on the unit space. Then by union bound, we have  $\PP\left(\left(\bcG^{\leq t}\right)^c\mid \bcG^{< t} \right) \leq 2(n^2M)e^{\frac{-a^2 (d/n -2T)}{2}} $.
  Thus
\begin{eqnarray}
\PP\left( \left(\bcG^{\leq T}\right)^c\right) &\leq&  2(T+1)n^2M\exp\left(\frac{-a^2 (d/n -2T)}{2}\right)\notag\\
& \overset{T = \frac{nM}{2}}{\leq} & 2(nM)(n^2M)\exp\left(\frac{-a^2 (d/n -2T)}{2}\right).https://www.overleaf.com/project/5df9fe6b72d0630001ce4cd4
\end{eqnarray}
Then by setting
\begin{eqnarray}
d/n &\geq& 2\max(9n^2M^2,12nMR^2)\log (\frac{2n^3M^2}{p}) + nM \notag\\
 &\geq& 2\max(9 (T+1)^2,2 (2\sqrt{3T})^2R^2)\log (\frac{2n^3M^2}{p}) + 2T\notag\\
 &\geq& 2\max(9 (T+1)^2,2 (1+\sqrt{3T})^2R^2)\log (\frac{2n^3M^2}{p}) + 2T\notag\\
 &\geq& \frac{2}{a^2}\log (\frac{2n^3M^2}{p}) + 2T,
\end{eqnarray}
we have $\PP\left(\left(\bcG^{\leq 	T}\right)^c\right)\leq p$. This completes the proof.
\end{proof}

\subsection{Proof of Assumptions~\ref{assum3} and \ref{assum4}}
Following the proof in \cite{wang2017stochastic} we prove that $H_k$ generated by Algorithm \ref{SdLBFGS} satisfies assumptions \ref{assum3} and \ref{assum4}. For convenience, we restate the formulations have already been stated in our manuscript. First, we prove that $H_k$ generated by SdLBFGS satisfies assumptions \ref{assum3} and then we prove that $H_k$ generated by the two-loop SdLBFGS also satisfies assumptions \ref{assum3}.

At current iteration $k$ (refers to iteration $k$ in Algorithms \ref{SPIDER-SQN}) to \ref{Spider-SQNM-online}, the stochastic gradient difference is defined as
\be\label{y-k}
\bar{y}_{k-1} : = v_{k} - v_{k-1} = \nabla f_{\xi_k} (x_k) - \nabla f_{\xi_k} (x_{k-1}).
\ee
The iterate difference is still defined as $s_{k-1}=x_{k}-x_{k-1}$. We introduce $\hat{y}_{k-1}$ as
\be\label{y-k-ini}
\hat{y}_{k-1} = \theta_{k-1} \bar{y}_{k-1} + (1-\theta_{k-1})H_{{k-1},0}^{-1}s_{k-1},
\ee
where
\be \label{theta-k}
\theta_{k}=\left\{\begin{array}{ll}{\frac{0.75 s_{k-1}^{\top} H_{k, 0}^{-1} s_{k-1}}{s_{k-1}^{\top} H_{k, 0}^{-1} s_{k-1}-s_{k-1}^{\top} \bar{y}_{k-1}},} & {\text { if } s_{k-1}^{\top} \bar{y}_{k-1}<0.25 s_{k-1}^{\top} H_{k, 0}^{-1} s_{k-1}} \\ {1,} & {\text { otherwise, }}\end{array}\right.
\ee
Then we prove that there is $s_{k-1}^{\top} \hat{y}_{k-1} \geq0.25 s_{k-1}^{\top} H_{k, 0}^{-1} s_{k-1}$
\begin{lemma}\label{lemmageq}
Given $\hat{y}_{k-1}$ defined in \eqref{y-k-ini}, there is $s_{k-1}^{\top} \hat{y}_{k-1} \geq0.25 s_{k-1}^{\top} H_{k, 0}^{-1} s_{k-1}$. Moreover, if $H_{k,0}\succ0$, then $H_{k,j}\succ 0$, $j=1, \ldots, m$.
\end{lemma}
\begin{proof}
From \eqref{y-k-ini} and \eqref{theta-k} we have that
\begin{align*}
s_{k-1}^\top \hat{y}_{k-1} = & \, {\theta}_{k}(s_{k-1}^\top \bar{y}_{k-1} - s_{k-1}^\top H_{k,0}^{-1}s_{k-1}) + s_{k-1}^\top H_{k,0}^{-1}s_{k-1}\\
               = & \, \begin{cases}
               0.25s_{k-1}^\top H_{k,0}^{-1}s_{k-1},\quad & \mbox{if }s_{k-1}^\top \bar{y}_{k-1} < 0.25 s_{k-1}^\top H_{k,0}^{-1}s_{k-1},\\
               s_{k-1}^\top \bar{y}_{k-1},\quad & \mbox{otherwise},
               \end{cases}
\end{align*}
which implies $s_{k-1}^\top\hat{y}_{k-1}\geq 0.25s_{k-1}^\top H_{k,0}s_{k-1}$. Therefore, if $H_{k,0}\succ 0$, there is  ${s}_{k-1}^\top\hat{y}_{k-1} >0$. Using $s_j$ and $\hat{y}_j$, $j=k-m, \ldots, k-1$, the formula of SdLBFGS is defined as
\be\label{mod-L-BFGS}
H_{k,i} = (I-\rho_{j} s_{j} \hat{y}_{j}^\top)H_{k,i-1}(I-\rho_{j} \hat{y}_{j} s_{j}^\top) + \rho_{j} s_{j} s_{j}^\top, \quad j=k-(m-i+1); \, i=1,\ldots,m,
\ee
where $\rho_j = ({s}_{j}^\top\hat{y}_{j})^{-1}$. Note that when $k <m$, we use $s_j$ and $\hat{y}_j$, $j=1,\ldots,k$ to perform SdLBFGS updates. As a result, for $H_{k,i}$ defined in \eqref{mod-L-BFGS} and any nonzero vector $z\in\RR^d$, and given $H_{k-1}\succ 0$ we have
\[
z^\top H_{k,i} z = z^\top(I-\rho_{j} s_{j} \hat{y}_{j}^\top) H_{k,i-1} (I -\rho_{j} \hat{y}_{j} s_{j}^\top)z + \rho_{j} (s_{j}^\top z)^2>0, \quad j=k-(m-i+1); \, i=1,\ldots,m,
\]
where $(z^\top(I-\rho_{j} s_{j} \hat{y}_{j}^\top))^\top = (I -\rho_{j} \hat{y}_{j} s_{j}^\top)z$. Through above analysis we have that given $H_{k,0}\succ 0$, $H_{k,j}\succ 0$, $j=1,\ldots,m$. This completes the proof.
\end{proof}

Note that, above proof relies on the assumption that $H_{k,0}\succ 0$ thus we turn to the discussion of choosing $H_{k,0}$.  In this paper we set
\be\label{gama-k}
H_{k,0}=\gamma_k^{-1} I_{d\times d}, \quad \mbox{ where } \gamma_k = \max\left\{
 \frac{\bar{y}_{k-1}^\top \bar{y}_{k-1}}{s_{k-1}^\top \bar{y}_{k-1}} , \delta\right\}\geq \delta.
\ee
Given $\delta>0$ it is obvious that $H_{k,0}\succ 0$.

To prove that $H_k=H_{k,m}$ (in Algorithm \ref{SdLBFGS}, there is $H_kv_k=H_{k,m}v_k=\bar{v}_m$) generated by \eqref{mod-L-BFGS}-\eqref{gama-k} satisfies assumptions \ref{assum3} and \ref{assum4}, we need use Assumption \ref{assum6}.
In the following analysis, we just focus on the finite-sum case, and that of online case is similar. Note that Assumption \ref{assum6} is equivalent to requiring that $-\kappa I\preceq \nabla^2 f_i(x) \preceq\kappa I$ for $i=1,\ldots,n$.
The following lemma shows that the eigenvalues of $H_k$ are bounded below away from zero under Assumption \ref{assum6}.
\begin{lemma}\label{low}
Suppose that Assumption \ref{assum6} holds. Given $H_{k,0}$ defined in \eqref{gama-k}, suppose that $H_k=H_{k,m}$ is updated through the SdLBFGS formula \eqref{mod-L-BFGS}. Then all the eigenvalues of $H_k$ satisfy
\be\label{low-H}
\|H_k\| \ge \left(\frac{4m\kappa^2}{\delta} +(4m+1)(\kappa+\delta)\right)^{-1},
\ee
\end{lemma}
where $\delta$ is a predefined positive constant and $m$ is the memory size.
\begin{proof}
According to Lemma \ref{lemmageq}, $H_{k,i}\succ 0$, $i=1,\ldots,m$. To prove that the eigenvalues of $H_k$ are bounded below away from zero, it suffices to prove that the eigenvalues of $B_k=H_k^{-1}$ are bounded from above. From the formula \eqref{mod-L-BFGS}, $B_k=B_{k,m}$ can be computed recursively as
\[
B_{k,i} = B_{k,i-1} + \frac{\hat{y}_j\hat{y}_j^\top}{s_j^\top \hat{y}_j} - \frac{B_{k,i-1}s_j s_j^\top B_{k,i-1}}{s_j^\top B_{k,i-1}s_j}, \quad j=k-(m-i+1);i=1,\ldots,m,
\]
starting from $B_{k,0}=H_{k,0}^{-1}=\gamma_k I$.
Since $B_{k,0}\succ0$, Lemma \ref{lemmageq} indicates that $B_{k,i}\succ0$ for $i=1,\ldots,m$. Moreover, the following inequalities hold:
{\be\label{tr-B}
\|B_{k,i}\| \le \left\|B_{k,i-1} - \frac{B_{k,i-1}s_j s_j^\top B_{k,i-1}}{s_j^\top B_{k,i-1}s_j}\right\| + \left\|\frac{\hat{y}_j\hat{y}_j^\top}{s_j^\top \hat{y}_j}\right\| \le \|B_{k,i-1}\| + \left\|\frac{\hat{y}_j\hat{y}_j^\top}{s_j^\top \hat{y}_j}\right\| =\|B_{k,i-1}\| + \frac{\hat{y}_j^\top\hat{y}_j}{s_j^\top \hat{y}_j}.
\ee}
From the definition of $\hat{y}_j$ in \eqref{y-k-ini} and the facts that $s_j^\top \hat{y}_j\ge 0.25 s_j^\top B_{j+1,0}s_j$ and $B_{j+1,0}=\gamma_{j+1}I$ from \eqref{gama-k}, we have that for any $j=k-1,\ldots,k-m$
\begin{align}
\frac{\hat{y}_j^\top\hat{y}_j}{s_j^\top \hat{y}_j} \le  \, 4 \frac{\|\theta_j \bar{y}_j + (1-\theta_j)B_{j+1,0}s_j\|^2}{s_j^\top B_{j+1,0}s_j} 
=   \, 4\theta_j^2\frac{\bar{y}_j^\top \bar{y}_j}{\gamma_{j+1}s_j^\top s_j} + 8\theta_j(1-\theta_j)\frac{\bar{y}_j^\top s_j}{s_j^\top s_j} + 4(1-\theta_j)^2 \gamma_{j+1}. \label{proof-lemma-low-1}
\end{align}
Note that from \eqref{y-k} we have
\begin{equation}\label{ybar}
  \bar{y}_j = \frac{1}{|\xi_{j+1}|}\sum_{i\in\xi_{j+1}}(\nabla f_{i} (x_{j+1}) - \nabla f_{i} (x_{j}) = \frac{1}{|\xi_{j+1}|}\left(\sum_{i\in\xi_{j+1}}\overline{\nabla^2 f_i}(x_{j},s_j)\right)s_j,
\end{equation}
where $\overline{\nabla^2 f_i}(x_j,s_j) = \int_0^1\nabla^2f_i(x_j+ts_j)dt$, because $g(x_{j+1})-g(x_j)=\int_0^1\frac{dg}{dt}(x_j+ts_j)dt = \int_0^1\nabla^2 f_i(x_j+ts_j)s_jdt$. Therefore, for any $j=k-1,\ldots,k-m$, from \eqref{proof-lemma-low-1}, and the facts that $0<\theta_j\le1$ and $\delta\le\gamma_{j+1} \le\kappa+\delta$ (according to Eq.~\ref{ybar} and Eq.~\ref{gama-k}, there is $\text{max}\{\delta,\kappa\} \le \gamma_{j+1}$), and the assumption Assumption \ref{assum6} it follows that
\begin{align}
\frac{\hat{y}_j^\top\hat{y}_j}{s_j^\top \hat{y}_j} \le  \, \frac{4\theta_j^2\kappa^2}{\gamma_{j+1}} + 8\theta_j(1-\theta_j)\kappa + 4(1-\theta_j)^2\gamma_{j+1} 
\le \, \frac{4\theta_j^2\kappa^2}{\delta} + 4[(1-\theta_j^2)\kappa + (1-\theta_j)^2\delta] 
\le \, \frac{4\kappa^2}{\delta} + 4(\kappa+\delta). \label{proof-lemma-low-2}
\end{align}
Combining \eqref{tr-B} and \eqref{proof-lemma-low-2} yields
\[
\|B_{k,i}\|\le \|B_{k,i-1}\| + 4\left(\frac{\kappa^2}{\delta}+\kappa+\delta\right).
\]
By induction, we have that
\begin{align*}
\|B_k\|=\|B_{k,m}\| \le \, \|B_{k,0}\| + 4m\left(\frac{\kappa^2}{\delta} + \kappa+\delta\right)
\le \, \frac{4m\kappa^2}{\delta} +(4m+1)(\kappa+\delta),
\end{align*}
which implies \eqref{low-H}.
\end{proof}

We now prove that $H_k$ is uniformly bounded above.
\begin{lemma}\label{upp}
Suppose that Assumption \ref{assum6} holds. Given $H_{k,0}$ defined in \eqref{gama-k}, suppose that $H_k=H_{k,m}$ is updated through formula \eqref{mod-L-BFGS}. Then $H_k$ satisfies
\be\label{up-H}
\|H_k\| \le \left(\frac{\alpha^{2m}-1}{\alpha^{2}-1}\right)\frac{4}{\delta} + \frac{\alpha^{2m}}{\delta},
\ee
where $\alpha = (4\kappa + 5\delta)/\delta$, $\delta$ is a predefined positive constant and $m$ is the memory size.
\end{lemma}
\begin{proof}
For notational simplicity, we omit the subscript, and let $H=H_{k,i-1}$, $H^+=H_{k,i}$, $s=s_j$, $\hat{y}=\hat{y}_j$, $\rho=(s_j^\top\hat{y}_j)^{-1}=(s^\top\hat{y})^{-1}$. Now Eq.~\eqref{proof-lemma-low-2} can be written as
\[H^+ = H -\rho(H\hat{y}s^\top+s\hat{y}^\top H) + \rho ss^\top + \rho^2(\hat{y}^\top H \hat{y})ss^\top.\]
Using the facts that $\|uv^\top\|\leq\|u\|\cdot\|v\|$ for any vectors $u$ and $v$, $\rho s^\top s = \rho\|s\|^2 = \frac{s^\top s}{s^\top\hat{y}}\leq \frac{4}{\delta}$, and $\frac{\|\hat{y}\|^2}{s^\top\hat{y}}\leq 4\left(\frac{\kappa^2}{\delta}+\kappa+\delta\right)<\frac{4}{\delta}(\kappa+\delta)^2$, which follows from \eqref{proof-lemma-low-2},
we have that
\[\|H^+\|\leq \|H\| + \frac{2\|H\|\cdot\|\hat{y}\|\cdot\|s\|}{s^\top\hat{y}}+\frac{s^\top s}{s^\top\hat{y}}+\frac{s^\top s}{s^\top\hat{y}}\cdot\frac{\|H\|\cdot\|\hat{y}\|^2}{s^\top\hat{y}}.\]
Noting that $\frac{\|\hat{y}\|\|s\|}{s^\top\hat{y}}=\left[\frac{\|\hat{y}\|^2}{s^\top\hat{y}}\cdot\frac{\|s\|^2}{s^\top\hat{y}}\right]^{1/2}$,
it follows that
\begin{align*}\|H^+\|\leq \, \left(1+2\cdot\frac{4}{\delta}(\kappa+\delta)+\left(\frac{4}{\delta}(\kappa+\delta)\right)^2\right)\|H\|+\frac{4}{\delta}   =   \, (1+(4\kappa+4\delta)/\delta)^2\|H\|+\frac{4}{\delta}.\end{align*}
Hence, by induction we obtain \eqref{up-H}.
\end{proof}

Lemmas \ref{low} and \ref{upp} indicate that $H_k$ generated by \eqref{y-k-ini}-\eqref{mod-L-BFGS} satisfies Assumption \ref{assum3}. Moreover, since $\bar{y}_{k-1}$ defined in \eqref{y-k} depends on random samplings in the $k$-th iteration i.e., $\xi_k$, it follows that  given $\xi_{k}$ and $v_{k-1}$ $H_{k}$ is determined and Assumption \ref{assum4} is satisfied.

\bibliography{ref}
\end{document}